%% file: main.tex
\documentclass{amsart}
\include{header}

\title{Modular Valenced Temperley-Lieb Algebras}

\newcommand{\bmu}{{\boldsymbol{\mu}}}
\newcommand{\bro}{{\boldsymbol{\rho}}}
\newcommand{\bet}{{\boldsymbol{\eta}}}
\newcommand{\pldigs}[1]{{[#1]_{p,\ell}}}
\newcommand{\smallyng}[2]{{\Yboxdim{#1}\Yvcentermath1\yng(#2)}}

\begin{document}

\input{abstract}
\maketitle
\input{introduction}
\section*{Acknowledgements}
The author would like to thank Louise Sutton, Daniel Tubbenhauer, Paul Wedrich and Jieru Zhu for a preview of their work.
The author is also grateful to Stuart Martin and Daniel Tubbenhauer for their comments on a draft of this paper.

\input{notation}

\input{cellular}

\input{valenced}
\input{valenced-cell-data}
\input{valenced-cell-modules}

\input{endomorphism}
\input{seam}
\input{two-part}
\input{small_tensor}

\input{general-composition-cells}

\input{further}

\bibliographystyle{alpha}
\bibliography{all_cite}

\end{document}

%% file: header.tex
\usepackage[utf8]{inputenc}

\usepackage[pdftex,dvipsnames]{xcolor}
\usepackage{mathpazo}
\usepackage{hyperref}
\hypersetup{%
  colorlinks,               
  linkcolor={red!50!black}, 
  citecolor={blue!50!black},
  urlcolor={blue!80!black}  
}

\usepackage{genyoungtabtikz}
\usepackage{amssymb}
\usepackage{amsmath}
\usepackage{amsfonts}
\usepackage{amsthm}

\usepackage{csquotes}

\usepackage{relsize}

\usepackage{graphicx}
\usepackage{epstopdf}

\usepackage{indentfirst}
\usepackage{lipsum}

\usepackage{marvosym}
\usepackage{booktabs}
\usepackage{lscape}

\usepackage{mathtools}

\usepackage{array}
\usepackage{xargs}

\usepackage{multirow}
\usepackage{multicol}

\usepackage[backgroundcolor=lime,linecolor=lime,bordercolor=white,figcolor=white]{todonotes}
\newcommandx{\todocomment}[2][1=]{\todo[linecolor=Plum,backgroundcolor=Plum!25,bordercolor=Plum,#1]{{\bf TODO (comment):} #2}}
\newcommandx{\todourgent}[2][1=]{\todo[linecolor=red,backgroundcolor=red!25,bordercolor=red,#1]{{\bf TODO (urgent):} #2}}
\newcommandx{\todohidden}[2][1=]{}

\usepackage{enumerate}

\usepackage{iftex}

\usepackage{cleveref}
\usepackage{cancel}
\usepackage{diagbox}

\newcommand{\supp}{\operatorname{supp}}
\newcommand{\ass}{\operatorname{ass}}
\newcommand{\spn}{\operatorname{span}}

\newcommand{\Hom}{\operatorname{Hom}}
\newcommand{\End}{\operatorname{End}}

\newcommand{\id}{{\operatorname{id}}}

\newcommand{\rad}{\operatorname{rad}}

\newcommand{\Char}{\operatorname{char}}
\newcommand{\cat}[1]{\underline{#1}}
\newcommand{\cmod}[1]{\cat{\text{mod}\,#1}}
\newcommand{\Z}{\mathbb{Z}}
\newcommand{\N}{\mathbb{N}}
\newcommand{\C}{\mathbb{C}}

\newcommand{\Q}{\mathbb{Q}}

\newcommand{\mathscr}{\mathcal}

\definecolor{rdylgn60}{HTML}{D73027}
\definecolor{rdylgn61}{HTML}{FC8D59}
\definecolor{rdylgn62}{HTML}{FEE08B}
\definecolor{rdylgn63}{HTML}{D9EF8B}
\definecolor{rdylgn64}{HTML}{91CF60}
\definecolor{rdylgn65}{HTML}{1A9850}
\definecolor{rdylgn66}{HTML}{006837}
\definecolor{ca1}{HTML}{72497F}
\definecolor{ca2}{HTML}{5C7B78}
\definecolor{cb1}{HTML}{0862A0}
\definecolor{cb2}{HTML}{CD9B3E}
\definecolor{cc1}{HTML}{C55824}
\definecolor{cc2}{HTML}{288338}
\definecolor{cd}{HTML}{CF838A}
\definecolor{ce}{HTML}{FFAF3A}
\definecolor{cf}{HTML}{AFBF5A}
\definecolor{cg}{HTML}{7ACF99}
\definecolor{ch}{HTML}{8AB1CF}
\definecolor{ci}{HTML}{3B58CF}
\definecolor{cj}{HTML}{736163}

\definecolor{lightgray}{HTML}{DDDDDD}

\makeatletter
\@ifclassloaded{beamer}{%
}{%
\usepackage[margin=1.5in]{geometry}
\makeatother
\theoremstyle{plain}
\newtheorem{theorem}{Theorem}[section]
\newtheorem{proposition}[theorem]{Proposition}
\newtheorem{definition}[theorem]{Definition}
\newtheorem{example}[theorem]{Example}
\newtheorem{lemma}[theorem]{Lemma}

\newtheorem{question}[theorem]{Question}
\newtheorem{corollary}[theorem]{Corollary}

\theoremstyle{remark}
\newtheorem{remark}{Remark}[section]
}

\numberwithin{equation}{section}

\usepackage{tikz}
\usetikzlibrary{shapes,arrows,calc}

\usetikzlibrary{decorations.pathreplacing,decorations.pathmorphing}

\usepackage{tikz-cd}
\ifLuaTeX%
\usetikzlibrary{graphdrawing}
\usetikzlibrary{graphdrawing.layered}
\fi

\tikzset{%
  spot/.style={color=black, thin, dashed},
  rline/.style={color=green, line width=2pt},
  sline/.style={color=blue, line width=2pt},
  tline/.style={color=red, line width=2pt},
  uline/.style={color=green, line width=2pt},
  line/.style={color=#1, line width=2pt},
  line/.default=blue,
  rdot/.style={color=green, thin, fill},
  sdot/.style={color=blue, thin, fill},
  tdot/.style={color=red, thin, fill},
  udot/.style={color=green, thin, fill},
  dot/.style={color=#1, thin, fill},
  dot/.default=blue
}

\newcommand{\TL}{{\rm TL}}
\newcommand{\JW}{{\rm JW}}

\newcommand{\TLcat}{{\mathscr{TL}}}

\newcommand{\gaussianquant}{\genfrac{[}{]}{0pt}{}}

\newcommand{\ind}{\mathord\uparrow}

\newcommand{\mother}[2][]{{{\mathrm m}^{#1}_{[#2]}}}
\newcommand{\generation}[1]{{{\mathrm g}_{[#1]}}}

\let\emph\relax % there's no \RedeclareTextFontCommand
\DeclareTextFontCommand{\emph}{\bfseries\em}

\usepackage{cite}
\makeatletter
\newcommand{\citecomment}[2][]{\citen{#2}#1\citevar}
\newcommand{\citeone}[1]{\citecomment{#1}}
\newcommand{\citetwo}[2][]{\citecomment[,~#1]{#2}}
\newcommand{\citevar}{\@ifnextchar\bgroup{;~\citeone}{\@ifnextchar[{;~\citetwo}{]}}}
\newcommand{\citefirst}{\@ifnextchar\bgroup{\citeone}{\@ifnextchar[{\citetwo}{]}}}
\newcommand{\cites}{[\citefirst}
\makeatother

\author[R. A. Spencer]{Robert A. Spencer}
\address{Department of Pure Mathematics and Mathematical Statistics\\ University of Cambridge \\Cambridge\\ CB3 0WA \\ UK}
\email{ras228@cam.ac.uk}
\date{}

%% file: abstract.tex
\begin{abstract}
  We investigate the representation theory of the valenced Temperley-Lieb algebras in mixed characteristic.
  These algebras, as described in characteristic zero by Flores and Peltola, arise naturally in statistical physics and conformal field theory and are a natural deformation of normal Temperley-Lieb algebras.
  In general characteristic, they encode the fusion rules for the category of $U_q(\mathfrak{sl}_2)$ tilting modules.

  We use the cellular properties of the Temperley-Lieb algebras to determine those of the valenced Temperley-Lieb algebras.
  Our approach is, at heart, entirely diagrammatic and we calculate cell indices, module dimensions and indecomposable modules for a wide class of valenced Temperley-Lieb algebras.
  We present a general framework for finding bases of cell modules and a formula for their dimensions.
\end{abstract}

%% file: introduction.tex
\section*{Introduction}\label{sec:introduction}
Given a well-understood algebra, there are various ways to derive new algebras related to the original but with subtle differences.
The simplest is taking quotients, but more exotic examples, such as idempotent truncation, crossed products or quantum deformations exist.
In this paper we will be interested in truncating an algebra by an idempotent.

The representation theory of the Temperley-Lieb algebras over positive characteristic has recently received some attention.
These algebras are cellular in the sense of Graham and Lehrer~\cite{graham_lehrer_1996} and this gives us a good handle on their structure and representation theory.
Indeed, we are able to give the complete decomposition matrix for these algebras~\cite{spencer_2020} as well as formulae for the idempotents of the projective cover of the trivial modules~\cite{martin_spencer_2021} (so-called generalised Jones-Wenzl elements).

We may ask what we can say about the truncations of these algebras at certain idempotents.
In this paper we will investigate a particular class of idempotents, formed out of generalised Jones-Wenzl elements and the algebras they induce.
These algebras inherit some of their cellular structure from the parent Temperley-Lieb algebra, but ``loose'' some of the complexity in interesting ways.
We would like to answer a few principal questions about the representation theory of these algebras:
\begin{enumerate}
    \item What are their cell modules?  
    \item What are the cell module dimensions?
    \item Can we construct a basis of the cell module?
    \item How many (isomorphism classes of) irreducible modules does the algebra have?
    \item What are the dimensions of the irreducible modules?
\end{enumerate}
Knowledge of these questions gives us answers to many others.
For example, knowing (1) and (2) gives us the dimension of the algebra and (3) gives a basis.
The question of identifying the irreducible modules for an algebra is classical and of its own importance.
Their dimensions have applications to higher representation theory as the ranks of intersection forms.
If the algebra happens to only have one-dimensional cell modules, then it is commutative.

In this paper, we answer these questions for certain classes of idempotent-truncated Temperley-Lieb algebras over mixed characteristic and thus derive a large tranche of their structure.
The particular truncation studied is of its own interest, but is motivated by both physical systems and Soergel bimodule theory.
We are able to give recursive formula to answer points (1) and (2) above and give an algorithm for constructing the basis of (3).
In particular cases, we can identify the irreducible modules explicitly and give their dimensions, answering (4) and (5).

It should be noted that the positive and mixed characteristic theory is very much more intricate and complicated than the characteristic zero theory, even specialised at a root of unity.
However, our results are strict generalisations of these cases and the concerned reader may, if they wish, consider all algebras as occurring over $\C$ (the mantra to hold in mind is that $p = \infty$).

\vspace{1.5em}\noindent
The Temperley-Lieb algebras arises from the monoidal Temperley-Lieb category, $\TLcat$, as the endomorphism spaces of the objects $\{\underline{n} \;:\; n \in \N_0\}$.
To be exact, these spaces consist of linear Temperley-Lieb diagrams on $n$ points.
This category is equipped with a natural tensor product which acts as addition on the objects and vertical concatenation of the morphisms:
\begin{equation*}
  \vcenter{\hbox{
  \begin{tikzpicture}
    \node at (0.5, 0.4) {$f_1$};
    \draw (0,.0325) rectangle (1,0.75);
    \foreach \i in {0,...,2} {
      \draw[very thick] (-.4, \i/4+0.125) -- (0,\i/4+0.125);
      \draw[very thick] (1.4, \i/4+0.125) -- (1,\i/4+0.125);
    }
  \end{tikzpicture}
  }}
  \otimes
  \vcenter{\hbox{
  \begin{tikzpicture}
    \node at (0.5, -0.7) {$f_2$};
    \draw (0,-.0325) rectangle (1,-1.25);
    \foreach \i in {-5,...,-1} {
      \draw[very thick] (-.4, \i/4+0.125) -- (0,\i/4+0.125);
      \draw[very thick] (1.4, \i/4+0.125) -- (1,\i/4+0.125);
    }
  \end{tikzpicture}
  }}
  =
  \vcenter{\hbox{
  \begin{tikzpicture}
    \node at (0.5, 0.4) {$f_1$};
    \node at (0.5, -0.7) {$f_2$};
    \draw (0,-.0325) rectangle (1,-1.25);
    \draw (0,.0325) rectangle (1,0.75);
    \foreach \i in {-5,...,2} {
      \draw[very thick] (-.4, \i/4+0.125) -- (0,\i/4+0.125);
      \draw[very thick] (1.4, \i/4+0.125) -- (1,\i/4+0.125);
    }
  \end{tikzpicture}
  }}
\end{equation*}
If $f_1$ and $f_2$ are each idempotents within their respective Temperley-Lieb algebras, then their tensor product, $e$ is too.
We can then construct the algebra $e\cdot\TL_n\cdot e$.
It is unital and associative algebra, but not a sub-algebra of $\TL_n$ since the identities differ.
Crucially for what follows, it is cellular as defined by Graham and Lehrer~\cite{graham_lehrer_1996}.

In more generality, $(f_i)_{i=1}^r$ is a family of idempotents with $f_i \in \TL_{n_i}$ with $\sum_i n_i = n$, then so too is $e = \bigotimes f_i$ in $\TL_n$.
Of particular importance will be when each of the $f_i$ are idempotents describing the projective covers of the trivial $\TL_{n_i}$ module --- the various types of Jones-Wenzl idempotents.
In this case, we are particularly interested in the representation theory of $e \cdot \TL_n \cdot e$.
The tuple $(n_1,\ldots, n_r)$ will be called the composition of the idempotent.
It is ``Eve'' if each projective cover is itself trivial.

The question in the Temperley-Lieb category in particular
is motivated in~\cite[1.3(5)]{burrull_libedinsky_sentinelli_2019} by the study of indecomposable objects in Hecke categories for Universal Coxeter systems~\cite{elias_libedinsky_2017}.
The systems considered there are restricted to Eve compositions over
realisations with Cartan matrix elements $\pm 2$.
Burrull, Libedinsky and Sentinelli, by constructing $p$-Jones-Wenzl idempotents, open the question to non-Eve compositions.
By using the extension in \cite{spencer_2020} we can study more general Cartan matrices using $(\ell, p)$-Jones-Wenzl idempotents.
It is this more general characteristic that this paper works.

The question also arises as a natural positive characteristic analogy of mathematics important to physics.
Originally conceived of as operators for statistical mechanics~\cite{temperley_lieb_1971}, the Temperley-Lieb algebras' close links to knot theory bestow strong links to physics.
Recently, the representation theory of the boundary seam algebras have been studied~\cite{langlois_remillard_saint_aubin_2020} after finding application to conformal field theory~\cite{morin_duchesne_rasmussen_david_2015}.
These are both special cases of valenced Temperley-Lieb algebras~\cite{flores_peltola_2018a, flores_peltola_2018b} .
The representation theory of these algebras is well understood over characteristic zero, where we have explicit forms for the Gram matrices, simple modules and semi-simplicity criteria.

However, as is to be expected, the results over characteristic zero do not extend simply to positive characteristic.
In the case of Temperley-Lieb theory, the most characteristic considered is \emph{mixed} characteristic $(\ell, p)$.
Here, even the representation theory of the Temperley-Lieb algebra is rich and fractal-like~\cite{spencer_2020}.

In this paper, we explore the representation theory of valenced Temperley-Lieb algebras over mixed characteristic.
We use knowledge of the cellular structure of the algebras, and explicit calculations of certain bilinear forms to enumerate the simple modules for some specific cases.
Key to our analysis is a particular factoring of morphisms in the cell modules of valenced Temperley-Lieb algebras, and the explicit forms of the ($\ell, p)$-Jones-Wenzl idempotents.

In concurrent work, Sutton, Tubbenhauer, Wedrich and Zhu~\cite{sutton_tubbenhauer_wedrich_zhu_2021} tackled an essentially identical question, building on earlier work in computing the $p$-Jones-Wenzl idempotents in~\cite{tubbenhauer_wedrich_2019}.
Through their working, the representation theory of $U_q(\mathfrak{sl}_2)$ is critical in providing the module structures and Clebsch-Gordan rules.
The idempotents $\otimes f_i$ are expressed as sums of mutually orthogonal idempotents explicitly.

Our work differs in that it builds on~\cite{spencer_2020} and~\cite{martin_spencer_2021} to work solely within the Temperley-Lieb category.
We don't split the idempotents explicitly, but focus on the cellular nature of $\TL_n$ and its valenced cousins and determine the cell data for these algebras.
Further, we consider some additional cases not covered by~\cite{sutton_tubbenhauer_wedrich_zhu_2021}.

\vspace{2em}

The remainder of this paper is arranged as follows.
In \cref{sec:notation} we introduce the notation necessary for describing the representation theory of $\TL_n$.  
The reader is directed to~\cite{ridout_saint_aubin_2014} for an overview of the characteristic zero theory,~\cite{spencer_2020} for the mixed characteristic and~\cite{tubbenhauer_wedrich_2019, martin_spencer_2021} for the theory on general Jones-Wenzl elements.

\Cref{sec:cellular} revises Graham and Lehrer's \cite{graham_lehrer_1996} cellular algebras and makes explicit K\"onig and Xie's observation that Hecke algebras (those of the form $eAe$ for idempotent $e$) are cellular~\cite[Proposition 4.3]{konig_xi_1998}.
We determine precisely how the representation theory of $eAe$ is related to that of $A$ and consider the critical example $\JW_n \cdot \TL_n \cdot \JW_n$ (where $\JW_n$ is the $(\ell, p)$-Jones-Wenzl idempotent of \cite{martin_spencer_2021}.

We then introduce the objects of study explicitly in \cref{sec:valenced} and consider some known examples.
In \cref{sec:valenced_cell_data} we restrict the idempotents to be ``Eve''.  That is, all Jones-Wenzl elements appearing are ``characteristic zero'' as opposed to $(\ell, p)$-Jones-Wenzl elements.
This substantially simplifies the analysis, and we construct diagram bases for the cell modules, and so find their dimension.
We consider the Gram matrix for these modules in \cref{sec:gram}.

Having laid the ground theory for the study of the valenced algebras, we turn to certain compositions.
The simplest of these, considered in \cref{sec:endomorphism}, is the composition $(n)$ in which case we study $\JW_n \cdot \TL_n \cdot \JW_n$ for some $(\ell, p)$-Jones-Wenzl idempotent $\JW_n$.
This is (for positive characteristic) is an object of study in~\cite{tubbenhauer_wedrich_2019} and this section is a recasting of those results into our notation and framework.
This is critical for studying more complicated algebras.

The seam algebras of~\cite{morin_duchesne_rasmussen_david_2015} (with $\beta_2 = 0$ and $\beta_1 = 1$) are studied over characteristic zero in~\cite{langlois_remillard_saint_aubin_2020} and \cref{sec:seam} can be viewed as the natural extension of that paper to mixed characteristic.
We compute the cell indices and cell-module dimensions in the general case as well as the irreducible modules when the composition is Eve.
Finally, we observe that in the special case of a two part seam partition, we are able to determine all the cell data explicitly (even for non-Eve compositions) and relate this to action of inducing modules from $\TL_{n}$ to $\TL_{n+1}$.

\Cref{sec:two-part} considers two-part, Eve, compositions and \cref{sec:small_tensor} considers two-part non-Eve compositions where the second part is sufficiently small.
In both cases, we are able to characterise the cell indices, bases and dimensions of cell modules (and hence the algebra) and which cell modules have non-degenerate bilinear form.

Finally, we collate those parts of the preceding analysis that are common into a general result in \cref{sec:cell_arbitrary} and muse on further directions and questions in \cref{sec:further}.

%% file: notation.tex
\section{Notation for Temperley--Lieb Theory}\label{sec:notation}
We define some notation for dealing with the modular theory of the Temperley-Lieb algebras.
We assume the reader is familiar with the basics of the Temperley-Lieb theory, at least over characteristic zero (see~\cite{ridout_saint_aubin_2014}).
When discussing the Temperley-Lieb category, $\TLcat$ we use underlines to denote the object set $\{\underline{n}\;:\; n\in \N_0\}$.

\subsection{Ring Parameters and Natural Numbers}
Throughout, implicitly, we will be fixing a ring $R$ which has $(\ell, p)$-torsion.
That is to say, $R$ has a distinguished element $\delta = [2]$ such that $[\ell]$ is the first quantum number after $[1]$ to vanish and $1\in R$ has additive order $p$.
Here, the quantum numbers $[n]$ are defined as polynomials in $\delta$ by $[n+1] + [n-1] = [2] [n]$.
Our Temperley-Lieb algebras will always be defined over $R$, or sometimes $\Q(\delta)$.
The ring in question will always be an integral domain.

This fixes a prime $p$ and integer $\ell > 1$.
Either $p$ or $\ell$ may be taken to be $+\infty$.
The case of $p = +\infty$ is the characteristic zero case, and $\ell = +\infty$ gives the semi-simple Temperley-Lieb algebra (even when $p < +\infty$).

\begin{definition}
  We write $p^{(i)} = \ell p^{i-1}$ with the understanding that $p^{(0)} = 1$.
  If $n = \sum_{i=0}^r n_i p^{(i)}$ where $0\le n_0 < \ell$ and $0\le n_i < p$ for $i>0$, we will write $n = \pldigs{n_r, n_{r-1},\ldots, n_0}$.  These are the $(\ell,p)$-digits of $n$.
  We extend this notation so that $\pldigs{n_r, \ldots, n_0} = \sum n_i p^{(i)}$ even when $n_i$ are possibly negative.
\end{definition}

We would like to emphasise that the sequel of the paper applies to the characteristic zero case and the semi-simple case equally.
That is, if the base ring is $\C$ (or similar) and $\delta$ (or $q$ when $\delta = q + q^{-1}$) is an indeterminate, setting $p,\ell = \infty$ gives that the digits of any $n$ are simply $n = \pldigs{n}$.
Similarly, if $q$ is specialised at a root of unity (so $\delta$ satisfies some quantum polynomial), the digits are $n = \pldigs{n_1,n_0}$ where $n = n_1 \ell + n_0$.
In these cases we recover the known theory of the Temperley-Lieb algebras (such as is described in~\cite{ridout_saint_aubin_2014} exactly.

\begin{definition}
  If $n + 1 = \pldigs{n_r, \ldots, n_s,0,\ldots, 0}$ with $n_s \neq 0$, then the \emph{mother} of $n$, denoted $\mother{n}$ is such that $\mother{n} + 1 = \pldigs{n_r, \ldots, n_{s+1},0,\ldots, 0}$ i.e. the mother of $n$ is found by adding one, setting its least significant non-zero digit to zero, and subtracting one.
  If $n+1$ has a single non-zero digit (i.e. $n = n_r p^{(r)}-1$) then we term $n$ \emph{Eve}.  It has no mother.

  The \emph{generation} of $n$ is one less than the number of non-zero digits in $n+1$.
  We write it $\generation{n}$ and note that $n$ is Eve iff $\generation{n}=0$.

  The set $A(n) = \{\mother{n},\mother[2]{n},\ldots\}$ are known as the \emph{ancestors} of $n$.
  Here, $\mother[t]{n} = \mother[t-1]{\mother{n}}$.
  Note that $|A(n)| = \generation{n}$.
  While in the literature, ``support'' is the generally used term for the next definition, really ``cousins'' is a more accurate description.
  The \emph{support} of $n = \pldigs{n_r, \ldots, n_0} - 1$ is
  \begin{equation*}
    \supp(n) = \{\pldigs{n_r, \pm n_{r-1}, \ldots, \pm n_0} - 1\}.
  \end{equation*}
  Clearly $n\in \supp(n)$ and $|\supp(n)| = 2^{\generation{n}}$.
\end{definition}
Usually, for set $S$, we will write $S + x$ to mean $\{s+x \;:\; s \in S\}$.  Hence we could have written
\begin{equation*}
  \supp(n) = \{\pldigs{n_r, \pm n_{r-1}, \ldots, \pm n_0}\} - 1.
\end{equation*}
More generally, for sets $S, T$, the set $S + T$ will mean $\{s + t\;:\; s \in S \text{ and }t \in T\}$.

We will let $L_n(m)$, $\Delta_n(m)$ and $P_n(m)$ be the trivial module, cell module and projective indecomposable module for $\TL_n$ labelled by $m$ respectively.
Recall that $L_n(m)$ is a composition factor of $\Delta_n(m')$ iff $m' \in \supp m$~\cite[Theorem 8.4]{spencer_2020} iff $\Delta_n(m')$ appears in a $\Delta$-filtration of $P_n(m)$ (for $m \in \Lambda_0$)~\cite[Theorems 2.4, 2.7]{xi_2006}.

The representation theory of $\TL_n$ depends on the parity of $n$.  We thus define $\Z_2^n$ to be the set of all integers of the same parity as $n$.

Another important set is
\begin{equation}\label{eq:two_part_e}
  E_{r,s} = \{|r-s|, |r-s| + 2, \ldots, r + s\}
\end{equation}
Note that $t \in E_{r,s} = E_{s,r}$ iff $r \in E_{t,s}$.

\subsection{Boxes and Pictures for Temperley-Lieb Morphisms}
When drawing morphisms in $\TLcat$ we may replace submorphisms by boxes where the meaning is clear.
The boxes will often have the morphism name written in the middle.

Where multiple strands run concurrently (a so-called ``ribbon'') we may denote this with a thicker line annotated by its width.
As such, the following three morphisms are all identical:
\begin{equation}
    \vcenter{\hbox{
      \begin{tikzpicture}[]
        \draw (-.4,-.4) rectangle (.4,.4);
        \node at (0,0) {$\JW_2$};
        \draw[very thick] (-.6,-.2) -- (-.4,-.2);
        \draw[very thick] (-.6,.2) -- (-.4,.2);

        \draw[very thick] (.4,-.2) arc (90:-90:.2) -- (-.6,-.6);
        \draw[very thick] (.4,.2) arc (90:-90:.6) -- (-.6,-1);
      \end{tikzpicture}
    }} =
    \vcenter{\hbox{
      \begin{tikzpicture}[]
        \draw (-.4,-.4) rectangle (.4,.4);
        \node at (0,0) {$\JW_2$};
        \draw[line width=2pt] (-.6,0) -- (-.4,0);
        \draw[line width=2pt] (.4,0) arc (90:-90:.3) -- (-.6,-.6);
        \node at (.8,-.6) {\tiny$2$};
      \end{tikzpicture}
    }} =
    \vcenter{\hbox{
      \begin{tikzpicture}[]
        \draw[very thick] (.4,-.2) arc (90:-90:.2);
        \draw[very thick] (.4,.2) arc (90:-90:.6);
      \end{tikzpicture}
  }} - \frac 1{[2]}
    \vcenter{\hbox{
      \begin{tikzpicture}[]
        \draw[very thick] (.4,.2) arc (90:-90:.2);
        \draw[very thick] (.4,-.6) arc (90:-90:.2);
      \end{tikzpicture}
  }}
\end{equation}

\subsection{Down- and Up-morphisms}
Throughout, a critical role will be played by certain morphisms known as up and down morphisms.
The mnemonic is that down morphisms decrease the number of strands and up morphisms increase them.
For $n + 1 = \pldigs{n_r, \ldots, n_0}$, and $S \subseteq \{0, \ldots, r-1\}$ write $n[S] = n - 2\sum_{i \in S} n_ip^{(i)}$.

\begin{definition}\cite[Definition 2.14]{tubbenhauer_wedrich_2019}
  Suppose $n+1 = [n_j, n_{j-1},\ldots, n_0]_{p, \ell}$.  Then for each $0\le i\le j$ consider $w = [n_j, \ldots, n_{i+1}, -n_i,0,\ldots, 0]_{p, \ell}-1$ and $x = [n_{i-1}, \ldots, n_0]_{p, \ell}$.
  Note that $w + x = n - 2n_ip^{(i)} = n[\{i\}]$.
  Then define diagram ${\rm d}_i : \underline{n} \to \underline{x+w}$ by 
  \begin{equation}
    {\rm d}_i = 
    \vcenter{\hbox{
      \begin{tikzpicture}[scale=0.4]
        \draw (0,-.3) -- (0,3.3);
        \draw[very thick] (0,3) -- (1.3,3);
        \draw[very thick] (0,.8) arc (-90:90:0.7);
        \draw[very thick] (0,0) -- (1.3,0);
        \node at (1.9,3) {\small $w$};
        \node at (1.9,1.5) {\small $n_i p^{(i)}$};
        \node at (1.9,0) {\small $x$};
      \end{tikzpicture}
    }}
  \end{equation}
  Now if $S = \{s_k> \ldots > s_0\}$ is down-admissible for $n$ set
  \begin{equation}\label{eq:chain downs}
    {\rm d}_{n}^{n[S]} = {\rm d}_{s_k} \ldots {\rm d}_{s_1}{\rm d}_{s_0}.
  \end{equation}
  Note that this is a morphism from $\underline n$ to $\underline{n[S]}$.
  Finally, we denote ${\rm u}_S = \iota {\rm d}_S$ where $\iota$ is the natural involution on the Temperley-Lieb category.
\end{definition}
These morphisms are a special case of a more general ``ladder construction''~\cite{elias_2015}.
\begin{definition}\label{def:ladder}
  Select morphisms $g_n : \underline{n} \to \underline{n}$ for each $n \in \N_0$.
  Then for $S$ a  down-admissible set for $n$ where $n + 1 = \pldigs{n_j,\ldots, n_0}$ construct morphisms ${}_g\widetilde{\rm d}_i$ for each $0\le i \le j$ as follows.
  Firstly,
  \begin{equation}
    {}_g\widetilde{\rm d}_j = 
    \vcenter{\hbox{
      \begin{tikzpicture}[scale=0.4]
        \draw[very thick] (0,0) -- (.6,0);
        \draw[very thick] (0.6,-.7) rectangle (1.9,0.7);
        \draw[very thick] (1.9,0) -- (2.5,0);
        \node at (4.5,0) {\small $n_j p^{(j)} - 1$};
        \node at (1.25,0) {\small $g$};
      \end{tikzpicture}
    }}
  \end{equation}
  and then
  \begin{equation}\label{eq:ladder_step}
    {}_g\widetilde{\rm d}_i = 
    \begin{cases}
    \vcenter{\hbox{
      \begin{tikzpicture}[scale=0.4]
        \draw[white] (0,-3) -- (0,1);
        \draw[very thick] (0,0) -- (.6,0);
        \draw[very thick] (0.6,-.8) rectangle (2.9,0.8);
        \draw[very thick] (2.9,0) -- (3.5,0);
        \draw[very thick] (3.5,-2.4) rectangle (4.9,0.8);
        \draw[very thick] (0,-2.0) -- (3.5,-2.0);
        \draw[very thick] (4.9,-0.8) -- (5.5,-0.8);
        \node at (4.2,-0.9) {\small $g$};
        \node at (1.75,-1.5) {\tiny $n_i p^{(i)}$};
        \node at (1.75,0) {\small ${}_g\widetilde{\rm d}_{i+1}$};
      \end{tikzpicture}
  }} & i \not\in S\\
    \vcenter{\hbox{
      \begin{tikzpicture}[scale=0.4]
        \draw[white] (0,-3) -- (0,1);
        \draw[very thick] (0,0) -- (.6,0);
        \draw[very thick] (0.6,-.8) rectangle (2.9,0.8);
        \draw[very thick] (2.9,0.4) -- (4.0,0.4);
        \draw[very thick] (4.0,-.8) rectangle (5.4,0.8);
        \draw[very thick] (0,-2.0) -- (2.9,-2.0);
        \draw[very thick] (2.9, -.4) arc (90:-90:0.8);
        \draw[very thick] (5.4,0) -- (6.0,0);
        \node at (4.7,0) {\small $g$};
        \node at (1.75,-1.5) {\tiny $n_i p^{(i)}$};
        \node at (1.75,0) {\small ${}_g\widetilde{\rm d}_{i+1}$};
      \end{tikzpicture}
  }} & i \in S\\
  \end{cases}.
  \end{equation}
  Then the ladder morphism $\underline{n} \to \underline{n[S]}$ with respect to $g$ is the morphism ${}_g{\rm d}_n^{n[S]}  = {}_g{\widetilde{\rm d}}_0$.
\end{definition}
In almost all examples we will need, the family $\{g_n\}_n$ are ``self absorbing'' in that $g_m\otimes\id_{n-m} \cdot g_{n} = g_n$ for $n \ge m$.  Examples of such families are the identify morphisms, the Jones-Wenzl idempotents and the $(\ell, p)$-Jones-Wenzl idempotents.
\begin{example}
  Suppose $\ell =4$ and $p = 3$.  Let $n = 278$ so that $n + 1 = [2,1,2,0,3]_{3,4}$.
  Then $n$ is third generation and has eight cousins: $\supp n = \{278,272,230,224,206,200,158,152\}$.
  We tabulate the down morphisms from $\underline{n}$ as well as the general form of the ladder morphism with respect to a self absorbing family.

  \begin{center}
    \begin{tabular}{llll}
      \toprule
      $S$ & $n[S]$ & ${\rm d}_n^{n[S]}$& ${}_g{\rm d}_n^{n[S]}$\\
      \midrule
      $\emptyset$ & 278 &
      $\vcenter{\hbox{
      \begin{tikzpicture}[scale=0.4]
        \draw[very thick] (0,0) -- (1.3,0);
        \node at (1.9,0) {\tiny $278$};
      \end{tikzpicture}}}$&
      $\vcenter{\hbox{
      \begin{tikzpicture}[scale=0.4]
        \draw[very thick] (0,0) -- (1.3,0);
        \draw[very thick, fill=white] (.3,-.2) rectangle (1,.2);
      \end{tikzpicture}}}$
      \\
      $\{0\}$     & 272 &
      $\vcenter{\hbox{
      \begin{tikzpicture}[scale=0.4]
        \draw[very thick] (0,0) arc (-90:90:0.7);
        \draw[very thick] (0,1.7) -- (1.3,1.7);
        \node at (1.2,0.7) {\tiny $3$};
        \node at (1.9,1.7) {\tiny $272$};
      \end{tikzpicture}}}$&
      $\vcenter{\hbox{
      \begin{tikzpicture}[scale=0.4]
        \draw[very thick] (0,.5) -- (0.8,.5) arc (-90:90:0.3);
        \draw[very thick] (0.8,1.7) -- (1.3,1.7);
        \draw[very thick] (0,1.35) -- (0.8,1.35);
        \draw[very thick, fill=white] (.3,.8) rectangle (.8,1.9);
      \end{tikzpicture}}}$
      \\
      $\{2\}$     & 230 &
      $\vcenter{\hbox{
      \begin{tikzpicture}[scale=0.4]
        \draw[very thick] (0,2.2) -- (1.3,2.2);
        \draw[very thick] (0,.4) arc (-90:90:0.7);
        \draw[very thick] (0,0) -- (1.3,0);
        \node at (1.9,2.2) {\tiny $227$};
        \node at (1.2,1.1) {\tiny $24$};
        \node at (1.9,0) {\tiny $3$};
      \end{tikzpicture}}}$&
      $\vcenter{\hbox{
      \begin{tikzpicture}[scale=0.4]
        \draw[very thick] (0,.5) -- (0.8,.5) arc (-90:90:0.3);
        \draw[very thick] (0.8,1.7) -- (1.3,1.7);
        \draw[very thick] (0,1.35) -- (0.8,1.35);
        \draw[very thick] (0,0.1) -- (1.3,0.1);
        \draw[very thick, fill=white] (.3,.8) rectangle (.8,1.9);
        \draw[very thick, fill=white] (1.3,-.1) rectangle (1.8,1.9);
        \draw[very thick] (1.8,0.9) -- (2.1,0.9);
      \end{tikzpicture}}}$
      \\
      $\{0,2\}$   & 224 &
      $\vcenter{\hbox{
      \begin{tikzpicture}[scale=0.4]
        \draw[very thick] (0,0) arc (-90:90:0.7);
        \draw[very thick] (0,1.7) -- (1.3,1.7);
        \node at (1.2,0.7) {\tiny $27$};
        \node at (1.9,1.7) {\tiny $224$};
        \end{tikzpicture}}}$&
      $\vcenter{\hbox{
      \begin{tikzpicture}[scale=0.4]
        \draw[very thick] (0,.5) -- (0.8,.5) arc (-90:90:0.3);
        \draw[very thick] (0.8,1.7) -- (1.3,1.7);
        \draw[very thick] (0,1.35) -- (0.8,1.35);
        \draw[very thick, fill=white] (.3,.8) rectangle (.8,1.9);
      \end{tikzpicture}}}$\\
      $\{3\}$ & 206 &
      $\vcenter{\hbox{
      \begin{tikzpicture}[scale=0.4]
        \draw[very thick] (0,2.2) -- (1.3,2.2);
        \draw[very thick] (0,.4) arc (-90:90:0.7);
        \draw[very thick] (0,0) -- (1.3,0);
        \node at (1.9,2.2) {\tiny $179$};
        \node at (1.2,1.1) {\tiny $36$};
        \node at (1.9,0) {\tiny $27$};
      \end{tikzpicture}}}$&
      $\vcenter{\hbox{
      \begin{tikzpicture}[scale=0.4]
        \draw[very thick] (0,.5) -- (0.8,.5) arc (-90:90:0.3);
        \draw[very thick] (0.8,1.7) -- (1.3,1.7);
        \draw[very thick] (0,1.35) -- (0.8,1.35);
        \draw[very thick] (0,0.1) -- (1.3,0.1);
        \draw[very thick, fill=white] (.3,.8) rectangle (.8,1.9);
        \draw[very thick, fill=white] (1.3,-.1) rectangle (1.8,1.9);
        \draw[very thick] (1.8,0.9) -- (2.1,0.9);
      \end{tikzpicture}}}$
      \\
      \bottomrule
    \end{tabular}
    \hspace{3em}
    \begin{tabular}{llll}
      \toprule
      $S$ & $n[S]$ & ${\rm d}_n^{n[S]}$\\
      \midrule
      $\{0,3\}$     & 200 &
      $\vcenter{\hbox{
      \begin{tikzpicture}[scale=0.4]
        \draw[very thick] (0,4) -- (1.3,4);
        \draw[very thick] (0,2.2) arc (-90:90:0.7);
        \draw[very thick] (0,1.8) -- (1.3,1.8);
        \draw[very thick] (0,0) arc (-90:90:0.7);
        \node at (1.9,4) {\tiny $179$};
        \node at (1.2,2.9) {\tiny $36$};
        \node at (1.9,1.8) {\tiny $21$};
        \node at (1.2,0.7) {\tiny $3$};
      \end{tikzpicture}}}$&
      $\vcenter{\hbox{
      \begin{tikzpicture}[scale=0.4]
        \draw[very thick] (0,.5) -- (0.8,.5) arc (-90:90:0.3);
        \draw[very thick] (0.8,1.7) -- (1.3,1.7);
        \draw[very thick] (0,1.35) -- (0.8,1.35);
        \draw[very thick] (0,0.1) -- (1.3,0.1);
        \draw[very thick, fill=white] (.3,.8) rectangle (.8,1.9);
        \draw[very thick, fill=white] (1.3,-.1) rectangle (1.8,1.9);
        \draw[very thick] (1.8,1.0) -- (2.3,1.0);
        \draw[very thick] (0,-.4) -- (1.8,-.4) arc (-90:90:0.3);
      \end{tikzpicture}}}$
      \\
      $\{2,3\}$     & 158 &
      $\vcenter{\hbox{
      \begin{tikzpicture}[scale=0.4]
        \draw[very thick] (0,2.2) -- (1.3,2.2);
        \draw[very thick] (0,.4) arc (-90:90:0.7);
        \draw[very thick] (0,0) -- (1.3,0);
        \node at (1.9,2.2) {\tiny $155$};
        \node at (1.2,1.1) {\tiny $60$};
        \node at (1.9,0) {\tiny $3$};
      \end{tikzpicture}}}$&
      $\vcenter{\hbox{
      \begin{tikzpicture}[scale=0.4]
        \draw[very thick] (0,.5) -- (0.8,.5) arc (-90:90:0.3);
        \draw[very thick] (0.8,1.7) -- (1.3,1.7);
        \draw[very thick] (0,1.35) -- (0.8,1.35);
        \draw[very thick] (0,0.1) -- (1.3,0.1);
        \draw[very thick, fill=white] (.3,.8) rectangle (.8,1.9);
        \draw[very thick, fill=white] (1.3,-.1) rectangle (1.8,1.9);
        \draw[very thick] (1.8,0.9) -- (2.1,0.9);
      \end{tikzpicture}}}$
      \\
      $\{0,2,3\}$   & 152 &
      $\vcenter{\hbox{
      \begin{tikzpicture}[scale=0.4]
        \draw[very thick] (0,0) arc (-90:90:0.7);
        \draw[very thick] (0,1.7) -- (1.3,1.7);
        \node at (1.2,0.7) {\tiny $63$};
        \node at (1.9,1.7) {\tiny $152$};
      \end{tikzpicture}}}$&
      $\vcenter{\hbox{
      \begin{tikzpicture}[scale=0.4]
        \draw[very thick] (0,.5) -- (0.8,.5) arc (-90:90:0.3);
        \draw[very thick] (0.8,1.7) -- (1.3,1.7);
        \draw[very thick] (0,1.35) -- (0.8,1.35);
        \draw[very thick, fill=white] (.3,.8) rectangle (.8,1.9);
      \end{tikzpicture}}}$\\
      \bottomrule
    \end{tabular}
  \end{center}
  Note that the caps of ${\rm d}_n^{n[S]}$ are all ``ancestor centred.''  That is, their centres lie $x$ strands from the top where $x\in A(n) =  \{275, 251, 215\}$.
\end{example}

In particular, if we pick $g_n = \JW_n$ to be the usual Jones-Wenzl idempotents over characteristic zero, we define $\overline{\rm d}_n^m$ and $\overline{\rm u}_n^m$ in $\TL_n$ over $\Q(\delta)$.
These are also indexed by $m \in \supp n$ and in general do not descend to our pointed ring $(R,\delta)$.

For each $m$ in $\supp n$, we define
\begin{equation}
  \overline{\rm L}_n^m =
  \overline{\rm d}_n^m \cdot \JW_m \cdot \iota\overline{\rm d}_n^m =
  \overline{\rm d}_n^m\cdot \overline{\rm u}_n^m.
\end{equation}
The specifics can be found in~\cite{martin_spencer_2021} where the variable $U_n^m$ is used in place of $\overline{\rm L}_n^m$ and $p_n^m$ replaces $\overline{\rm d}_n^m$.

Now, let $\lambda_n^m \in \Q(\delta)$ be defined as $\overline{\rm u}_n^m\cdot\overline{\rm d}_n^m = (\lambda_n^m )^{-1}\JW_m$.
The remarkable fact is that $\sum_{m\in \supp n} \lambda_n^m \overline{\rm L}_n^m$ is a sum of orthogonal idempotents which descends to mixed characteristic $(\ell, p)$ (the constants $\lambda_n^m$ as well as the set $\supp n$ are dependent on $\ell$ and $p$) and gives the idempotent describing the projective cover of the trivial $\TL_n$-module.
This element is termed the \emph{$(\ell, p)$-Jones-Wenzl idempotent}.

\subsection{Boxes for Distinguished Morphisms}
Certain morphisms will be represented by the following diagrams.
\begin{equation}
  \JW_n =
  \begin{cases}
    \vcenter{\hbox{
      \begin{tikzpicture}[]
        \draw[white] (0,0) -- (0,-.3);
        \draw[thick,fill=purple] (-.2,0) rectangle (0.2,1.6);
        \foreach \i in {0,...,6} {
          \draw[very thick] (-.2,.2+ \i/5) -- (-.4,.2 + \i/5);
          \draw[very thick] (.2,.2+ \i/5) -- (.4,.2 + \i/5);
        }
      \end{tikzpicture}
    }}&\text{if $n$ is Eve}\\
    \vcenter{\hbox{
      \begin{tikzpicture}[]
        \draw[thick,fill=purple] (-.2,0) rectangle (0.2,1.6);
        \draw[thick,fill=black] (-.2,1.5) rectangle (0.2,1.6);
        \foreach \i in {0,...,6} {
          \draw[very thick] (-.2,.2+ \i/5) -- (-.4,.2 + \i/5);
          \draw[very thick] (.2,.2+ \i/5) -- (.4,.2 + \i/5);
        }
      \end{tikzpicture}
    }}&\text{otherwise}
  \end{cases}
\end{equation}
The black box in the non-Eve $(\ell, p)$-Jones-Wenzl idempotents indicates the construction from the ``top down'' in \cref{def:ladder}.

\begin{equation}
  {\rm d}_n^m =
    \vcenter{\hbox{
  \begin{tikzpicture}
    \draw[thick] (1.6,0.2) -- (0.95,0) -- (0.95,1.8) -- (1.6,1.6) -- cycle;
    \foreach \i in {1,...,8} {
      \draw[very thick] (0.75,\i/5) -- (0.95,\i/5);
    }
    \foreach \i in {2,...,7} {
      \draw[very thick] (1.6,\i/5) -- (1.8,\i/5);
    }
    \draw [decorate,decoration={brace,amplitude=5pt}] (0.6,0.1) -- (.6,1.7) node [black,midway,xshift=-10pt] {\footnotesize $n$};
    \draw [decorate,decoration={brace,amplitude=5pt}] (1.9,1.5) -- (1.9,0.3) node [black,midway,xshift=10pt] {\footnotesize $m$};
  \end{tikzpicture}
}}\quad\quad\quad\quad\quad\quad
  \overline{\rm d}_n^m =
    \vcenter{\hbox{
  \begin{tikzpicture}
    \draw[thick] (1.6,0.2) -- (0.95,0) -- (0.95,1.8) -- (1.6,1.6) -- cycle;
    \fill[thick] (1.6,1.45) -- (1.6,1.6) -- (0.95,1.8) -- (0.95,1.65) -- cycle;
    \foreach \i in {1,...,8} {
      \draw[very thick] (0.75,\i/5) -- (0.95,\i/5);
    }
    \foreach \i in {2,...,7} {
      \draw[very thick] (1.6,\i/5) -- (1.8,\i/5);
    }
    \draw [decorate,decoration={brace,amplitude=5pt}] (0.6,0.1) -- (.6,1.7) node [black,midway,xshift=-10pt] {\footnotesize $n$};
    \draw [decorate,decoration={brace,amplitude=5pt}] (1.9,1.5) -- (1.9,0.3) node [black,midway,xshift=10pt] {\footnotesize $m$};
  \end{tikzpicture}
}}\quad\quad\quad
\end{equation}
To avoid confusion, we may write the name of the morphism in the box where necessary.

%% file: cellular.tex
\section{Idempotents and Cellular Algebras}\label{sec:cellular}
\subsection{Definitions}
We review the definition of, and basic results from cellular algebra theory.

\begin{definition}\cite[1.1]{graham_lehrer_1996}
  An $k$-algebra $A$ is \emph{cellular} with \emph{cell data} $(\Lambda, M, C, \iota)$ iff
  \begin{enumerate}[(C1)]
    \item The finite set $\Lambda$ is partially ordered and to each $\lambda \in \Lambda$ we have a non-empty finite set of ``$\lambda$-tableaux'' or ``$\lambda$-diagrams'', $M(\lambda)$.
    \item The set $\{C^\lambda_{S,T} \;:\; \lambda \in \Lambda;\; S,T\in M(\lambda)\}\subseteq A$ forms a $k$-basis of $A$.
    \item The map $\iota$ is an $k$-linear anti-automorphism of $A$ with $\iota^2 = \id$ and $\iota C^{\lambda}_{S,T} = C^\lambda_{T,S}$.
    \item For each $\lambda \in \Lambda$, $S,T \in M(\lambda)$, and $a \in A$, there is a map $r_a : M(\lambda)\times M(\lambda) \to k$ such that
      \begin{equation*}
        aC^\lambda_{S,T} = \sum_{U \in M (\lambda)} r_a(S,U)C^\lambda_{U,T} \mod A^{<\lambda}
      \end{equation*}
      where $A^{<\lambda}$ is the linear span of all $C^\mu_{V,W}$ with $\mu < \lambda$.
  \end{enumerate}
\end{definition}
We will write $A^{\le \lambda}$ to denote the linear span of all $C^{\mu}_{V,W}$ with $\mu \le \lambda$ in the natural way.
We will refer to $k$-linear anti-automorphisms that square to the identity as involutions.
\begin{proposition}\cite[1.5]{graham_lehrer_1996}\label{prop:a_lambda_ideal}
  The linear spaces $A^{\le\lambda}$ and $A^{<\lambda}$ are two sided ideals of $A$ fixed by $\iota$.
\end{proposition}
From this, and the definitions, it is clear that $A / A^{\le \lambda}$ is cellular with identical cell data to $A$, except that $\Lambda$ is restricted to elements not less than or equal to $\lambda$.

\begin{proposition}\cite[2.2]{graham_lehrer_1996}\label{prop:a_lambda_decomp}
  As a left module, $A^{\le \lambda}/A^{<\lambda}$ is the direct sum of copies of a module denoted $\Delta(\lambda)$ which has basis $C^\lambda_{S,T}$ as $S$ varies and $T$ is fixed.
\end{proposition}
We may write this last condition as
\begin{equation}
  A^{\le\lambda}/A^{<\lambda} \simeq \bigoplus_{T \in M(\lambda)} \Delta(\lambda)
\end{equation}
where $\Delta(\lambda)$ is simply the $k$-span of $M(\lambda)$ and the action of $a\in A$ on $S \in M(\lambda)$ is given (from (C4)) by
\begin{equation}
  a \cdot S = \sum_{U \in M(\lambda)} r_a(S,U) U.
\end{equation}
These modules, $\Delta(\lambda)$ are known as cell modules or standard modules and are crucial in what follows.

An alternative definition is in terms of cell ideals.
\begin{definition}\cite[Definition 3.2]{konig_xi_1998}
  Let $A$ be an $k$-algebra for $k$ a commutative Noetherian domain equipped with involution $\iota$.
  A two sided ideal $J\subseteq A$ is called a \emph{cell ideal} iff it is fixed by $\iota$ and there is a left ideal $\Delta\subseteq J$ such that $\Delta$ is finitely generated and free over $k$ and $J\simeq \Delta\otimes_k \iota\Delta$ such that the action of $\iota$ sends $x\otimes y \mapsto \iota y \otimes \iota x$.
\end{definition}
The definition of a cell ideal is motivated as follows.
Let $\lambda\in\Lambda$ be minimal so that $A^\lambda = A^{\le \lambda}$ is isomorphic to $\Delta(\lambda)\otimes \iota\Delta(\lambda)$.
Then $A^\lambda$ is certainly fixed by $\iota$ and the involution acts as we expect it to if the isomorphism is given by $S\otimes \iota T \mapsto C^\lambda_{S,T}$.
Hence $A^\lambda$ is a cell ideal.

\begin{definition}\cite[Definition 3.2]{konig_xi_1998}
  The algebra $A$ is called \emph{cellular} iff there is an $k$-module decomposition $A = J_1'\oplus\cdots\oplus J_n'$ with each $J_i'$ fixed by $\iota$ such that if $J_j = \oplus_{\ell = 1}^j J_\ell'$, we have a chain of two sided ideals
  \begin{equation*}
    0 = J_0 \subset J_1 \subset\cdots \subset J_n = A
  \end{equation*}
  such that $J_j / J_{j-1}$ is a cell ideal of $A/J_{i-1}$.
\end{definition}

The arguments succeeding \cref{prop:a_lambda_ideal,prop:a_lambda_decomp} show that the first definition of a cell ideal implies the second definition.  It is not too hard to show the converse.

\subsection{Cell modules}
Let us return to the important modules $\Delta(\lambda)$.
We know from (C4) that (modulo $A^{<\lambda}$),
\begin{equation}
  C^\lambda_{V,W}aC^\lambda_{S,T} = \sum_{U \in M(\lambda)} r_{ C^\lambda_{V,W}a}(S,U) C^\lambda_{U,T}.
\end{equation}
Similarly
\begin{equation}
  C^\lambda_{T,S}\iota aC^\lambda_{W,V} = \sum_{U \in M(\lambda)} r_{ C^\lambda_{T,S}\iota a}(W,U) C^\lambda_{U,V}.
\end{equation}
But these equations are simply involutions of each-other so in particular, $U=V$ is the only term to survive in the first equation and $U=T$ the only in the second.
We may rewrite this as asserting that there is a map $\phi_a : M(\lambda)\times M(\lambda) \to k$ such that
\begin{equation}\label{eq:cac_multiply}
  C^\lambda_{V,W}aC^\lambda_{S,T} = \phi_a(W,S) C^\lambda_{V,T} \mod A^{<\lambda}.
\end{equation}

A crucial r\^ole is then played by the bilinear form $\langle -,- \rangle_\lambda:\Delta(\lambda)\times\Delta(\lambda) \to k$ defined to be the linear extension of $\phi_\id$.
When $\lambda$ is understood, it will be dropped from the notation.

Let $R(\lambda)\subseteq \Delta(\lambda)$ be the radical of the bilinear form on $\Delta(\lambda)$ and $L(\lambda) = \Delta(\lambda) / R(\lambda)$.
We will denote by $\Lambda_0\subseteq\Lambda$ the set of all $\lambda$ such that $L(\lambda) \neq 0$.

\begin{lemma}\cite[3.4]{graham_lehrer_1996}
  If $k$ is a field, then the set $\{L(\lambda) \;:\; \lambda\in \Lambda_0\}$ is a complete set of pairwise non-isomorphic completely irreducible modules for $A$.
\end{lemma}

Fix $T \in M(\lambda)$.
Then $\spn \{C^\lambda_{S,T} \;:\; S\in M(\lambda)\} = \Delta(\lambda) \otimes \iota T$ is a left ideal of $A^{\le \lambda}/A^{<\lambda}$.
It lifts to the left ideal generated by $\{C^\lambda_{S,T} \;:\; S\in M(\lambda)\}$ of $A$, which we denote $D(\lambda)$.
Note that it is not always true that $D(\lambda) \supseteq A^{<\lambda}$.

\begin{lemma}\label{lem:linear_comm}
  The algebra $A$ is commutative if $\dim \Delta(\lambda) = 1$ for all $\lambda \in \Lambda$.
\end{lemma}
\begin{proof}
  The algebra $A$ has an involution-fixed $k$-basis since $C^\lambda_{x,x} = \iota C^\lambda_{x,x}$ and so $\iota = \id$.
  But then $ab = \iota(ab) = (\iota b)(\iota a) = b a$.
\end{proof}

\begin{example}\label{eg:cellular_ks3}
  If $k$ is any field and $\mathfrak{S}_3$ is generated by the transpositions $\sigma_1 = (12)$ and $\sigma_2 = (23)$, then $k\mathfrak{S}_3$ is cellular with involution given by the linear extension of the inverse operator on $\mathfrak{S}_3$.  The set $\Lambda$ is the set of three partitions:
  \begin{equation}
    \Yboxdim{5pt}
    \Yvcentermath1
    \Lambda =
    \left\{
      \yng(3)\;>\;
      \yng(2,1)\;>\;
      \yng(1^3)
    \right\}
  \end{equation}
  and the sets $M(\lambda)$ consist of standard tableaux:
  \begin{equation}
    \Yboxdim{5pt}
    \Yvcentermath1
    M\left(\yng(3)\right)
    \Yboxdim{9pt}
    = \left\{
      \young(123)
    \right\}\quad\quad
    \Yboxdim{5pt}
    M\left(\yng(2,1)\right)
    \Yboxdim{9pt}
    = \left\{
      \young(12,3),
      \young(13,2)
    \right\}\quad\quad
    \Yboxdim{5pt}
    M\left(\yng(1^3)\right)
    \Yboxdim{9pt}
    = \left\{
      \young(1,2,3)
    \right\}
  \end{equation}
  We will denote these by $\mathbf{123}$, $\mathbf{12}$, $\mathbf{13}$ and $\mathbf{1}$ respectively representing their first rows.

  A possible cellular basis (there are multiple) is given by
  \begin{align*}
    \Yboxdim{3pt}
    \Yvcentermath1
    C^{\yng(1^3)}_{\mathbf{1},\mathbf{1}} &=
    \sigma_1\sigma_2\sigma_1 +
    \sigma_1\sigma_2 +
    \sigma_2\sigma_1 +
    \sigma_1 + \sigma_2 + \id &
    \Yboxdim{3pt}
    \Yvcentermath1
    C^{\yng(3)}_{\mathbf{123},\mathbf{123}} &= \id
    \\
    \Yboxdim{3pt}
    \Yvcentermath1
    C^{\yng(2,1)}_{\mathbf{12},\mathbf{12}} &=
    \sigma_2 + \id &
    \Yboxdim{3pt}
    \Yvcentermath1
    C^{\yng(2,1)}_{\mathbf{12},\mathbf{13}} &=
    \sigma_2\sigma_1 + \sigma_2 + \sigma_1 + \id\\
    \Yboxdim{3pt}
    \Yvcentermath1
    C^{\yng(2,1)}_{\mathbf{13},\mathbf{13}} &=
    \sigma_1 + \id &
    \Yboxdim{3pt}
    \Yvcentermath1
    C^{\yng(2,1)}_{\mathbf{13},\mathbf{12}} &=
    \sigma_1\sigma_2 + \sigma_1 + \sigma_2 + \id,
  \end{align*}
  and this satisfies the involutive property.

  From this we can calculate $\Yboxdim{3pt}\Yvcentermath1
  C^{\yng(1^3)}_{\mathbf{1},\mathbf{1}}\cdot C^{\yng(1^3)}_{\mathbf{1},\mathbf{1}} =
  6 C^{\yng(1^3)}_{\mathbf{1},\mathbf{1}}$
  from which $\langle \mathbf{1},\mathbf{1}\rangle = 6$ and similarly $\Yboxdim{3pt}\Yvcentermath1
  C^{\yng(2,1)}_{\mathbf{12},\mathbf{12}}\cdot C^{\yng(2,1)}_{\mathbf{13},\mathbf{12}} =
  C^{\yng(1^3)}_{\mathbf{1},\mathbf{1}} + C^{\yng(2,1)}_{\mathbf{12},\mathbf{12}}$ so $\langle\mathbf{12},\mathbf{13}\rangle =1$.
  Using these calculations we find that in the given basis,
  \begin{equation}\Yboxdim{3pt}
    \langle-,-\rangle_{\yng(1^3)} = \begin{pmatrix}6\end{pmatrix}\quad\quad\quad
    \langle-,-\rangle_{\yng(2,1)} = \begin{pmatrix}2&1\\1&2\end{pmatrix}\quad\quad\quad
    \langle-,-\rangle_{\yng(3)} = \begin{pmatrix}1\end{pmatrix}.
  \end{equation}
  Thus if $\Char k\not\in\{2,3\}$ then $\Lambda_0 = \Lambda$ and otherwise $\Yboxdim{3pt}\Yvcentermath1\Lambda_0 = \Lambda\setminus\left\{\yng(1^3)\right\}$.

  In all cases $\Delta(\smallyng{5pt}{3})$ is irreducible as $R(\smallyng{5pt}{3}) = 0$.  It is the linear, sign representation.
  Similarly $\Delta(\smallyng{5pt}{1^3})$ is the trivial representation.

  If $\Char k = 2$ then $\Lambda_0 = \{\smallyng{5pt}{2,1},\smallyng{5pt}{3}\}$ and $\Delta(\smallyng{5pt}{2,1})$ is the two dimensional irreducible.  Notice then that the sign representation $\Delta(\smallyng{5pt}{3})$ is actually the trivial and is isomorphic to $\Delta(\smallyng{5pt}{1^3})$.

  On the other hand if $\Char k = 3$ then $\Lambda_0$ is also $ \{\smallyng{5pt}{2,1},\smallyng{5pt}{3}\}$ but now the trivial module arises as the indecomposable head of $\Delta(\smallyng{5pt}{2,1})$.  There is in fact a short exact sequence
  \begin{equation}
    0\to\Delta(\smallyng{5pt}{3}) \to \Delta(\smallyng{5pt}{2,1}) \to \Delta(\smallyng{5pt}{1^3})\to 0.
  \end{equation}
\end{example}

\subsection{The Cellular Data of a Hecke Algebra}\label{sec:cell_hecke}
Let $A$ be a cellular algebra with cell data $(\Lambda, M, C, \iota)$ and let $e \in A$ be idempotent fixed by $\iota$.
We will show that $eAe = \widetilde A$ is cellular and construct cellular data $(\widetilde \Lambda, \widetilde M, \widetilde C, \iota)$.
We will have $\widetilde\Lambda \subseteq \Lambda$ and $\widetilde M(\lambda) \subseteq M(\lambda)$.
If $\widetilde M(\lambda) = \emptyset$ we will exclude $\lambda$ from $\Lambda$, thus satisfying (C1).

As a first step, let $\widetilde C^\lambda_{S,T} = eC^\lambda_{S,T}e$.
Then, $\iota$ is an involution of $eAe$ such that $\iota \widetilde C^\lambda_{S,T} = \widetilde C^\lambda_{T,S}$.
Hence (C3) holds.
It is clear that as $C^\lambda_{S,T}$ span $A$, the $\widetilde C^\lambda_{S,T}$ span $\widetilde A$.
Further, $\widetilde A^{\le \lambda} \subseteq \widetilde A \cap A ^{\le\lambda}$ and if $x \in \widetilde A \cap A^{\le\lambda}$ is written in terms of $C^\mu_{S,T}$, then since it is fixed by left and right multiplication by $e$, we see the reverse inclusion holds too.

Now note that, for any $eae \in \widetilde A$,
\begin{align}
  eae\, \widetilde C^\lambda_{S,T} &= \sum_{U\in M(\lambda)} e\,r_{ae}(U,S)\,C^\lambda_{U,T}e \mod A^{<\lambda}\nonumber\\
                             &= \sum_{U\in M(\lambda)} r_{ae}(U,S)\widetilde C^\lambda_{U,T}\mod A^{<\lambda}
\end{align}
and so (C4) is satisfied too, if we set $\widetilde r_{eae}(-,-)$ to be $r_{ae}(-,-)$.

The final constraint to show is that the $\widetilde C^{\lambda}_{S,T}$ form an $k$-basis of $\widetilde A$.
Here is where we remove elements of $M(\lambda)$ to form $\widetilde M(\lambda)$ and, should the result be empty, $\lambda$ from $\Lambda$.
Clearly by induction it will suffice to show that $\widetilde C^\lambda_{S,T}$ is a basis of $\widetilde A^{\le \lambda}$ for minimal $\lambda\in\widetilde\Lambda$.
On the other hand, since $\widetilde A^{\le\lambda} = \widetilde A \cap A^{\le\lambda}$, we will require that $\widetilde C^\lambda_{S,T}$ is a basis of $\widetilde A^{\le \lambda}$ for minimal $\lambda$.
Thus this is a necessary and sufficient condition.

Equivalently, we require that $\{e\cdot S \;:\; S \in \widetilde M(\lambda)\}$ is a linearly independent set in $\Delta(\lambda)$ for each $\lambda$ (it is already spanning).
In fact, the preceding remark makes clear that as we can pick any subset of $M(\lambda)$ for $\widetilde M(\lambda)$ as long as such a set is a basis.

Let $N^\Delta_e(\lambda) = \{ S \in M(\lambda) \;:\; e \cdot S = 0\}$ be those tableaux killed by $e$ in $\Delta(\lambda)$ and $N^D_e(\lambda) = \{ S\in M(\lambda) \;:\; e \cdot C^\lambda_{S,T} = 0\}$ in $D(\lambda)$.
Clearly $\widetilde M(\lambda) \subseteq M(\lambda) \setminus N_e(\lambda)$ and $N^{D}_e(\lambda) \subseteq N^{\Delta}_e(\lambda)$.
\begin{definition}\label{def:generous}
  We will say that the idempotent $e\in A$ is \emph{generous} if $M(\lambda) \setminus N^\Delta_e(\lambda)$ forms a basis for $\Delta(\lambda)$ for each $\lambda \in \Lambda$.
  It will be termed \emph{lavish} if $N^D_e(\lambda) = N^\Delta_e(\lambda)$ for each $\lambda$.
\end{definition}
Notice that even a generous idempotent does not necessarily have $\widetilde \Lambda = \Lambda$.

The above analysis makes it clear that $\widetilde\Delta(\lambda) = e\Delta(\lambda)$.
Further, by making the substitution $a \to e$ and multiplying \cref{eq:cac_multiply} by $e$ on both sides, we see that the bilinear form on $\widetilde\Delta(\lambda)$ is exactly the restriction of the bilinear form on $\Delta(\lambda)$.

Now
\begin{align}\nonumber
  ey \in \rad \widetilde \Delta(\lambda) &\iff \langle ey, ez \rangle = 0 \quad\forall ez \in \widetilde \Delta(\lambda)\\\nonumber
  &\iff \langle ey, z \rangle = 0 \quad\forall z \in \Delta(\lambda)\\
  &\iff ey \in \rad \Delta(\lambda)
\end{align}
Hence $\rad \widetilde\Delta(\lambda) = e\Delta(\lambda)\cap \rad \Delta(\lambda)$.  However, since $e$ acts as identity on $\rad \widetilde\Delta(\lambda)$ it must be that
\begin{equation}\label{eq:e_rad_is_rad}
  \rad \widetilde\Delta(\lambda) = e \rad \Delta(\lambda).
\end{equation}
Recall that the functor ${\rm res}_e : \cmod{A} \to \cmod{\widetilde A}$ sending $M \mapsto eM$ and restricting morphisms is exact.
Hence we have the exact sequence
\begin{equation}\label{eq:exact_restrict_delta}
  0 \to
  \rad \widetilde\Delta(\lambda) = e\rad \Delta(\lambda) \to
  \widetilde\Delta(\lambda) \to
  \widetilde L(\lambda) = eL(\lambda) \to
  0
\end{equation}
giving that $\{e L(\lambda) : \lambda \in \widetilde \Lambda_0\}$ is a complete set of irreducible modules for $\widetilde A$ (the only thing to prove here was that $e L(\lambda) = \widetilde L(\lambda)$).

\begin{corollary}\label{cor:semisimple}
  If $A$ is semi-simple, then $\tilde A$ is semi-simple.
\end{corollary}
\begin{proof}
  The algebra $A$ (resp. $\widetilde A$) is semi-simple iff $\Delta(m)$ (resp. $\widetilde\Delta(m)$) is simple for all $m$.
  That is to say, $\rad\Delta(m)$ (resp. $\rad\widetilde\Delta(m)$) is zero.
\end{proof}

\begin{example}
  Recall \cref{eg:cellular_ks3} and suppose $k$ has characteristic 3.
  Let $e$ be the idempotent $-\sigma_1 - \id$.
  Then $e\cdot \mathbf{1} = \mathbf{1}$ as $e\cdot C^{\smallyng{3pt}{1^3}}_{\mathbf{1},\mathbf{1}} = C^{\smallyng{3pt}{1^3}}_{\mathbf{1},\mathbf{1}}$.
  On the other hand, in the module $\Delta(\smallyng{5pt}{3})$, we have $e\cdot \mathbf{123} = 0$ as $e\cdot C^{\smallyng{3pt}{3}}_{\mathbf{123},\mathbf{123}} = e = -C^{\smallyng{3pt}{2,1}}_{\mathbf{13},\mathbf{13}}$.

  Similar calculations show that
  \begin{align*}
    e \cdot C^{\smallyng{3pt}{2,1}}_{\mathbf{12},\mathbf{12}} &= -C^{\smallyng{3pt}{2,1}}_{\mathbf{13},\mathbf{12}}&
    e \cdot C^{\smallyng{3pt}{2,1}}_{\mathbf{13},\mathbf{13}} &= C^{\smallyng{3pt}{2,1}}_{\mathbf{13},\mathbf{13}}\\
    e \cdot C^{\smallyng{3pt}{2,1}}_{\mathbf{12},\mathbf{13}} &=
    -C^{\smallyng{3pt}{1^3}}_{\mathbf{1},\mathbf{1}}  +
    -C^{\smallyng{3pt}{2,1}}_{\mathbf{13},\mathbf{13}}&
    e \cdot C^{\smallyng{3pt}{2,1}}_{\mathbf{13},\mathbf{12}} &= C^{\smallyng{3pt}{2,1}}_{\mathbf{13},\mathbf{12}}
  \end{align*}
  from which we deduce that $e\cdot\mathbf{13} = \mathbf{13}$ and $e\cdot\mathbf{12} = -\mathbf{13}$ in $\Delta(\smallyng{5pt}{2,1})$.
  Its clear then that an appropriate choice of $\widetilde M(\smallyng{5pt}{2,1})$ is $\{\mathbf{13}\}$.

  However, this idempotent is not generous since $N^\Delta_e(\smallyng{5pt}{2,1}) = \emptyset$.
  It is also not lavish since $N^\Delta_e(\smallyng{5pt}{3}) = \{\mathbf{123}\}$ but $N^D_e(\smallyng{5pt}{3}) = \emptyset$.
\end{example}

\subsection{The case of generalised Jones-Wenzl idempotents}\label{sec:lp_hecke}
Let us now turn to a central example.
Let $A = \TL_n$ over pointed ring $(k,\delta)$ with $(\ell, p)$-torsion.
It is well known that $A$ is cellular with $\Lambda = \{ 0 \le m \le n \;:\; m \equiv_2 n\}$,
the sets $M(m)$ being monic $(n,m)$-diagrams, the map $\iota$ the natural duality on the Temperley-Lieb category and $C^m_{S,T}$ being the natural image of $\Delta(m)\otimes\iota\Delta(m)$.

Let $e$ be the idempotent describing the projective cover of the trivial module.
The trivial module is a composition factor of $\Delta(m)$ iff $m \in \supp n$.  If so, it does so with multiplicity 1.

Notice now that
\begin{equation}
  \Hom_{A}(Ae, \Delta(m)) \cong e\Delta(m) = \widetilde\Delta(m) \subseteq \Delta(m),
\end{equation}
where a morphism is identified with the image of $e$.
Suppose that $m \in \supp n$, and let $S \subset \Delta(m)$ be the largest (by inclusion) submodule of $\Delta(m)$ which does not have the trivial module as a composition factor.

Given non-zero $\phi_1, \phi_2 \in \Hom_A(Ae,\Delta(m))$, the images of $e$ under the maps, $v_1 = \phi_1(e)$ and $v_2 = \phi_2(e)$ are non-zero in $\Delta(m)/ S$.
Both lie in the single trivial submodule of $\Delta(m)/S$ and thus there is a linear combination, $\alpha_1v_1 + \alpha_2v_2$ that vanishes modulo $S$.
But then we have that the image of $\alpha_1\phi_1 + \alpha_2\phi_2$ lies in $S$, which has no trivial factors.
Thus indeed $\alpha_1\phi_1 + \alpha_2\phi_2 = 0$ so the two were collinear to begin with.

Hence we see that
\begin{equation}\label{eq:dim_hom_edelta}
  \dim \Hom_{A}(Ae, \Delta(m)) = \dim \widetilde\Delta(m) =
  \begin{cases}
    1 & m \in \supp n\\
    0 & m \not \in \supp n
  \end{cases}
\end{equation}

\begin{remark}\label{rem:lavish_not_eve}
Readers may be surprised that this is the case.
Surely if $\Delta(m-2)$ has a non-zero element indexed by diagram $S$\footnote{If $e \Delta(m) \neq 0$, then the non-zero element can be chosen to be $e\cdot S$, by the discussion preceding \cref{def:generous}},
then simply ``popping'' one of the outermost links of the diagram $S$ should give a diagram $S' \in \Delta(m)$ and since $0\neq e\cdot  S = e\cdot S' \cdot u$, where $u$ is a simple cup diagram, we must have that $e \cdot S' \neq 0$ and so $e\widetilde\Delta(m)\neq 0$?
The key here lies in the fact that these diagram manipulations are taken modulo different conditions.
While indeed $0\neq e\cdot S' \cdot u$ as morphisms in $\TLcat$, the map $e\cdot S'$ factors through $m-2$ and thus vanishes in $\widetilde\Delta(m)$.
Equivalently, $e$ is not lavish.

Now, it is clear that when $m\not\in \supp n$, that $N^\Delta_e(m) = M(n)$.
On the other hand, when $m\in \supp n$, there may be multiple diagrams in $\Delta(m)$ not killed by $e$.
Indeed, the diagram ${\rm d}_n^m$ is one such (canonical) example, but other ``ancestor centred'' options exist (see~\cite{tubbenhauer_wedrich_2019}).
Hence $e$ is not even generous.
\end{remark}

We can conclude
\begin{corollary}
  Let $\supp n = \{s_1 < s_2 < \cdots < s_{2^\generation{n}} = n\}$ and morphism
  $x_i : \underline{n} \to \underline{s_i}$ such that $e\cdot x_i$ has through degree $s_i$.
  Then for any morphism $u : \underline{s_i} \to \underline{m}$ with through degree less than $s_i$, the morphism $e\cdot x_i \cdot u$ has through degree at most $s_{i-1}$.
\end{corollary}

Further, a direct application of \cref{cor:semisimple} shows that if $\ell = +\infty$ then $\widetilde A$ is semi-simple.

%% file: valenced.tex
\section{Valenced Temperley-Lieb or Jones-Wenzl Algebras}\label{sec:valenced}
A composition of $n \in \N$ into $r$ parts, written $\bmu\vdash n$ is a tuple $(\mu_1, \mu_2, \ldots, \mu_r)$ such that $0\le \mu_i$ and $\sum_{i=1}^r \mu_i = n$.
The sub-tuple $(\mu_1, \mu_2, \ldots, \mu_{r-1})$ will be written as $\hat\bmu \vdash n-\mu_r$.
We will denote the distinguished composition $(1,1,\ldots, 1) \vdash n$ as $\boldsymbol{n}$.

Throughout, fix a pointed ring $(k,\delta)$ of $(\ell,p)$-torsion.
Recall the definition of $\TLcat$ and let $e^{n}$ be the $(\ell,p)$-Jones-Wenzl idempotent on $n$ strands over $k$.
Note that $e^{n}$ is not ``the'' Jones-Wenzl idempotent, $\JW_n$, unless $n < \ell$ or $n = a p^{(r)}-1$ for $1\le a \le \ell$ if $r=0$ and $1\le a \le p$ else.
Such $n$ will be called Eve.

For any $\bmu \vdash n$, let $e^\bmu$ be the idempotent on $n$ strands given by
\begin{equation}
  e^{\bmu} = e^{\mu_1} \otimes \cdots \otimes e^{\mu_r}.
\end{equation}
An idempotent $e^{\bmu}$ will be termed Eve if all $\mu_i$ are Eve.

Our interest lies in $\TL_\bmu = e^\bmu\cdot \TL_n \cdot e^\bmu$.
This is a $k$-subspace of $\TL_n$, and an algebra with unit $e^\bmu$.
That makes it a (non-unital) subalgebra of $\TL_n$ but it is not a unital subalgebra (since the units differ).
It is isomorphic to the endomorphism ring $\End_{\TL_n}(\TL_n \cdot e^\bmu)$.
\begin{remark}
  In~\cite{flores_peltola_2018a, flores_peltola_2018b}, the algebra $\TL_\bmu$ is called the ``Jones-Wenzl algebra''.
  The Valenced Temperley-Lieb algebra defined in~\cite{flores_peltola_2018b} is not an associative algebra in general, and in fact, multiplication is only defined when $\bmu$ is Eve.
  However, when $\bmu$ is Eve, the concepts coincide, which is why we consider the terms to be interchangeable whenever the word ``algebra'' is involved.
\end{remark}

\begin{example}\label{eg:rings_tl_mu}
  We present some of the rings $\TL_\bmu$ for simple $\bmu$ as well as some of their representation theory.
  \begin{enumerate}[(i)]
    \item If $\bmu = \boldsymbol{n}$ then $e ^\bmu = \id$ so $\TL_{\boldsymbol n} = \TL_n$.

    \item On the other hand, if $n$ is Eve and $\bmu = (n)$ then $\TL_\bmu \simeq k$.
      The case when $n$ is not Eve is a subject of~\cite{tubbenhauer_wedrich_2019} and \cref{sec:lp_hecke} for $\ell = p$.  It is shown that
      \begin{equation}\label{eq:two_part_eve_simeq}
        \TL_{(n)} \simeq k[X_1,\ldots, X_r]/(X_1^2,\ldots, X_r^2).
      \end{equation}
      This algebra has a single simple module, which is linear, on which the images of the $X_i$ act as zero.
    \item Suppose that $\bmu = (\mu_1, \mu_2)\vdash n$ is such that $\mu_i < \ell$.  In~\cite[Lemma 3.5]{flores_peltola_2018a} it is shown that this is a quotient of a polynomial algebra with one generator.
      The proof supplied there determines the form of this algebra but does not state the result, so we show it below.

      The element $U_1$, defined as $e^\bmu\cdot u_{\mu_1} \cdot e^\bmu$ generates the algebra, and if $U_k$ is the element $e^\bmu\cdot u^{(k)}_{\mu_1}\cdot e^\bmu$, then $\{U_k \;:\; 0\le k \le \min\{\mu_1,\mu_2\}\}$ gives a $k$-basis.
      Here $u^{(k)}_i$ is the (possibly not simple) cap of $k$ strands centred just after $i$.
      It is shown that\todohidden{This isn't quite what is written there, but is what the computer says is true. Check the working in FP18}
      \begin{equation}\label{eq:two_part_recurse}
        U_1 \cdot U_k = \frac{[k][n - k + 1]}{[\mu_1][\mu_2]} U_k + \frac{[\mu_1-k][\mu_2-k]}{[\mu_1][\mu_2]} U_{k+1}.
      \end{equation}
      Since $\mu_i < \ell -1$, we see that $[\mu_i]\neq 0$ in our ring $k$ and so we may consider the alternative normalisation $\widetilde U_k = U_k [\mu_1][\mu_2]$ which descends to $k$ so then 
      \begin{equation}
        \widetilde U_{k+1} =\frac{\widetilde U_1  - [k][n-k+1]}{[\mu_1-k][\mu_2-k]}\cdot \widetilde U_k
      \end{equation}
      It is clear that $\widetilde U_{k}$ is non-zero for all $0\le k < \min\{\bmu\}$ but that if we extend this definition, $\widetilde U_{k+1} = 0$.
      Hence a polynomial satisfied by $\widetilde U_1$ is
      \begin{equation}
        f(X) = \prod_{r = 0}^{\min\{\bmu\}} \left(X - [r][n-r+1]\right).
      \end{equation}
      Thus $\TL_\bmu \simeq k[X]/(f(X))$.
      There are $\min\{\bmu\}$ simple, linear modules on which $X$ acts as the scalar $[r][n-r+1]$.
      
      We consider this example in more detail in \cref{sec:two-part}.

    \item Let $\bmu = (k,1,\ldots, 1) \vdash n$ and set $b = n - k$.  When $\bmu$ is Eve, this is the ``seam algebra'' studied in~\cite{langlois_remillard_saint_aubin_2020}.
  \end{enumerate}
\end{example}

As in \cref{sec:cell_hecke}, we inherit cellular data for $\TL_\bmu$ from $\TL_n$.
Following \cref{sec:notation}, we let $L_\bmu(m)$, $\Delta_\bmu(m)$ and $P_\bmu(m)$ be the simple, cell and projective modules indexed by $m$ respectively.

%% file: valenced-cell-data.tex
\section{Cell Data for Jones-Wenzl Algebras over Eve Composition}\label{sec:valenced_cell_data}
In general,
since $\TL_n$ is cellular, we may use the results of \cref{sec:cellular} to study $\TL_\bmu$.
However, this may prove to be complicated, depending on the composition $\bmu$.

Here we determine the cellular data $(\Lambda_\bmu, M_\bmu(\lambda), C_\bmu, \iota)$ for the algebra $\TL_\bmu$ when $\bmu$ is Eve.
The task to be done is to determine the appropriate restrictions of $M(\lambda)$ and $\Lambda$, since we will show momentarily that $e^\bmu$ is both generous and lavish.

\subsection{The Sets of Indices}
We determine which $m\in \Lambda$ we will need to drop -- i.e. which $\Delta_\bmu(m) = 0$.
\begin{proposition}\label{prop:eve_mu_generous_lavish}
  If $\bmu$ is Eve, then $e^\bmu$ is both lavish and generous in the sense of \cref{def:generous}.
\end{proposition}
\begin{proof}

  Fix some $m \in \Lambda$.
  %Recall that $\{C^m_{S,R} \;:\; S,R \in M(m)\}$ is the set of all diagrams with through degree $m$.
  Let the ``boundary'' be the set of sites 
  \begin{equation}
    \{ 1,2,\ldots,n-1\} \setminus \{\mu_1, \mu_1+\mu_2,\ldots,\mu_1 + \mu_2 +\cdots+\mu_{r-1}\}
  \end{equation}
  and
  let $B(m) \subseteq \Delta(m)$ be the linear span of all diagrams in $M(n)$ which have a simple left cap in the boundary.
  It is clear that $B(m)$ is killed by the action of $e$ as $\bmu$ is Eve and so $N^\Delta_e(m) \supseteq B(m)$.
  Moreover, since every non-identity diagram in the support of $e^\bmu$ has a simple cap on the boundary, $e^\bmu$ acts as the identity on the vector space $\Delta(m) / B(m)$.
  Hence $N^\Delta_e(m)= B(m)$.

  Now suppose that
  \begin{equation}
    \sum_{S \in M(m)\setminus N^\Delta_e(m)} \alpha_{S} \;e \cdot S = 0.
  \end{equation}
  Then certainly $\sum_{S \in M(m)\setminus N^{\Delta}_e(m)}\alpha_{S} S \in N^\Delta_e(m) = B(m)$, which is a contradiction since we are summing over diagrams not living in $B(m)$.
  Hence $\{ e\cdot S \;:\; S \in M(m) \setminus N^\Delta_e(m)\}$ is a linearly independent set and $e^\bmu$ is generous.

  If $S \in N^\Delta_e(\lambda)$ then $S$ has a left boundary cap.
  Clearly then $e \cdot C^m_{S,T} = 0$ for any $T$ and so $S \in N^{D}_e(m)$.
  Hence $e^\bmu$ is lavish.
\end{proof}
If $\bmu$ is not Eve, then $e^\bmu$ is no longer generous or lavish, as shown in \cref{rem:lavish_not_eve}.

When $\bmu$ is Eve, we have shown that $\Delta_\bmu(m)$ has a basis consisting of all diagrams without a simple cap on the boundary.
This allows us to enumerate which $m$ lie in $\Lambda_\bmu$.

\begin{example}\cite[eq. 3.8]{langlois_remillard_saint_aubin_2020}\label{eg:when_boundary_sites}
  If $\bmu = (k,1,1,\ldots,1)\vdash n$ is Eve, then
  \begin{equation}
    \Lambda_\bmu = \{m \in \N_0 \;:\; 2k-n \le m \le n, \;\text{and}\; m \equiv_2 n \;\}.
  \end{equation}
  Indeed, in this case, the boundary consists of the sites $\{1, \ldots, k-1\}$.
  It is clear that a diagram with $r$ loops (so with $m = n-2r$) is possible for each $0\le r \le n-k$.
  Indeed if $r < k$, simply connect sites $\{k-r,\ldots, k\}$ to $\{k+1,\ldots, k+r\}$.
  Otherwise, connect $\{1,\ldots,k\}$ to $\{k+1,\ldots,2k\}$ and then pair up remaining sites until $r$ loops are made.

  On the other hand, if $n-k < r$, by pigeon-hole principle, there is a boundary site connected to another, and so a simple loop on the boundary.  Hence such a diagram is killed by $e^\bmu$.
  \todohidden[inline]{Make diagram?}
\end{example}

\begin{example}
  If $\bmu = (\mu_1, \mu_2)$ is Eve, then $\Lambda_\bmu = E_{\mu_1,\mu_2}$, as defined in \cref{eq:two_part_e}.
\end{example}

Suppose that diagram $x \in \Delta(m)$ is not killed by $e^\bmu$.
If $m < n$ then $x = x' u$ where $u$ is a simple cap diagram $\underline{m+2} \to \underline{m}$.
It is then the case that $e^\bmu x'$ is not zero in $D(m+2)$ and, because $e^\bmu$ is lavish, this means that it is not zero in $\Delta(m+2)$.
As such, for each Eve composition $\bmu\vdash n$, there is a $s_\bmu$ such that the set $\Lambda_\bmu$ is of the form
\begin{equation}\label{eq:Lambda_bmu}
  \Lambda_\bmu = \{s_\bmu, s_\bmu+2,\ldots, n\}.
\end{equation}

\begin{lemma}\cite[Lemma 2.3]{flores_peltola_2018b}\label{lem:eve_composition_valids}
  For Eve composition $\bmu =(\mu_1,\ldots,\mu_r)\vdash n$,
  \begin{equation}\label{eq:recurse_smin}
    s_{(n)} = n\quad\quad ;
    \quad\quad
    s_{\bmu} = \begin{cases}
      s_{\hat\bmu} - \mu_r, & \mu_r \le s_{\hat\bmu}\\
      s_{\hat\bmu} - \mu_r\mod 2, & s_{\hat\bmu} < \mu_r < |\hat\bmu|\\
      \mu_r - |\hat\bmu|, & |\hat\bmu| \le \mu_r\\
    \end{cases}
  \end{equation}
\end{lemma}

\begin{example}
  Suppose that $\bmu = (5,3,4,7,22)\vdash 41$.  Then we can calculate, using the above
  \begin{equation}
    s_{(5)} = 5\quad\quad
    s_{(5,3)} = 2\quad\quad
    s_{(5,3,4)} = 0\quad\quad
    s_{(5,3,4,7)} = 1\quad\quad
    s_{(5,3,4,2,22)} = 3
  \end{equation}
  On the other hand, if $\bmu = (22,7,4,3,5)\vdash 41$, then
  \begin{equation}
    s_{(22)} = 22\quad\quad
    s_{(22,7)} = 15\quad\quad
    s_{(22,7,4)} = 11\quad\quad
    s_{(22,7,4,3)} = 8\quad\quad
    s_{(22,7,4,3,5)} = 3.
  \end{equation}
  This illustrates that \cref{lem:eve_composition_valids} gives the same result for any composition and its reverse --- a fact that does not follow easily from the definition.
\end{example}

This set will also be critical for evaluating non-Eve compositions.
Thus if $\bmu$ is any composition of $n$ (Eve or otherwise), let $E_\bmu$ be the set given by \cref{eq:Lambda_bmu} with $s_\bmu$ as in \cref{eq:recurse_smin}.
Notice that if $\bmu$ is two-part, this coincides with the definition of $E_{r,s}$ given in \cref{eq:two_part_e}.

\subsection{The Sets of Tableaux}
Let $\bmu = (\mu_1, \ldots,\mu_r)\vdash n$ be Eve.
We now turn to enumerating the tableaux in $M_\bmu(m)$ for $m \in \Lambda_\bmu$.
Recall that diagrams from $\underline{n} \to \underline{m}$ are in bijection with standard $(\frac{n+m}{2}, \frac{n-m}{2})$-Young tableaux.
The diagrams in $\Delta_\bmu(m)$, are then such diagrams without a simple link in the boundary.

Let the ``$i$-th bucket'' be the subset of sites,
\begin{equation}
  B_i = \{\mu_1+\cdots\mu_{i-1}+1,\mu_1+\cdots\mu_{i-1}+1,\cdots,\mu_1+\cdots\mu_{i}\}.
\end{equation}
Then we seek all diagrams without a link between two points in the same bucket.
Given such a diagram
it is clear then that for each $i$ there is a $t_i$ such that the first $t_i$ sites in bucket $B_i$ are connected to sites of smaller index by the diagram.
The remaining sites are either free (defects) or are connected to sites in later buckets.

Thus an alternative characterisation of this diagram is thus by a tuple $\mathbf{t} = (t_1, t_2, \ldots, t_r)$ with $0 \le t_i \le \mu_i$ for each $i$ and $n - \sum_{i=1}^n t_i = m$ such that
\begin{equation}
  \sum_{i=1}^{r-1} \mu_i - t_i \ge t_r \quad\quad\text{and}\quad\quad t_1 = 0.
\end{equation}
We also construct the tuple $\bro = \bmu - \mathbf{t} = (\mu_1 - t_1, \ldots, \mu_r - t_r)$.

This can also be envisioned as a walk on a planar lattice.
The walk begins at $(0,0)$ and ends at $(n, m)$.
For each $0 \le i \le n$, it moves down and right one unit if $i$ is a site connected to a smaller index and up and right otherwise.
At no point may the walk cross the $x$-axis and each ``bucket'' consists of some number of steps down followed by some number of steps up.

\begin{example}\label{eg:walk1}
  Let $\bmu = (3,2,4,4,2) \vdash 15 = n$ and $m = 2$.
  Then the tuple $\mathbf{t}=(0,1,2,1,2)$ corresponds to diagram
  \begin{center}
    \begin{tikzpicture}[scale=0.7]
      \draw (0.5,0) -- (15.5,0);
      \foreach \i in {1,...,15} {
        \fill (\i,0) circle (0.1);
      }
      \draw[very thick] (1,0) -- (1,2);
      \draw[very thick] (2,0) edge[out=90,in=90] (7,0);
      \draw[very thick] (3,0) edge[out=90,in=90] (4,0);
      \draw[very thick] (5,0) edge[out=90,in=90] (6,0);
      \draw[very thick] (8,0) -- (8,2);
      \draw[very thick] (9,0) edge[out=90,in=90] (10,0);
      \draw[very thick] (11,0) -- (11,2);
      \draw[very thick] (12,0) edge[out=90,in=90] (15,0);
      \draw[very thick] (13,0) edge[out=90,in=90] (14,0);

      \draw[dashed,red] (3.5,-0.15) -- (3.5,1.5);
      \draw[dashed,red] (5.5,-0.15) -- (5.5,1.5);
      \draw[dashed,red] (9.5,-0.15) -- (9.5,1.5);
      \draw[dashed,red] (13.5,-0.15) -- (13.5,1.5);

      \draw [decorate,decoration={brace,amplitude=5pt}] (3.2,-.2) -- (.7,-.2) node [black,midway,yshift=-10pt] {\footnotesize $B_1$};
      \draw [decorate,decoration={brace,amplitude=5pt}] (5.2,-.2) -- (3.7,-.2) node [black,midway,yshift=-10pt] {\footnotesize $B_2$};
      \draw [decorate,decoration={brace,amplitude=5pt}] (9.2,-.2) -- (5.7,-.2) node [black,midway,yshift=-10pt] {\footnotesize $B_3$};
      \draw [decorate,decoration={brace,amplitude=5pt}] (13.2,-.2) -- (9.7,-.2) node [black,midway,yshift=-10pt] {\footnotesize $B_4$};
      \draw [decorate,decoration={brace,amplitude=5pt}] (15.2,-.2) -- (13.7,-.2) node [black,midway,yshift=-10pt] {\footnotesize $B_5$};
    \end{tikzpicture}
  \end{center}
  has $\bro = (3,1,2,3,0)$ and walk
  \begin{center}
    \begin{tikzpicture}[scale=0.7]
      \draw (-0.5,0) -- (16.5,0);
      \fill (1,1) circle (0.1);
      \fill (2,2) circle (0.1);
      \fill (3,3) circle (0.1);
      \fill (4,2) circle (0.1);
      \fill (5,3) circle (0.1);
      \fill (6,2) circle (0.1);
      \fill (7,1) circle (0.1);
      \fill (8,2) circle (0.1);
      \fill (9,3) circle (0.1);
      \fill (10,2) circle (0.1);
      \fill (11,3) circle (0.1);
      \fill (12,4) circle (0.1);
      \fill (13,5) circle (0.1);
      \fill (14,4) circle (0.1);
      \fill (15,3) circle (0.1);
      \draw[thick] (0,0) -- (3,3) -- (4,2) -- (5,3) -- (7,1) -- (9,3) -- (10,2) -- (13,5) -- (15,3);
      \foreach \i in {0,2,4,6} {
        \draw[dotted, thin] (-.2,\i+.2) -- (\i+.2,-.2);
        \draw[dotted, thin] (16+.2,\i+.2) -- (16-\i-.2,-.2);
      }
      \foreach \i in {8,10,12,14,16} {
        \draw[dotted, thin] (\i-6-.2,6+.2) -- (\i+.2,-.2);
        \draw[dotted, thin] (16-\i+6+.2,6+.2) -- (16-\i-.2,-.2);
      }
      \foreach \i in {18,20,22} {
        \draw[dotted, thin] (\i-6-.2,6+.2) -- (16+.2,\i-16-.2);
        \draw[dotted, thin] (16-\i+6+.2,6+.2) -- (-.2,\i-16-.2);
      }
      \draw[dashed,red] (3.5,-0.2) -- (3.5,6.2);
      \draw[dashed,red] (5.5,-0.2) -- (5.5,6.2);
      \draw[dashed,red] (9.5,-0.2) -- (9.5,6.2);
      \draw[dashed,red] (13.5,-0.2) -- (13.5,6.2);

      \draw[very thick, green!60!black, ->] (0,0.2) -- (3,3.2);

      \draw[very thick, orange, ->] (3.6,2.6) -- (4,2.2);
      \draw[very thick, green!60!black, ->] (4,2.2) -- (5,3.2);

      \draw[very thick, orange, ->] (5.6,2.6) -- (7,1.2);
      \draw[very thick, green!60!black, ->] (7,1.2) -- (9,3.2);

      \draw[very thick, orange, ->] (9.6,2.6) -- (10,2.2);
      \draw[very thick, green!60!black, ->] (10,2.2) -- (13,5.2);

      \draw[very thick, orange, ->] (13.6,4.6) -- (15,3.2);
    \end{tikzpicture}
  \end{center}
  Note that the rises in the walk (indicated in green) are of lengths given by $\bro$ and the falls (in orange are given by $\textbf{t}$.
\end{example}

We now wish to count the number of such walks.  This will give us the dimension of the cell module $\Delta_\bmu(m)$.

For a composition (of any number), $\bet$, let $C^\bet_m$ be the number of walks from $(0,0)$ to $(|\bet|, m)$ that do not cross the origin and which obey the ``down-then-up'' rule within each of the buckets described by $\bet$.
Such a walk is termed a ``walk over $\bet$''.
\begin{proposition}\cite[Lemma 2.8]{flores_peltola_2018b}\label{prop:count_Cmun}
  If $\bmu = (n)\vdash n$,
  \begin{equation}\label{eq:count_Cmun_1}
    C^{(n)}_m = \begin{cases}
      1 & n =  m\\
      0 & \text{else}
    \end{cases}
  \end{equation}
  and if $\bmu = (\mu_1,\ldots, \mu_r)\vdash n$ for $r > 1$,
  \begin{equation}
    C^\bmu_n = \sum _{g = 0}^{\min\{\mu_r, m\}} C^{\hat\bmu}_{m + \mu_r- 2g }.
  \end{equation}
\end{proposition}
\begin{proof}
  \Cref{eq:count_Cmun_1} is clear.
  Consider then some valid walk ending at $(|\bmu|,m)$.
  The final bucket of the walk must consist of a fall, followed by a rise of length $0\le g \le \mu_r$.  The $x$ coordinate of the walk increases by $\mu_r$ and the $y$ coordinate by first decreases by $\mu_r - g$ and then increases by $g$ for a  net change of $2g - \mu_r$.
  However, to prevent the walk from sinking below the $x$-axis, we require that $g \le m$.
  The formula follows.
\end{proof}
Note that this is still an inherently recursive formula - $\hat\bmu$ is just the composition with the last element removed for which we will need to know the values of $C^{\hat\bmu}_n$ to calculate $C^\bmu_n$.

However, let us reiterate the importance of this number:
\begin{equation}
  \dim \Delta_\bmu(m) = C^{\bmu}_m,
\end{equation}
and this is non-zero exactly when $m$ lies in the set described by \cref{eq:Lambda_bmu}.
\begin{example}
  Setting $\bmu = \boldsymbol{n}$, we recover the well known recurrence for the Catalan triangle (which computes the Catalan numbers on the diagonal)
  \begin{equation}
    C^{\boldsymbol{n}}_m = \begin{cases}
      0&m > n\\
      1&m = n\\
      C^{\boldsymbol{n-1}}_1&m = 0\\
      C^{\boldsymbol{n-1}}_{m-1} + C^{\boldsymbol{n-1}}_{m+1} &\text{else}\\
    \end{cases}
  \end{equation}
  The solution to this recursion is that $C^{\boldsymbol{n}}_m$ vanishes if $n$ and $m$ are of different parity, and if they are of the same parity,
  \begin{equation}
    C^{\boldsymbol{n}}_m = 
      \binom{n}{(n-m)/2} - \binom{n}{(n-m)/2-1}.
  \end{equation}
  We have recovered, naturally, the dimension of the cell module $\Delta(m)$.
\end{example}

With this, we conclude our analysis of the cellular data of $\TL_\bmu$ when $\bmu$ is Eve.
Our choice of diagrams are those corresponding to walks over $\bmu$, which are counted by \cref{prop:count_Cmun} and this leaves us with indices described by \cref{eq:Lambda_bmu}.

%% file: valenced-cell-modules.tex
\section{Gram Matrices for Cell Modules}\label{sec:gram}
Recall the notation of $\Lambda_0$ for all cell indices with non-degenerate cell module forms from \cref{sec:cellular}.
We recall and expand a useful lemma.
\begin{lemma}\cites[Proposition 3.4]{flores_peltola_2018b}[Lemma 7.1]{spencer_2020}\label{lem:folklore_1}
  Suppose $k$ is a field and fix cell indices $m \in (\Lambda_\bmu)_0$ and $m' \in \Lambda_\bmu$.
  Then for any submodules $M \subseteq \Delta_\bmu(m)$ and $M'\subseteq \Delta_\bmu(m')$, let $\theta : \Delta_\bmu(m)/M \to \Delta_\bmu(m')/M'$  be a $\TL_\bmu$-morphism.
  Then
  \begin{enumerate}[(i)]
    \item If $m < m'$ then $\theta = 0$.
    \item If $m = m'$ then $\theta(z+M') = \lambda_\theta z + M$ for some scalar $\lambda_\theta \in k$.
    \item If $m \ge m'$ then there is a morphism $v_\theta : \underline{m'} \to \underline{m}$ such that $\theta(z + M') = z \cdot v_\theta + M'$ and $v_\theta$ is in the linear span of monic $\underline m' \to \underline m$ diagrams.
  \end{enumerate}
  Further, the image is cyclic.
\end{lemma}
\begin{proof}
  Since $m \in (\Lambda_\bmu)_0$, there is an  $0 \neq x \in \Delta_\bmu(m) \setminus \rad\Delta_\bmu(m)$.  Then $x \neq 0$ in $\Delta_\bmu(m) / M$ and $x$ generates $\Delta_\bmu(m)$.
  Since $k$ is a field, and $x \not \in \rad\Delta_\bmu(m)$, there is a $y \in \Delta_\bmu(m)$ such that $\langle x, y \rangle = 1$.
  We may lift $\theta$ to a $\TL_\bmu$-map $\Delta_\bmu(m) \to \Delta(m')/M'$ along the natural projection $\Delta_\bmu(m) \to \Delta_\bmu(m)/M$.

  But now for any $z \in \Delta_\bmu(m)$,
  \begin{equation}
   \phi (|z\rangle) =
    \phi(|z\rangle\langle x|y\rangle) =
    |z\rangle\langle x|\theta(|y\rangle).
  \end{equation}
  This makes it clear that the image of $\theta$ is generated by $\theta(|y\rangle)$ and is thus cyclic.

  Now, suppose $m < m'$.
  Then the algebra element $|z\rangle\langle x|$ lies in  $\TL_\bmu^{\le m}$  and so kills all quotients of $\Delta_\bmu(m')$.  Hence the map $\theta$ is zero.

  If $m \ge m'$ then note that part (ii) is a special case of part (iii).
  Then $\langle x|\theta(|y \rangle)$ is a morphism from $m\to m'$.
  Any diagram in $\langle x|\theta(|y \rangle)$ that is not monic factors through $m'' < m'$ and thus does not contribute to the resultant element of $\Delta_\bmu(m')/M'$.
  Thus such diagrams can be removed to obtain the morphism $v_\theta$.
\end{proof}

\begin{proposition}
  If $k$ is a characteristic zero field (so that $p= \infty$), $\bmu$ is Eve and $m \in (\Lambda_\bmu)_0$, then the radical of $\Delta_\bmu(m)$ is either zero or a simple module.
\end{proposition}
\begin{proof}
  Recall \cref{eq:e_rad_is_rad} where it was shown that $\rad \Delta_\bmu(m) = e^\bmu\cdot\rad_\mathbf{n}\Delta(m)$.
  In the characteristic zero case, if $\rad\Delta_\mathbf{n}(m)$ is non-zero it is simple and isomorphic to $L_\mathbf{n}(m')$ for some $m' > m$~\cite[Corollary 7.3]{ridout_saint_aubin_2014}.
  This means that $m' \in (\Lambda_\bmu)_0$ and so $e^\bmu\cdot L_\mathbf{n}(m') = L_\bmu(m')$.
  However, \cref{eq:exact_restrict_delta} shows that multiplication by $e^\bmu$ preserves simplicity, showing the result.
\end{proof}

\subsection{Trivalent Link States}
We define open and closed trivalent link states and expound upon some of their properties.

\begin{definition}\label{def:trivalent}
  Let $r,s,t \in \N$ such that $r \in E_{s,t}$ and let
  \begin{equation*}
    i = \frac{r + s - t}{2}
    \quad\quad\quad
    j = \frac{r - s + t}{2}
    \quad\quad\quad
    k = \frac{-r + s + t}{2}
  \end{equation*}
  The open trivalent link state is a shorthand for a diagram of the following form
    \begin{equation}
    \vcenter{\hbox{
    \begin{tikzpicture}
      \draw[very thick] (0,0) circle (0.2);
      \foreach \ang/\label in {0/r,120/s,240/t} {
        \begin{scope}[rotate=\ang]
          \draw[very thick] (0,0.2) -- (0,1.2);
          \node at (0,1.4) {$\label$};
        \end{scope}
      }
    \end{tikzpicture}}}
    \quad=\quad
    \vcenter{\hbox{
    \begin{tikzpicture}
      \draw[very thick] (1.03+0.05, -0.6+0.086) to[out=150,in=-90] (0.1,1.2);
      \draw[very thick] (-1.03-0.05, -0.6+0.086) to[out=30,in=-90] (-.1,1.2);
      \draw[very thick] (-1.03+0.05, -0.6-0.086) to[out=30,in=150] (1.03-0.05,-0.6-0.086);
      \node at (0,1.4) {$r$};
      \node at (1.2,-.7) {$t$};
      \node at (-1.2,-.7) {$s$};
      \node at (0.606,0.35) {$j$};
      \node at (-0.606,0.35) {$i$};
      \node at (0,-0.7) {$k$};
    \end{tikzpicture}}}.
    \end{equation}
    If $r$, $s$ and $t$ are all Eve, the closed trivalent link state is defined to be
    \begin{equation}
    \vcenter{\hbox{
    \begin{tikzpicture}
      \draw[very thick, fill=purple] (0,0) circle (0.2);
      \foreach \ang/\label in {0/r,120/s,240/t} {
        \begin{scope}[rotate=\ang]
          \draw[very thick] (0,0.2) -- (0,1.2);
          \node at (0,1.4) {$\label$};
        \end{scope}
      }
    \end{tikzpicture}}}
    \quad=\quad
    \vcenter{\hbox{
    \begin{tikzpicture}
      \draw[very thick] (0,0) circle (0.2);
      \foreach \ang/\label in {0/r,120/s,240/t} {
        \begin{scope}[rotate=\ang]
          \draw[very thick] (0,0.2) -- (0,0.5);
          \draw[very thick] (0,0.8) -- (0,1.2);
          \draw[very thick, fill=purple] (-.3,0.5) rectangle (0.3,0.8);
          \node at (0,1.4) {$\label$};
        \end{scope}
      }
    \end{tikzpicture}}}.
    \end{equation}
    where the three boxes represent the classical Jones-Wenzl idempotents $\JW_r$, $\JW_s$ and $\JW_t$ (all defined, since $r$, $s$ and $t$ are Eve).
\end{definition}

An important morphism is the so-called ``theta network'' of Kauffman and Lins~\cite{kauffman_lins_94}.
This is a map $\underline 0 \to \underline 0$ and thus can be identified with a scalar $\Theta(r,s,t)$.
It is given by evaluating the morphism
\begin{equation}
    \vcenter{\hbox{
    \begin{tikzpicture}[scale=.5]
      \draw[very thick] (0,0) circle (2);
      \draw[very thick] (-2,0) to (2,0);
      \draw[very thick, fill=purple] (2,0) circle (0.3);
      \draw[very thick, fill=purple] (-2,0) circle (0.3);
      \node at (0,2.3) {$r$};
      \node at (0,0.3) {$s$};
      \node at (0,-1.7) {$t$};
    \end{tikzpicture}}}
\end{equation}
For this to be well defined, we require that $r \in E_{s,t}$.
If $i,j,k$ are as in \cref{def:trivalent}, then
\begin{equation}
  \Theta(r,s,t) = \frac{(-1)^{i+j+k}[i+j+k+1]![i]![j]![k]!}{[i+j]![j+k]![k+i]!}.
\end{equation}
Note that unlike the quantum binomials, this is in general not a polynomial in $\delta$.
For example $\Theta(2,2,2) = -\frac{[4][3]}{[2]^2} = -(\delta^4 - 3\delta^2 + 1)/\delta$.  Further examples are:
\begin{align*}
  \Theta(r, s, r + s) &= (-1)^{r+s}[r+s+1]\\
  \Theta(2r', 2r', 2r') &= \frac{(-1)^{r'}[3r'+1]!([r']!)^3}{([2r']!)^3}
\end{align*}
A useful normalisation of the theta value is the following:
\begin{align}\label{eq:useful_gauss}
  \Theta(r,s,t)
          & = \frac{(-1)^{i+j+k}[i+j+k+1]![i]![j]![k]!}{[r]![s]![t]!}\\\nonumber
          & = \frac{(-1)^{i+j+k}[i+j+k+1]!}{[k]![i+j]!}\Big/ \gaussianquant{s}{k}\gaussianquant{t}{k}\\\nonumber
          & = (-1)^{i+j+k}[i+j+k+1]\gaussianquant{i+j+k}{k} \Big/ \gaussianquant{s}{k}\gaussianquant{t}{k}
\end{align}

With the language of trivalent nodes, the basis of $\Delta_\bmu(m)$ given by defect $m$ walks over $\bmu$ has a new description.

\begin{example}
  The tableau in \cref{eg:walk1} can further be described by the network
  \begin{center}
    \begin{tikzpicture}
      \draw (-0.5,0) -- (8.5,0);
      \draw[very thick] (0,0) to (1-0.07, 1-0.07);
      \node at (0.15,0.5) {3};
      \node at (1.35,1.7) {3};
      \node at (2.35,2.7) {3};
      \node at (3.35,3.7) {5};
      \draw[very thick] (2,0) to (1+0.07, 1-0.07);
      \node at (1.9,0.5) {2};
      \draw[very thick] (1,1) circle (0.1);
      \draw[very thick] (1+0.07,1+0.07) to (2-0.07, 2-0.07);
      \draw[very thick] (4,0) to (2+0.07, 2-0.07);
      \node at (3.9,0.5) {4};
      \draw[very thick] (2,2) circle (0.1);
      \draw[very thick] (2+0.07,2+0.07) to (3-0.07, 3-0.07);
      \draw[very thick] (6,0) to (3+0.07, 3-0.07);
      \node at (5.9,0.5) {4};
      \draw[very thick] (3,3) circle (0.1);
      \draw[very thick] (3+0.07,3+0.07) to (4-0.07, 4-0.07);
      \draw[very thick] (8,0) to (4+0.07, 4-0.07);
      \node at (7.9,0.5) {2};
      \draw[very thick] (4,4) circle (0.1);

      \draw[very thick] (4.,4.07) to (4,5);
      \node at (4.3,4.5) {3};

      \draw[fill=white] (-0.1,0.2) rectangle (0.5,0);
      \draw[fill=white] (1.5,0.2) rectangle (2.1,0);
      \draw[fill=white] (3.5,0.2) rectangle (4.1,0);
      \draw[fill=white] (5.5,0.2) rectangle (6.1,0);
      \draw[fill=white] (7.5,0.2) rectangle (8.1,0);
    \end{tikzpicture}
  \end{center}
  We will call this the \emph{ladder} form of the tableaux.
  This could, should we wish, be taken to the extreme:
  \begin{center}
    \begin{tikzpicture}[scale=0.7]
      \draw[thick] (1,1) -- (2,2);
      \draw[thick] (1,0) -- (1,1);
      \draw[thick] (2,0) -- (2,2);
      \draw[thick] (3,0) -- (3,3);
      \draw[thick] (4,0) -- (4,2);
      \draw[thick] (5,0) -- (5,3);
      \draw[thick] (6,0) -- (6,2);
      \draw[thick] (7,0) -- (7,1);
      \draw[thick] (8,0) -- (8,2);
      \draw[thick] (9,0) -- (9,3);
      \draw[thick] (10,0) -- (10,2);
      \draw[thick] (11,0) -- (11,3);
      \draw[thick] (12,0) -- (12,4);
      \draw[thick] (13,0) -- (13,5);
      \draw[thick] (14,0) -- (14,4);
      \draw[thick] (15,0) -- (15,3);
      \draw[very thick] (2,2) --(3,3) -- (4,2) -- (5,3) -- (7,1) -- (9,3) -- (10,2) -- (13,5) -- (15,3) -- (16,3);
      \draw (-0.5,0) -- (16.5,0);
      \draw[very thick, fill=white] (1,1) circle (0.1);
      \draw[very thick, fill=white] (2,2) circle (0.1);
      \draw[very thick, fill=white] (3,3) circle (0.1);
      \draw[very thick, fill=white] (4,2) circle (0.1);
      \draw[very thick, fill=white] (5,3) circle (0.1);
      \draw[very thick, fill=white] (6,2) circle (0.1);
      \draw[very thick, fill=white] (7,1) circle (0.1);
      \draw[very thick, fill=white] (8,2) circle (0.1);
      \draw[very thick, fill=white] (9,3) circle (0.1);
      \draw[very thick, fill=white] (10,2) circle (0.1);
      \draw[very thick, fill=white] (11,3) circle (0.1);
      \draw[very thick, fill=white] (12,4) circle (0.1);
      \draw[very thick, fill=white] (13,5) circle (0.1);
      \draw[very thick, fill=white] (14,4) circle (0.1);
      \draw[very thick, fill=white] (15,3) circle (0.1);

      \draw[fill=white] (0.7,0.2) rectangle (3.3,0);
      \draw[fill=white] (3.7,0.2) rectangle (5.3,0);
      \draw[fill=white] (5.7,0.2) rectangle (9.3,0);
      \draw[fill=white] (9.7,0.2) rectangle (13.3,0);
      \draw[fill=white] (13.7,0.2) rectangle (15.3,0);

      \node at (1.4,1.8) {1};
      \node at (2.4,2.8) {2};
      \node at (3.6,2.8) {3};
      \node at (4.4,2.8) {2};
      \node at (5.6,2.8) {3};
      \node at (6.6,1.8) {2};
      \node at (7.4,1.8) {1};
      \node at (8.4,2.8) {2};
      \node at (9.6,2.8) {3};
      \node at (10.4,2.8) {2};
      \node at (11.4,3.8) {3};
      \node at (12.4,4.8) {4};
      \node at (13.6,4.8) {5};
      \node at (14.6,3.8) {4};
      \node at (15.6,3.3) {3};
    \end{tikzpicture}
  \end{center}
  Here the link with a ``walk'' in the traditional sense is made explicit.

  Note in both these examples, we have drawn links in $\Delta_\mathbf{15}(3)$.
  The corresponding elements of $\Delta_{(3,2,4,4,2)}(3)$ would have $\JW$ projectors at the places marked with boxes.
\end{example}

Let us now work over $\Q(\delta)$.
The key operation we will be undertaking is to replace open nodes in ladder forms with closed ones.
This change of basis introduces a number of Jones-Wenzl idempotents, which may not be defined over $k$ (even if $\bmu$ is Eve).
However, the key result of this section is the form of the determinant of the Gram matrix for $\Delta_\bmu(m)$ and that is independent of the underlying field.

To be clear, this takes a ``diagram'' element of $\Delta_\bmu(m)$ (which we will identify with the tableaux $\mathbf{t}$ and the walk $\bro$) and replaces it by a morphism in $\Hom_{\TLcat}(\underline{n}, \underline{m})$.
We will quotient this by any morphisms factoring through objects less than $\underline m$ to get our resulting element of $\Delta_\bmu(m)$.

\begin{proposition}\cite[Lemma 4.6]{flores_peltola_2018b}
  The set of all ladders with filled in trivalent nodes forms a basis for $\Delta_\bmu(m)$.
\end{proposition}
\begin{proof}
  A sketch of the proof is to introduce a partial order on the tuples $\{\bro\}$ (lexicographically) and then to show that introduction of Jones-Wenzl idempotents along the diagonal give terms that are smaller.
\end{proof}

\begin{proposition}\cite[Proposition 4.7]{flores_peltola_2018b}
  The determinant of the Gram matrix has form
  \begin{equation}
    \det G^m_\bmu = \prod_{\bro}\prod_{i=1}^{r-1}
    \frac{\Theta(\rho_i, \rho_{i+1}, \mu_{i+1})}{[\rho_{i+1}+1]}
  \end{equation}
  where the product is over all walks $\bro$ over $\bmu$.
\end{proposition}
Unfortunately, in almost all cases, this form of the determinant is unwieldy.
\begin{example}~
  \begin{enumerate}[(i)]
    \item Recall that the form of the determinant for the cell modules of $\TL_n = \TL_{\mathbf{n}}$ is~\cite{ridout_saint_aubin_2014}
      \begin{equation}
        \det G_n^m = \prod_{j = 1}^{\frac{n-m}{2}} \left(
          \frac{[m + j + 1]}{[j]}
        \right) ^ {\dim \Delta_n(m + 2j)}.
      \end{equation}
      Now, notice that if $\rho_{i+1} = \rho_i + 1$ that $\Theta(\rho_i, \rho_{i+1}, 1)/[\rho_{i+1}+1] = 1$ and if $\rho_{i+1} = \rho_i - 1$ that $\Theta(\rho_i, \rho_{i+1}, 1)/[\rho_{i+1}+1] = [\rho_i+1]/[\rho_i]$.
      Let $R^+$ be the set of walks over $\mathbf{n-1}$ that end at $m-1$ and $R^-$ be those that end at $m+1$ (note that either of these may be empty).
      This partitions the set of all walks ending at $m$, by those that finish $\rho_n = \rho_{n-1}+1$ and $\rho_n = \rho_{n-1}-1$.
      Using inducting on $n$,
      \begin{align*}
        \det G_{\mathbf{n}}^m &=
        \left(
          \prod_{\bro \in R^+} \prod_{i = 1}^{r-2}
          \frac{\Theta(\rho_i, \rho_{i+1}, \mu_{i+1})}{[\rho_{i+1}+1]}
        \right)
        \left(
          \prod_{\bro \in R^-} \prod_{i = 1}^{r-2}
          \frac{\Theta(\rho_i, \rho_{i+1}, \mu_{i+1})}{[\rho_{i+1}+1]}
          \frac{[m+2]}{[m+1]}
        \right)\\
                              &=
        \left(
          \det G_{n-1}^{m-1}
        \right)
        \left(
          \det G_{n-1}^{m+1}
        \right)
        \left(
          \frac{[m+2]}{[m+1]}
        \right)^{\dim \Delta_{n-1}(m+1)}
      \end{align*}
      which we recognise as the recursion relation for the Gram determinants of $\TL_n$ as shown in~\cite[equation 4.19a]{ridout_saint_aubin_2014}.
    \item If $\bmu = (n)$ then the only path $\bro$ over $\bmu$ is $\bro = (n)$ and the product is empty.
      Note that the only valid $m$ in this situation is $n$ itself.
    \item If $\bmu = (\mu_1, \mu_2) \vdash n$, where without loss, $\mu_1 > \mu_2$, then again there is a unique $\bro$ for any $m$ given by $\bro = (\mu_1, m)$.
      The determinant is thus $\Theta(\mu_1, m, \mu_2)/[m+1]$ as should be expected by evaluating the following morphism modulo diagrams of non-maximal through degree.
      \begin{center}
        \begin{tikzpicture}
          \draw[very thick] (0.5,0.2) arc (0:180:0.5);
          \draw[very thick] (0.5,-0.2) arc (0:-180:0.5);
          \draw[very thick] (0.75,1) -- (0.75,-1);
          \draw[very thick] (-0.75,1) -- (-0.75,-1);
          \draw[fill=white] (-1,0.2) rectangle (-.3, -.2);
          \draw[fill=white] (1,0.2) rectangle (.3, -.2);
          \node at (0,0.4) {$i$};
          \node at (0,-0.4) {$i$};
          \node at (-.65,0) {$\mu_1$};
          \node at (.65,0) {$\mu_2$};
        \end{tikzpicture}
      \end{center}
      This leads to some interesting behaviour.
      For example, if $\bmu = (3,3)\vdash 6 = n$, then $[5] \mid \det G_\bmu^2$ but $[5]\nmid \det G_\bmu^m$ for any other $m\in \{0,2,4,6\}$.
      We can evaluate \cref{eq:two_part_eve_simeq}
      \begin{equation}
        \TL_{(3,3)}\simeq k[X]/(X, X-[6], X-[2][5], X-[3][4]).
      \end{equation}
      If we specialise to a pointed ring where $\ell = 5$, then this simplifies to $k[X]/(X, X + [4], X - [2])$ which has three simple modules, in line with the three values of $m$ for which the Gram determinant does not vanish.

      This highlights a key feature.
    Though $\Lambda_\bmu$ may be ``easy'' to calculate and may contain a contiguous stretch of weights, $(\Lambda_\bmu)_0$ will not be so well behaved.
    \item If $\bmu = (k, 1,\ldots, 1) \vdash n$ and $b = n-k$,we recover a recursion formula almost identical to that in part (i) and, after calculating the base case of $b = 0$, find that
      \begin{equation}
        \det G_{(k,1^b)}^m = \prod_{j = 1}^{\lfloor k/2\rfloor} \left(
          \frac{[j]}{[k-j+1]}
        \right)^{\dim \Delta_{(k-2j,1^{k+2j})}(m)}
        \prod_{j = 1}^{(n-m)/2}\left(
          \frac{[m + j + 1]}{[j]}
        \right)^{\dim \Delta_{(k,1^b)}(m+2j)}
      \end{equation}
  \end{enumerate}
\end{example}

The observation in \cref{eq:exact_restrict_delta} has an interesting consequence.
\begin{lemma}
  If
  \begin{equation*}
    \prod_{j = 1}^{\frac{n-m}{2}} \left(
      \frac{[m + j + 1]}{[j + 1]}
    \right) ^ {\dim \Delta_n(m + 2j)}\neq 0
  \end{equation*}
  in $(k, \delta)$, then
  \begin{equation*}
    \prod_{\bro}\prod_{i=1}^{r-1}
    \frac{\Theta(\rho_i, \rho_{i+1}, \mu_{i+1})}{[\rho_{i+1}+1]}
    \neq 0
  \end{equation*}
  for all Eve $\bmu$ such that $m \in (\Lambda_\bmu)_0$.
\end{lemma}

%% file: endomorphism.tex
\section{The Projective Endomorphism Algebra}\label{sec:endomorphism}

Let $\bmu = (n)$, a single-part, not necessarily Eve composition.
Then elements of the algebra $\TL_{(n)}$ are of the form
\begin{center}
  \begin{tikzpicture}
    \foreach \x in {0, 1.5} {
      \begin{scope}[shift={(\x,0)}]
        \foreach \i in {6,...,14} {
          \draw[very thick] (-.2,.2+ \i/5) -- (.5,.2 + \i/5);
        }
        \draw[thick,fill=purple] (0,1.3) rectangle (0.3,3.1);
        \draw[thick,fill=black] (0,3.0) rectangle (0.3,3.1);
      \end{scope}
    }
    \draw[thick] (.5,1.3) rectangle (1.3, 3.1);
    \draw [decorate,decoration={brace,amplitude=5pt},xshift=-1pt,yshift=0pt] (-0.3,1.3) -- (-0.3,3.1) node [black,midway,xshift=-10pt] {\footnotesize $n$};
  \end{tikzpicture}
\end{center}
where the purple boxes are $(\ell,p)$-Jones-Wenzl idempotents and the white box is any element of $\TL_n$.

Recall from \cref{eg:rings_tl_mu} (ii) that
\begin{equation}\label{eq:TL_JW_is_poly}
  \TL_\bmu \simeq k[X_1, \ldots, X_{\generation{n}}]/(X_1^2, \ldots, X_{\generation{n}}^2)
\end{equation}
where $n + 1 = \sum_{i = 1}^{\generation{n}+1} s_i p^{(d_i)}$ are the digits of the $(\ell, p)$-adic expansion of $n+1$.
In fact, the case studied in~\cite{tubbenhauer_wedrich_2019} is only for $\delta = -2$ and so one must extend their argument to mixed (rather than positive) characteristic.
However, the result as stated above still holds.

If $S \subseteq \{1, \ldots, \generation{n}\}$, write $X_S$ for $\prod_{i \in S} X_i$.
These $S$ can be compared to the ``down-admissible sets'' from~\cite{tubbenhauer_wedrich_2019}.
They are in bijection with numbers $m_S = n - 2\sum_{i \in S} s_i p^{(d_i)}$ which form $\supp(n)$.
Further the set $\supp(n)$ is exactly the indices of those cell modules of $\TL_{n}$ which have a trivial composition factor.

\begin{lemma}
  The set $\supp(n) = \{m_S \;:\; S\subseteq\{1,\ldots, \generation{n}\}\} \subseteq \Lambda$ indexes the cell modules of $\TL_\bmu$, all of which are one dimensional.
  Further, $(\Lambda_\bmu)_0 = \{n\}$.
\end{lemma}
\begin{proof}
  The explicit construction in~\cite{tubbenhauer_wedrich_2019} constructs $X_S = e^\bmu {\rm d}^m_{n} {\rm u}^m_{n}e^\bmu$.
  Recall the morphism ${\rm d}^m_{n} : \underline {n} \to \underline {m}$ is actually a diagram (and not a linear combination of diagrams).

  Now recall that $e M \simeq \Hom(Ae, M)$ for any idempotent $e$.
  In this case $\TL_{n}\cdot e^\bmu$ is the projective cover of the trivial module and by the last comment above, the trivial module of $\TL_{n}$ appears in each of the cell modules $\{\Delta(m_S) \;:\; S \subseteq \{1,\ldots, \generation{n}\}\}$.
  Hence there is a morphism from the projective cover of the trivial into each of those, so $e^\bmu\cdot\Delta(m_S)\neq 0$.

  The algebra $\TL_\bmu$ has $k$-dimension $2^{\generation{n}}$ since
  a $k$-basis is given by all the $\{X_S : S \subseteq \{1, \ldots, \generation{n}\}\}$.
  Recall that the dimension of a cellular algebra is the sum of the squares of the dimensions of its cell modules (see \cref{prop:a_lambda_decomp}):
  \begin{equation}
    \dim A = \sum_{\lambda\in \Lambda} \big(\dim \Delta(\lambda)\big)^2.
  \end{equation}

  However, we have already constructed $2^{g_r}$ cell modules $\Delta_\bmu(m_S) = e^\bmu\cdot\Delta(m_S)$.
  Hence we deduce that these are all one dimensional.
  An induction proof shows that $\Delta_\bmu(m_S) = \operatorname{span}\{e^\bmu\cdot {\rm d}^{m_S}_{n}\}$.
  Note that an inductive proof is necessary as, without knowing that the cell modules are not zero, we could not rule out the possibility that some $e^\bmu{\rm d}^{m_S}_n$ was not of through degree less than $m_S$.

  However, the algebra in \cref{eq:TL_JW_is_poly} has a single simple module on which all the indeterminates act as zero.
Thus we expect $(\Lambda_\bmu)_0 \subseteq \Lambda_\bmu = \supp n$ to be a singleton.
It is clear that $n \in (\Lambda_\bmu)_0$ and so we deduce that the inner product is zero on all cell modules apart from the trivial one.
\end{proof}

Note that the final part of this proof indicates that ${\rm u}_n^m e^{(n)} {\rm d}_n^m$ has through degree less than $m$ whenever $m<n$ over positive characteristic.

Finally, we note that the decomposition numbers of this algebra are trivial: all the cell modules are isomorphic and one-dimensional.

This can be demonstrated by the below image.
Here each row corresponds with a value of $n$ and the columns are values of $m$.
A gray dot indicates that $e^{(n)}\Delta_n(m) = 0$ and an orange dot denotes those $\Delta_{(n)}(m)$ which are one dimensional, i.e. those $m$ lying in $\Lambda_{(n)}$
The dot is circled if that value of $m$ lies in $(\Lambda_{(n)})_0$.

\begin{center}
  \includegraphics[width=0.5\textwidth]{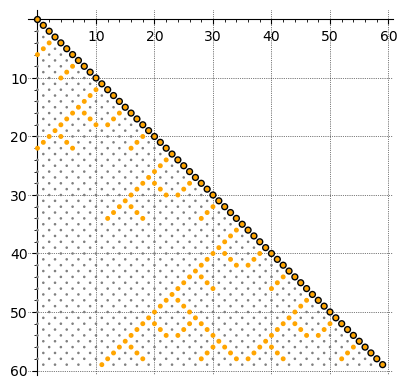}
\end{center}

This diagram (or diagrams similar to it) appear in~\cite[figure 3]{spencer_2020}, ~\cite[figure 1]{sutton_tubbenhauer_wedrich_zhu_2021} and~\cite[figure 1]{jensen_williamson_2015}.

Further enlightenment can be obtained by a subtle change of basis.
While not strictly necessary for analysing $\TL_\bmu$ for one-part $\bmu$, this will be very useful in analysing other algebras.

Throughout, we have been tacitly imagining $\Delta(\lambda)$ as a quotient of $D(\lambda)$.
In terms of our morphisms, we have been considering $\Delta(m)$ as the set of all \emph{monic} morphisms - that is morphisms in $D(m) \simeq \Hom_{\TLcat}(\underline{n}, \underline{m})$ with propagation number equal to $m$.
Thus, when we consider ``the'' element $e^\bmu {\rm d}_n^m$ as a basis of $\Delta_\bmu(m)$ we are really considering its image.

While it is sometimes convenient to use this element (as it is a idempotent multiplied by diagram), an alternative is available.
Consider the morphism $\overline{\rm d}_n^m$ over characteristic zero.
This is a non-zero element of $\Hom(n,m)$, but might not descend to mixed characteristic due to the presence of the various Jones-Wenzl idempotents.
The important fact about this element is neatly encapsulated by the following lemma.
\begin{lemma}\cite[Proposition 3.2]{burrull_libedinsky_sentinelli_2019}\label{lem:mutual_orthogonal}
  Suppose $m, m' \in \supp n$.
  Then
  \begin{equation}
    \lambda_n^{m'} \; \overline{\rm u}_n^{m'}\cdot
    \overline{\rm d}_n^m 
     = \delta_{m,m'}\JW_m.
  \end{equation}
\end{lemma}

This tells us several things.
Firstly, as an equation over characteristic zero,
\begin{equation}\label{eq:idemoptent_not_kill_nice_basis}
  e^\bmu\cdot\overline{\rm d}_n^m
  = \sum_{m'\in \supp n}
  \lambda_n^{m'}
  \overline{\rm d}_n^{m'} \overline{\rm u}_n^{m'}
    \left(\overline{\rm d}_n^m \right)
    =
    \overline{\rm d}_n^m
\end{equation}
\begin{lemma}\label{lem:suitable_multiples_of_basis}
  The monic submorphism of $\overline{\rm d}_n^m$ descends to $e^{(n)}{\rm d}_n^m$ over characteristic $(\ell, p)$.
\end{lemma}
\begin{proof}
  {\bf Step 1} {\it The coefficient of the diagram ${\rm d}_n^m$ in $\overline{\rm d}_n^m$ is exactly 1.}
  We can prove this by induction on the generation of $n$.
  Clearly it is true for all $m$ in the support of an Eve $n$.
  Then, if it is so for $n$ of generation $g-1$, consider the ladder construction from \cref{def:ladder}.
  The only terms that contribute to the coefficient of ${\rm d}_n^m$ are those that have the form ${\rm d}_{\mother[m']{n}}$ for some $m'$, followed by an optional Jones-Wenzl idempotent.
  The diagram ${\rm d}_n^m$ then only occurs from the identity coefficient in $\JW_i$ and so has unit coefficient as desired.
  This is effectively half the proof that the change of basis of $\Delta_{\mathbf{n}}(m)$ from $\{{\rm d}_n^m\}_m$ to $\{\overline{\rm d}_n^m\}$ is upper uni-triangular.

  {\bf Step 2} {\it The coefficient of the diagram ${\rm d}_n^m$ in $e^{(n)}{\rm d}_n^m$ is exactly 1.}
  We can consider the morphism $e^{(n)}{\rm d}_n^m$ modulo morphisms of through degree $<m$.  In this case, we can write
  \begin{equation}\label{eq:down_morphism_expansion}
    e^{(n)} {\rm d}_n^m = \sum_{m\le m' \in \supp n}\lambda_n^{m'}\overline{\rm d}_n^{m'}\overline{\rm u}_n^m{\rm d}_n^m
  \end{equation}
  However, if $m' > m$ then $\overline{\rm u}_n^{m'}{\rm  d}_n^{m'}$ is of the form of $\JW_{m'}$ multiplied by a morphism of through degree at most $m$ and hence vanishes unless $m'=m$.
  Thus \cref{eq:down_morphism_expansion} expands to $\lambda_n^m \overline{\rm d}_n^m\overline{\rm u}_n^m{\rm d}_n^m$.
  Finally, note that by construction $\overline{\rm u}_n^m{\rm d}_n^m = \overline{\rm u}_n^m\overline{\rm d}_n^m = (\lambda_n^m)^{-1}$.
  Then step 1 finishes the sub-proof.

  {\bf Step 3} {\it Some multiple of the monic part of $\overline {\rm d}_n^m$ descends to mixed characteristic:}
  We consider the construction of $\overline{\rm d}_n^m$ from \cref{def:ladder}.
  If $0\in S$, then all idempotents occurring in the definition are of the form $\JW_i$ where $i \equiv_\ell -1$.
  The general theory of Jones-Wenzl idempotents (see \cite{spencer_2020, ridout_saint_aubin_2014}) then tells us that these descend to the ring $\Z[X]_{\mathfrak m}$ where $\mathfrak m$ is the ideal generated by the minimal polynomial of $\delta$.
  As such, the entire morphism descends to that ring.

  On the other hand, if $0\not\in S$ we must be more careful.
  Here, we note that all but the last $\JW_i$ in the definition of $\overline{\rm d}_n^m$ exist over $\ell$.
  However, if we consider only monic diagrams, all terms but the identity of the final idempotent disappear.
  Hence, if we consider the monic image of $\overline{\rm d}_n^m$, this is defined over $\Z[X]_{\mathfrak m}$.

  Either way, the monic image of $\overline{\rm d}_n^m$ is defined over $\Z[X]_\mathfrak{m} \subseteq \Q(X)$.
  If we quotient out by the ideal $\mathfrak{m}\Z[X]_\mathfrak{m}$ we are left with coefficients  with denominators in $\Z[X]/\mathfrak{m}$.
  In this ring, the ideal $(p)$ is principle and maximal.
  Hence we can multiply by a suitable power of $p$ to obtain an element that descends to our characteristic $(\ell, p)$.

  {\bf Step 4} {\it This multiple of the monic part of $\overline {\rm d}_n^m$ is fixed (modulo $< m$) by the action of $e^{(n)}$, and hence must be a multiple of $e^{(n)}{\rm d}_n^m$ over mixed characteristic.}
  This follows, since $\overline{\rm d}_n^m$ is fixed by $e^{(k)}$ over characteristic zero so its multiple is too.
  If we quotient out by morphisms of through degree less than $m$ we still get an morphism fixed by $e^{(n)}$.  Since both this morphism and $e^{(n)}$ exist over the mixed characteristic ring it is still fixed.
  But the space of fixed morphisms is one dimensional so it must be a scalar multiple of $e^{(n)}{\rm d}_n^m$.

  {\bf Step 5} {\it By steps 1 and 2 above, the multiple must be exactly 1.}
\end{proof}

We also present an alternative, more direct proof by calculation.
This proof relies on the evaluation principle~\cite{goodman_wenzl_1993}.
\begin{proof}
  We consider calculating $e^{(n)}{\rm d}_n^m$ in characteristic zero.
  Now, if $m' > m$ then clearly $\overline {\rm u}_n^{m'}{\rm d}_n^m$ vanishes since $\JW_{m'}\overline {\rm u}_n^{m'} = \overline {\rm u}_n^{m'}$.
  Thus we may write
  \begin{equation}\label{eq:what_is_down_baby_dont_hurt}
    e^{(n)}{\rm d}_n^m = \sum_{m'\in\supp n}\lambda_n^{m'}\overline{\rm d}_n^{m'}\overline{\rm u}_n^{m'}{\rm d}_n^m
    = \sum_{m \ge m'\in\supp n}\lambda_n^{m'}\overline{\rm d}_n^{m'}\overline{\rm u}_n^{m'}{\rm d}_n^m
  \end{equation}
  However, if we quotient out by morphisms of through degree less than $m$ we get that \cref{eq:what_is_down_baby_dont_hurt} evaluates to $\overline{\rm d}_n^m$.

  Now we carefully apply the evaluation principle.  Since both $e^{(n)}$ and ${\rm d}_n^m$ are defined over mixed characteristic, so too must their product be.
  Since their product, up to through degree less than $m$ is $\overline{\rm d}_n^m$, this must hold over positive characteristic.
\end{proof}

From this we can obtain information on the scalars $\lambda_n^m$ in the definition of the $(\ell, p)$-Jones-Wenzl idempotents from \cref{sec:notation}.
Recall that they were defined such that $\overline{\rm u}_n^m\overline{\rm d}_n^m = (\lambda_n^m)^{-1}\JW_m$.
But by considering this modulo morphisms of through degree $<m$, we see that in fact $(\lambda_n^m)^{-1}$ is simply the inner product on the cell module $\Delta_{(n)}(m)$.
This gives a possibly surprising corollary.

\begin{corollary}
For each $m \in \supp(n)$, $(\lambda_n^m)^{-1}$ is defined over mixed characteristic.
Further, over this characteristic, it is zero if $m\neq n$ and one if $m=n$.
\end{corollary}

%% file: seam.tex
\section{The Seam Algebra}\label{sec:seam}

Here we examine in more detail one of the algebras that has been appearing in our examples.
We will change notation slightly in order to maintain parity with the characteristic zero literature, particularly~\cite{langlois_remillard_saint_aubin_2020}.
\begin{definition}
  The \emph{boundary seam algebra} is $\TL_\bmu$ where $\bmu = (k, 1^n)$.
\end{definition}
This algebra can be imagined diagrammatically as follows.
\begin{center}
  \begin{tikzpicture}
    \foreach \x in {0, 1.5} {
      \begin{scope}[shift={(\x,0)}]
        \foreach \i in {0,...,14} {
          \draw[very thick] (-.2,.2+ \i/5) -- (.5,.2 + \i/5);
        }
        \draw[thick,fill=purple] (0,1.3) rectangle (0.3,3.1);
        \draw[thick,fill=black] (0,3.0) rectangle (0.3,3.1);
      \end{scope}
    }
    \draw[thick] (.5,0.1) rectangle (1.3, 3.1);
    \draw [decorate,decoration={brace,amplitude=5pt},xshift=-1pt,yshift=0pt] (-0.3,1.3) -- (-0.3,3.1) node [black,midway,xshift=-10pt] {\footnotesize $k$};
    \draw [decorate,decoration={brace,amplitude=5pt},xshift=-1pt,yshift=0pt] (-0.3,0.1) -- (-0.3,1.3) node [black,midway,xshift=-10pt] {\footnotesize $n$};
  \end{tikzpicture}
\end{center}
Here the purple block is $e^k$ and the white rectangle can be any morphism in $\TL_{k+n}$.

Let us calculate the cell data for this algebra.
We have already done much of this in examples from the first few chapters when $\bmu$ is Eve.
We will now relax that assumption.
We begin by formalising \cref{eg:when_boundary_sites} into a lemma for later use.
\begin{lemma}\label{lem:count_dissapearing_morphisms}
  There is a monic diagram $\underline{a+n} \to \underline{m}$ such that none of the first $a$ sites are connected to each other iff $a-n \le m \le a+n$ and $m \equiv_2 a+n$.
  If so, there are
  \begin{equation}
    \binom{n}{(n + a - m)/2} - \binom{n}{(n - a - m - 2)/2}
  \end{equation}
  such diagrams.
\end{lemma}
Note that this is equivalent to counting the number of defect-$m$ walks over $\bmu = (a,1^n)$, i.e. the number of walks from $(a,a)$ to $(a+n, m)$ that do not cross the $x$-axis.

Now, consider an arbitrary monic diagram $\mathbf{t}:\underline{k + n} \to \underline{m}$.
Let $a_{\mathbf{t}}$ be the number of the first $k$ sites not connected to other sites in the first $k$.
In terms of the associated walk, $\bro$, this is exactly $a_{\mathbf{t}} = \rho_k$.
There is now a unique factorisation into monic diagrams $\mathbf{t} = (\mathbf{t}_0 \otimes \id_{n}) \circ \mathbf{t}_1$ where $\mathbf{t}_0 : \underline{k} \to \underline{a_{\mathbf{t}}}$ and $\mathbf{t}_1:\underline{a_{\mathbf{t}}+n}\to\underline{m}$ is a diagram which does not join any of the first $a$ sites.

Further, any $a \in \{k, k-2,\ldots\}$ and monic diagrams $\mathbf{t}_0 : \underline{k} \to \underline a$ and $\mathbf{t}_1: \underline{a + n} \to \underline{m}$ with the condition above uniquely describe a diagram from $\underline{k+n}$ to $\underline{m}$ such that exactly $a$ of the first $k$ sites are not connected to each other.

\begin{center}
  \begin{tikzpicture}[scale = 0.5];
    \draw (0,-.3) -- (20,-0.3);
    \draw[fill=yellow!50!white] (.5,-.3) rectangle (10.5,1.7);
    \draw[fill=blue!30!white] (.5,1.7) rectangle (19.5,3);
    \foreach \i in {1,...,19} {
      \draw[very thick] (\i,-.3) -- (\i,0);
    }
    \draw[very thick] (6,0) arc (0:180:1.5);
    \draw[very thick] (5,0) arc (0:180:.5);
    \draw[very thick] (9,0) arc (0:180:.5);
    \draw[very thick] (11,1.7) arc (0:180:.5);
    \draw[very thick] (13,1.7) arc (0:180:.5);
    \draw[very thick] (18,1.7) arc (0:180:.5);
    \foreach \i in {1,2,15,16,19} {
      \draw[very thick] (\i,0) -- (\i,3.3);
    }
    \foreach \i in {7,10,11,12,13,14,17,18} {
      \draw[very thick] (\i,0) -- (\i,1.7);
    }
    \draw[very thick] (7+.8,1.7+.8) arc (90:180:.8);
    \draw[very thick] (14,1.7) arc (0:90:.8);
    \draw[very thick] (14-.8,1.7+.8) -- (7+.8, 1.7+.8);
  \end{tikzpicture}
\end{center}
In the above figure, we have factorised a $\underline{19} \to \underline{5}$ diagram with $k = 10$.  This diagram has $a_{\mathbf{t}} = 4$ and $\mathbf{t}_0$ is highlighted in yellow.  The diagram $\mathbf{t}_1 : \underline{13} \to \underline{5}$ is shown in blue.
In terms of walks, we can represent this below:
\begin{center}
  \begin{tikzpicture}[scale=0.5]
    \foreach \i in {0,2,4,6} {
      \draw[dotted, thin] (-.2,\i+.2) -- (\i+.2,-.2);
      \draw[dotted, thin] (20+.2,\i+.2) -- (20-\i-.2,-.2);
    }
    \foreach \i in {8,10,12,14,16,18} {
      \draw[dotted, thin] (\i-6-.2,6+.2) -- (\i+.2,-.2);
      \draw[dotted, thin] (20-\i+6+.2,6+.2) -- (20-\i-.2,-.2);
    }
    \foreach \i in {20,22,24} {
      \draw[dotted, thin] (\i-6-.2,6+.2) -- (20+.2,\i-20-.2);
      \draw[dotted, thin] (20-\i+6+.2,6+.2) -- (-.2,\i-20-.2);
    }
    \draw (-0.5,0) -- (20.5,0);
    \draw[very thick,blue!80!black] (6,0)--(10,4)-- (11,3) -- (12,4) -- (14,2) -- (17,5) -- (18,4)--(19,5) ;
    \draw[very thick,yellow!90!black] (0,0) -- (4,4) -- (6,2) -- (8,4) -- (9,3);
    \draw[very thick,yellow!90!black,dashed] (9,3)  -- (10,4);
    \fill (1,1) circle (0.1);
    \fill (2,2) circle (0.1);
    \fill (3,3) circle (0.1);
    \fill (4,4) circle (0.1);
    \fill (5,3) circle (0.1);
    \fill (6,2) circle (0.1);
    \fill (7,3) circle (0.1);
    \fill (8,4) circle (0.1);
    \fill (9,3) circle (0.1);
    \fill (10,4) circle (0.1);
    \fill (11,3) circle (0.1);
    \fill (12,4) circle (0.1);
    \fill (13,3) circle (0.1);
    \fill (14,2) circle (0.1);
    \fill (15,3) circle (0.1);
    \fill (16,4) circle (0.1);
    \fill (17,5) circle (0.1);
    \fill (18,4) circle (0.1);
    \fill (19,5) circle (0.1);

  \end{tikzpicture}
\end{center}
The yellow walk corresponds to $\mathbf{t}_0$.
The blue walk corresponds to $\mathbf{t}_1$.
To factorise any walk into $\mathbf{t}_0$ and $\mathbf{t}_1$, we mark the point $(k, \rho_k)$ and colour the walk to the left yellow.
This is $\mathbf{t}_0$.
We then add a ramp from $(k-\rho_k, 0)$ to $(k,\rho_k)$ and add this to the remainder of the walk.  This is $\mathbf{t}_1$.
Note in the diagram above that the yellow and blue paths overlap briefly.

Fix $k$ and $n$ and consider the algebra $\TL_\bmu$.
We will now determine a valid set $M_\bmu(m)$.
Recall that a strategy for this is to sort the tableau in $M(m)$ and progressively take tableaux, skipping those that would result in a linearly dependent set in $\Delta_\bmu(m)$.

Sort the tableaux in $M(m)$ by $\mathbf{t} \prec \mathbf{u}$ if $a_{\mathbf t} < a_{\mathbf u}$.
We will consider the tableau in this order and select some and discard others.
Now, suppose that $\mathbf{t}$ factors into $\mathbf{t}_0$ and $\mathbf{t}_1$.
Suppose $a'$ is the through degree of $e^k \cdot \mathbf{t}_0$.
Then certainly $e^\bmu \cdot \mathbf{t}$ is in the span of $\{ e^\bmu \cdot \mathbf u \;:\; a_{\mathbf u} \le a'\}$.
If $a' < a_{\mathbf{t}}$ then $e^\bmu \cdot \mathbf{t}$ is in the span of basis elements we have chosen before.
Otherwise $a' = a_{\mathbf{t}}$ and so $e^k \cdot \mathbf{t}_0$ does not vanish in the cell module of $\TL_{(k)}$ indexed by $a$.
We know that this cell module is one-dimensional and so there is a ``canonical choice'' (which we will take to be ${\rm d}_{k}^a$) for the diagram $\mathbf{t}_0$.
We can use \cref{lem:count_dissapearing_morphisms} to count the number of diagrams we can choose for $\mathbf{t}_1$.
We deduce

\begin{proposition}\label{prop:tow_part_mu_dims}
  Let $\bmu = (k,1^n)$ be a (possibly not Eve) composition of $k + n$.
  Then
  \begin{equation}
    \Lambda_\bmu = \bigcup_{a \in \supp(k)}\{
    a-n, a-n + 2, \ldots, a+n
  \}\cap \N_0
  \end{equation}
  and
  \begin{equation}
    \dim \Delta_\bmu(m) =
    \sum_{a \in \supp(k)}
    \binom{n}{(n+a-m)/2} -
    \binom{n}{(n-a-m-2)/2}.
  \end{equation}
\end{proposition}
Note that in the case that $k$ is Eve, $\supp (k) = \{k\}$ and we recover the results of~\cite{langlois_remillard_saint_aubin_2020}.

\subsection{Eve Compositions}
Let us suppose that $k$ is Eve.
Our next step is to determine the set $(\Lambda_\bmu)_0$.
For a morphism $x : \underline{n} \to \underline{n}$, let $x_\id$ be the coefficient of the identity diagram in $x$.

We begin with a lemma regarding evaluating fractions of quantum numbers.
\begin{lemma}\label{lem:unreasonable_cancellation}
    If $i > 0$ and $b \neq 0$ then, over $(\ell, p)$-characteristic,
    \begin{equation}\label{eq:unreasonable_cancellation}
      \frac{[a p^{(i)}]}{[b p^{(i)}]} = \frac{a}{b}.
  \end{equation}
\end{lemma}
\begin{proof}
  This result is well known, but we would like to highlight a proof that does not depend on the ``$q$-formulation'' of quantum numbers.
  
  If we write the left hand side out in $\Q(\delta)$, we find that
  \begin{equation}
    \frac{[a\ell]_{\delta}}{[b\ell]_{\delta}} = \frac{[\ell]_{\delta}[a]_{\delta'}}{[\ell]_{\delta}[b]_{\delta'}}
    = \frac{[a]_{\delta'}}{[b]_{\delta'}}
  \end{equation}
  Here $\delta'$ is a polynomial in $\delta$ which is 2 modulo $[\ell]$.  Recall that $[n]_2 = n$ so that in $\Q(\delta)_{\mathfrak p}/\mathfrak p$ (where $\mathfrak p$ is a minimal integral polynomial of $\delta$), $[a\ell]/[b\ell] = a / b$.
  Now, if $\delta' = 2$ then $\ell = p$, and so \cref{eq:unreasonable_cancellation} follows by induction on $i$.
\end{proof}

\begin{lemma}\label{lem:tracing_jw}
    Let $n = a p^{(r)}$ for $0 \le a \le \ell \vee p$.
    Hence $\JW_{n-1}$ exists over characteristic $(\ell, p)$, and
    $\tau^k(\JW_{n-1})_{\id}$ vanishes unless $k = b p^{(r)}$ for some $0\le b<a$.
\end{lemma}
\begin{proof}
  The proof is an exercise in the evaluation principle.
  We know that over characteristic zero,
  \begin{equation}\label{eq:trace_jw_zero}
    \tau^k(\JW_m) = \frac{[m+1]}{[m+1-k]} \JW_{m-k}.
  \end{equation}
  Hence $\tau^k(\JW_{n-1})_\id = [n]/[n-k]$.
  However, in the morphism $\tau^k(\JW_{n-1})$, the coefficient of the identity diagram is a linear combination of the coefficients of diagrams in $\JW_{n-1}$.
  This linear combination takes coefficients from $\{1, \delta, \delta^2,\ldots\}$.
  Thus we can consider this statement as stating that a combination of elements (possibly multiplied by a suitable power of $\delta$) gives $[n]/[n-k]$:
  \begin{equation}
    \sum_{i \in I} \delta^{f_i} c_i = \frac{[n]}{[n-k]}.
  \end{equation}
  Now, since $n$ is Eve, all the coefficients $c_i$ descend to an element of our base ring, as does $\delta$.
  We deduce that $[n]/[n-k]$ must too.
  Indeed, \cref{eq:unreasonable_cancellation} assures us that these fractions {\it exist} even if they are not all non-zero.

  Finally, let $s$ be the maximal natural such that $p^{(s)} \mid n - k$ so we can write $[n]/[n-k] = [ap^{(r)}]/[b p^{(s)}] = a p^{(r-s)} / b$.
  Note that $s$ is, by definition, at most $r$.
  Clearly for this coefficient to not vanish in our ring, we must have $s=r$ giving the result.
\end{proof}

\begin{theorem}\label{prop:eve_seam_mu_0}
  Let $\bmu = (k, 1^n)$ be an Eve composition so $k+1 = k_r p^{(r)}$.
  Then,
  \begin{equation}
    (\Lambda_\bmu)_0 =
    \bigcup_{\lfloor k_r - n/p^{(r)}\rfloor\le b\le k_r}bp^{(r)} -1 + (\Lambda_{\mathbf{n-k+bp^{(r)}-1}})_0.
  \end{equation}
\end{theorem}
Recall that $(\Lambda_{\mathbf{n}})_0 = \{ m \in \Z_n^2 \;:\; m \equiv_2 n\}$ unless $\delta=0$ and $n$ is even in which case $(\Lambda_{\mathbf{n}})_0 = \{ m \in \Z_n^2\setminus\{0\} \;:\; m \equiv_2 n\}$.
\begin{proof}
  We want to find all $m \in \Lambda_{\bmu} = \{k -n \le m \le k + n\;:\; m \equiv_2 k + n\}$ such that there are two diagrams $\mathbf{u}, \mathbf{t} \in \Delta(m)$ such that $(\iota\mathbf{u}\cdot e^\bmu\cdot \mathbf{t})_\id\neq 0$.
  If such a pair exists, then the bilinear form is not degenerate on $\Delta_\bmu(m)$ and $m \in (\Lambda_\bmu)_0$.
If no such pair exists, since all elements of $\Delta_\bmu(m)$ are linear combinations of $e^\bmu \mathbf{t}$ we deduce the bilinear form is identically zero and so $m\not\in (\Lambda_\bmu)_0$.

  Now we examine the morphism $\iota\mathbf{u}\cdot e^\bmu\cdot \mathbf{t}$.
  Suppose the first $0\le a\le m$ sites of $\mathbf{t}$ are ``defects'' - not connected to other sources sites, and without loss, suppose this quantity is at least that for $\mathbf{u}$.
  We will be evaluating the morphism up to isotopy without expanding $e^k$.  Thus it will remain in the form $\iota\mathbf{u}\cdot e^\bmu\cdot \mathbf{t}$ up to a factor of $\delta^r$ for some $r$.
  If $\delta = 0$ then this vanishes identically unless $r = 0$, so we discard it.
  Otherwise we may assume that the $a+1$-st source site of $\mathbf{t}$ is not connected to the target sites of $\iota\mathbf{u}\cdot e^\bmu\cdot \mathbf{t}$, even after evaluation of the morphism (but not expanding $e^k$).
  \begin{center}
    \begin{tikzpicture}
     \draw[thick,fill=purple] (-.15,1.3) rectangle (0.15,3.1);
     \foreach \i in {0,...,6}
     \draw[very thick] (-.15,\i *0.2) -- (.15,\i*.2);
     \draw[thick] (-.15, -.1) -- (-.85,.2) -- (-.85, 2.8) -- (-.15,3.1) -- cycle;
     \node at (-.5,1.5) {$\iota\mathbf{u}$};
     \draw[thick] (.15, -.1) -- (.85,.2) -- (.85, 2.8) -- (.15,3.1) -- cycle;
     \node at (.5,1.5) {$\mathbf{t}$};
     \foreach \i in {0,...,3}
     \draw[very thick] (.15,2.6-\i *0.2) -- (1,2.6-\i*.2);
     \draw[very thick] (.15,1.8) edge[out=0,in = 0] (0.15,1);
     \draw[very thick] (.15,1.6) edge[out=0,in = 0] (0.15,1.2);
     \draw[very thick] (.15,.8) edge[out=0,in = 180] (1,0.8);
     \draw[very thick] (.15,.6) edge[out=0,in = 180] (1,0.6);
     \draw[very thick] (.15,.4) edge[out=0,in = 180] (1,0.4);
     \draw[very thick] (.15,.2) edge[out=0,in = 0] (.15,0.0);

    \draw [decorate,decoration={brace,amplitude=5pt},xshift=1pt,yshift=0pt] (1,2.7) -- (1,1.9) node [black,midway,xshift=10pt] {\footnotesize $a$};

     \foreach \i in {0,...,6}
     \draw[very thick] (-.85,2.1-\i *0.2) -- (-1,2.1-\i*.2);
    \end{tikzpicture}
  \end{center}

  {\bf Case $a \ge k$:}
  In this case, $\mathbf{t}$ and $\mathbf{u}$ can be written $\id_k \otimes \mathbf{t}'$ and $\id_k \otimes \mathbf{u}'$ respectively, for some monic diagrams $\mathbf{t}', \mathbf{u}' : \underline{n} \to \underline{m-k}$.
  If so,
  \begin{equation}
    (\iota\mathbf{u}\cdot e^\bmu\cdot \mathbf{t})_\id
    =
    \Big((\id_k\otimes\iota\mathbf{u'})\cdot (e^k\otimes\id_{m-k})\cdot (\id_k \otimes\mathbf{t'})\Big)_\id
    = (e^k)_\id (\iota\mathbf{u'}\cdot\mathbf{t'})_\id
  \end{equation}
  Since $(e^k)_\id = 1$, the problem reduces to asking which $m-k$ are in $(\Lambda_\textbf{n})_0$.
  Recall that these are all naturals of the same parity as $n$ with the exception of $0$ if $\delta = 0$.

  Hence the set of possible $m$ with such a pair of diagrams is $k + (\Lambda_\mathbf{n})_0$.

  {\bf Case $a < k$:}
  We consider which site the $a+1$-st source site of $\mathbf{t}$ is connected to after simplifying all loops and isotopies without expanding $e^k$.
  Recall that it is not connected to a target site.

  Suppose it connects to a source site, as in the below diagram.
  Let the first $a' < a$ sites of $\textbf{u}$ be defects.
  This must be a strict inequality as $a$ and $a'$ have opposite parity.
  Then the $a'+1$-st site of $\textbf{u}$ is not a defect so it must connect to some other point on the Jones-Wenzl idempotent.  This means that either two of the upper sites or two of the lower sites are connected and the entire morphism vanishes identically.
  \begin{center}
    \begin{tikzpicture}
      \draw[thick,fill=purple] (-.15,1.3) rectangle (0.15,3.1);
      \foreach \i in {-1,...,6}
      \draw[very thick] (-.15,\i *0.2) -- (.15,\i*.2);
      \node at (.4,1.4) {\footnotesize ?};
      \node at (-.6,1.2) {\footnotesize ?};
      \node at (-.6,2.2) {\footnotesize ?};

      \foreach \i in {0,...,1}{
        \draw[very thick] (-.85,\i *0.2+0.1) -- (-1,\i*.2+0.1);
        \draw[very thick] (.85,\i *0.2+0.1) -- (1,\i*.2+0.1);
      }

      \draw[very thick] (-.15, 1.6) -- (-.3,1.6) edge[out=180, in=180] (-.3,1.8);
      \draw[very thick] (-.3,1.8) -- (-.15,1.8);
      \draw[thick] (-.15, -.3) -- (-.85,0) -- (-.85, 0.4) -- (-.15,0.7) -- cycle;
      \draw[thick] (.15, -.3) -- (.85,0) -- (.85, 0.4) -- (.15,0.7) -- cycle;
      \foreach \i in {-2,...,2}
        \draw[very thick] (.15,2.6-\i *0.2) -- (1,2.6-\i*.2);
      \foreach \i in {-2,...,0}
        \draw[very thick] (-.15,2.6-\i *0.2) -- (-1,2.6-\i*.2);
      \draw[very thick] (.15,2.0) --(.4,2) edge[out=0,in = 0] (0.4,.8);
      \draw[very thick] (0.4,.8) -- (-1,.8);

     \draw [decorate,decoration={brace,amplitude=5pt},xshift=-1pt,yshift=0pt] (-1,2.5) -- (-1,3.1) node [black,midway,xshift=-10pt] {\footnotesize $a'$};
     \draw [decorate,decoration={brace,amplitude=5pt},xshift=1pt,yshift=0pt] (1,3.1) -- (1,2.1) node [black,midway,xshift=10pt] {\footnotesize $a$};
     \draw [decorate,decoration={brace,amplitude=3pt},xshift=1pt,yshift=0pt] (1,0.4) -- (1,0) node [black,midway,xshift=15pt] {\footnotesize $m-a$};
    \end{tikzpicture}
  \end{center}

  Hence for $\iota\mathbf{u}\cdot e^\bmu\cdot\mathbf{t}$ to not vanish identically, the $a+1$-st site of $\mathbf{t}$ must connect to the another site of the idempotent.
  This is illustrated in the diagram below, lest two upper or two lower sites of the Jones-Wenzl idempotent be joined, killing the morphism completely.
  \begin{center}
    \begin{tikzpicture}
      \draw[thick,fill=purple] (-.15,1.3) rectangle (0.15,3.1);
      \foreach \i in {-1,...,6}
      \draw[very thick] (-.15,\i *0.2) -- (.15,\i*.2);

      \foreach \i in {0,...,1}{
        \draw[very thick] (-.85,\i *0.2+0.1) -- (-1,\i*.2+0.1);
        \draw[very thick] (.85,\i *0.2+0.1) -- (1,\i*.2+0.1);
      }
      \node at (.5,1.6) {\footnotesize ?};
      \node at (-.5,1.6) {\footnotesize ?};

      \draw[thick] (-.15, -.3) -- (-.85,0) -- (-.85, 0.4) -- (-.15,0.7) -- cycle;
      \draw[thick] (.15, -.3) -- (.85,0) -- (.85, 0.4) -- (.15,0.7) -- cycle;
      \foreach \i in {-2,...,1}
        \draw[very thick] (.15,2.6-\i *0.2) -- (1,2.6-\i*.2);
      \foreach \i in {-2,...,1}
        \draw[very thick] (-.15,2.6-\i *0.2) -- (-1,2.6-\i*.2);
      \draw[very thick] (.15,2.2) --(.5,2.2) edge[out=0,in = 0] (0.5,.8);
      \draw[very thick] (0.5,.8) -- (.15,.8);
      \draw[very thick] (-.15,2.2) --(-.5,2.2) edge[out=180,in = 180] (-0.5,.8);
      \draw[very thick] (-0.5,.8) -- (-.15,.8);

     \draw [decorate,decoration={brace,amplitude=5pt},xshift=-1pt,yshift=0pt] (-1,2.3) -- (-1,3.1) node [black,midway,xshift=-10pt] {\footnotesize $a$};
     \draw [decorate,decoration={brace,amplitude=5pt},xshift=1pt,yshift=0pt] (1,3.1) -- (1,2.3) node [black,midway,xshift=10pt] {\footnotesize $a$};
     \draw [decorate,decoration={brace,amplitude=3pt},xshift=1pt,yshift=0pt] (1,0.4) -- (1,0) node [black,midway,xshift=15pt] {\footnotesize $m-a$};
     \draw [decorate,decoration={brace,amplitude=3pt},xshift=-1pt,yshift=0pt] (-1,0) -- (-1,0.4) node [black,midway,xshift=-15pt] {\footnotesize $m-a$};
    \end{tikzpicture}
  \end{center}

  Further, by \cref{lem:tracing_jw}, this vanishes identically unless $a = b p^{(r)}-1$ for some $1\le b\le k_r$.
  But we also require that $n \ge k - a$ and so $b p^{(r)} \ge k - n+1$.

  Suppose, after isotopy, that there are $n-k'$ remaining central sites not connected to the Jones-Wenzl idempotent in any way.
  It is clear that $m \in (\Lambda_\bmu)_0$ if $m-a \in (\Lambda_\mathbf{n-k'})_0$.
But $k'$ is of the same parity as $k-a$ and since $(\Lambda_\mathbf{n-k'})_0\subseteq(\Lambda_\mathbf{n-k+a})_0$ we may assume without loss that $k' = k-a$.
  Recall $a = b p^{(r)}-1$ so $k-a = (k_r-b)p^{(r)} \le n$.
Hence $m-a \in (\Lambda_\mathbf{n-k+a})_0$ if the inner product is not zero.

  Thus the indices coved by this case are
  \begin{equation}
  \bigcup_{\lfloor k_r - n/p^{(r)}\rfloor\le b\le k_r}bp^{(r)} -1 + (\Lambda_\mathbf{n-k+bp^{(r)}-1})_0.
\end{equation}
Note that the consideration $a \ge k$ is actually a special case of this, where $b = k_r$.
\end{proof}
\begin{example}
  To elucidate on the above we present the \cref{fig:seam dots}.
\begin{figure}[htpb]
  \centering
  \includegraphics[width=\textwidth]{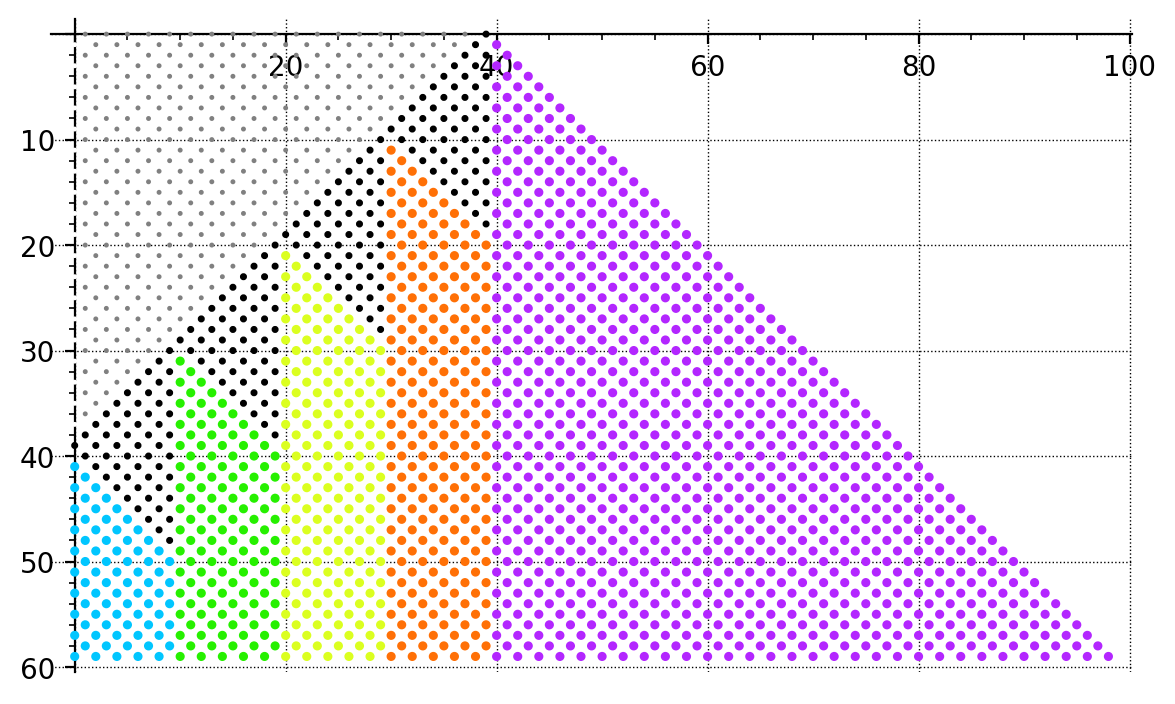}
  \caption{}%
  \label{fig:seam dots}
\end{figure}
Here we have picked $\ell = 2$ (so $\delta = 0$) and $p = 5$ with $k= 39 = [4,0,0]_{5,2}-1$ (which is Eve).
The $y$ axis indicates $n$, increasing downwards, and the $x$ axis $m$.
Note how $m \le n + k$ and that $n+k \equiv_2 m$, giving the offset ``chequerboard'' pattern.
The light grey circles indicate elements of $\Lambda _{\mathbf{k+n}}\setminus \Lambda_{\bmu}$ and the larger black and coloured circles elements of $\Lambda_{\bmu}$.
Circles shaded purple indicate values of $n$ and $m$ that are non-degenerate due to the first case in the proof of \cref{prop:eve_seam_mu_0}.  Those shaded in different colours come from the second case, with the different colours indicating different values of $b$.
The remaining circles, shaded black represent the elements of $\Lambda_{\bmu} \setminus \Lambda_{\bmu0}$.

We can read \cref{fig:seam dots} as follows.
Consider the partition $\bmu = (39,1^{11})$.
This is represented by the row labelled twelve (the first row being $(39, 1^0)$).
There are $25 = \lfloor(39 + 11 )/ 2\rfloor$ dots in this row, representing the 25 cell modules of $\TL_{\mathbf{50}}$.
Of those, the smallest thirteen are light grey and thus index cell modules that do not descend to $\TL_\bmu$.
The remaining twelve consist of five dots that are black (and so the cell modules represented by these are degenerate) and seven that are coloured.
We can deduce that the algebra $\TL_\bmu$ has seven simple modules in mixed characteristic $(2,5)$.
\end{example}

\subsection{Two Part Seam Algebras}\label{ssec:two_part_seams}
Now let us lift the restriction on $k$ but force $n = 1$.
This algebra has been studied in~\cite{sutton_tubbenhauer_wedrich_zhu_2021} where $e^\bmu$ is split into a sum of idempotents.

We can specialise \cref{prop:tow_part_mu_dims} at $n=1$:
\begin{proposition}
  Let $\bmu = (k,1)$ be a (possibly not Eve) composition of $k + 1$.
  Then
  \begin{equation}
    \Lambda_{\bmu} = \bigcup_{a \in \supp(k)}\{
    a-1, a+1
  \}
  \cap \N_0
  \end{equation}
  and
  \begin{equation}
    \dim \Delta_\bmu(m) =
    \begin{cases}
      2 & m \in (\supp k -1) \cap (\supp k + 1)\\
      1 & \text{ else }
    \end{cases}.
  \end{equation}
\end{proposition}

To calculate the dimension of $\TL_\bmu$, it will first be useful to know when two elements of $\supp n$ are separated by exactly 2, as this tells us which cell modules are two-dimensional.

\begin{definition}\label{def:tail_length}
  If $n + 1 = \pldigs{n_r, \ldots, n_0}$, we say $n$ has \emph{tail length} $t$ if $n_0 = \ell -1$ and $n_i = p-1$ for $0<i<t$.  Equivalently, $p^{(t_n)} \mid n+1$.
  The maximum tail length of $n$ is denoted $t_n$.
\end{definition}
Note that any given $n$ can have multiple tail lengths and that every $n$ has a tail of length 0 and so $t_n \ge 0$.

We let $\mother[k]{n}$ be the $k$-th (matrilineal) ancestor of $n$.
\begin{lemma}
  Let $n+1 = \pldigs{n_r, \ldots, n_0}$.
  If $m, m'$ are elements of $\supp(n)$ such that $m' = m + 2$ then $n$ has a tail of length $t$ with $n_{t} = 1$, and $m + 1 \in \supp(\mother[t+1]{n})$.
\end{lemma}
\begin{proof}
    Suppose that $m' = n[S_1]$ and $m = n[S_2]$ (notation from \cref{sec:notation}).  Then
    \begin{equation*}\label{eq:twoooo}
    2 = \sum_i n_i p^{(i)} - 2\sum_{i \in S_1} n_i p^{(i)} -
    \sum_i n_i p^{(i)} + 2\sum_{i \in S_2} n_i p^{(i)}
    = 2 \sum_{i\in S_2} n_i p^{(i)} - 2 \sum_{i\in S_1} n_i p^{(i)}.
    \end{equation*}
    \begin{equation}\label{eq:diff_two_one}
      \Rightarrow \sum_{i\in S_2\setminus S_1} n_i p^{(i)} -
                  \sum_{i\in S_1\setminus S_2} n_i p^{(i)} = 1.
    \end{equation}
    Note that the sets $S_1\setminus S_2$ and $S_2\setminus S_1$ are distinct.
    Further by considering \cref{eq:diff_two_one} modulo $\ell$ we see that one of the two contains $0$.

    If $0 \in S_2\setminus S_1$ then clearly $n_0 = 1$ and $S_1\setminus S_2 = \emptyset$ and $S_2 \setminus S_1 = \{0\}$.
    Here, $t = 0$ is a tail of $n$.

    On the other hand if $0\in S_1\setminus S_2$ then $n_0 = \ell-1$.
    By considering \cref{eq:diff_two_one} divided by $\ell$ and applying induction, we see that there is a tail $r>t>0$ and $\{0,1,\ldots, t-1\}\subset S_1\setminus S_2$ but $ S_2\setminus S_1 = \{t\}$ and $n_{t}=1$.
    
    Either way, $S_2 \setminus S_1 = \{t\}$ and $S_1\setminus S_2 = \{0,1,\ldots, t-1\}$.
    If $S_0 = S_1 \cap S_2$ then 
    \begin{equation}
        n[S_1] - 1 = n[S_2] + 1 = \mother[t+1]{n}[S_0]
    \end{equation}
    which lies in $\supp(\mother[t+1]{n})$ as desired.
\end{proof}
\begin{example}
  Let $n=513\,398$ so that $n+1 = [4,0,2,3,1,4,4,7]_{5,8}$.
  Then $n$ has maximal tail length $t_n = 3$ and, for example, $513\,000 =[4,0,2,3,1,-4,-4,-7]_{5,8}$ and $512\,998 =[4,0,2,3,-1,4,4,7]_{5,8}$ differ by 2.
  
  Now, $\mother[4]{n} = 513999$ so that $\mother[4]{n}+1 = [4,0,2,3,0,0,0]_{5,8}$.
  It is clear that the support of $\mother[4]{n}$ has 4 elements so we deduce that there are 4 pairs of elements of $\supp n$ differing by 2.
\end{example}

If $n_{t} \neq 1$ for any tail length $t$ of $n$, we term $n$ \emph{interior}.
We see that two elements of $\supp n$ differ by two iff $n$ is not interior.

We deduce:
\begin{proposition}
  Let $k + 1 = \pldigs{n_r,\ldots, n_{0}}$ and $T_n = \{0 \le t \le t_n \;:\; n_t = 1\}$ be the set of all tails that start with the digit 1.
  Then if $\bmu = (k,1)$,
  \begin{equation}\label{eq:dim_of_ind_one}
    \dim \TL_\bmu =
    \begin{cases}
      2^{\generation{n}+1}& n \text{ interior}\\
      2^{\generation{n}+1} + \sum_{t \in T_n}2^{\generation{n}-t}& \text{else}\\
    \end{cases}.
  \end{equation}
\end{proposition}
\begin{proof}
  The cell modules of $\TL_\bmu$ are one-dimensional, unless they are indexed by a number lying exactly between two elements of $\supp n$ separated by 2, in which case they are two-dimensional.
  Hence the number of two-dimensional cell modules is $\sum_{t \in T_n}2^{\generation{m}-t-1}$.

  Recall that the dimension of a cellular algebra is the sum of the squares of the dimensions of its cell modules.  The result follows.
\end{proof}

\begin{remark}
  The form of \cref{eq:dim_of_ind_one} is typical of what we will see going forward.
  Note that there is actually only one case, as $T_n = \emptyset$ if $n$ is interior.
  Further, in the majority of cases, $T_n$ is either zero or a singleton.
  Indeed, it is only if $\ell = 2$ or $p = 2$ that multiple tail lengths can have $n_t = 1$.
\end{remark}

Let us be more explicit.  Recall that $e^{(k)}{\rm d}_k^m$ was a basis vector for the linear module $\Delta_{(k)}(m)$ for each $m \in \supp k$.
To find an element of a cell module for $e^\bmu$ we have two options for each such $m$:
\begin{equation}
  d_k^{m_+} = 
  \vcenter{\hbox{
  \begin{tikzpicture}
    \draw[thick,fill=purple] (0,0.1) rectangle (.3,1.8);
    \draw[thick,fill=black] (0,1.7) rectangle (0.3,1.8);
    \draw[thick] (.3,.1) -- (0.95,.3) -- (0.95,1.6) -- (.3,1.8);
    \node at (.65,.9) {${\rm d}_k^m$};
    \foreach \i in {1,...,7} {
      \draw[very thick] (0,\i/4) -- (-.2,\i/4);
    }
    \foreach \i in {2,...,6} {
      \draw[very thick] (0.95,\i/4) -- (1.2,\i/4);
    }
    \draw[very thick] (-.2,0) -- (1.2,0);
  \end{tikzpicture}
}}
  \quad\quad
  \quad\quad
  d_k^{m_-} = 
  \vcenter{\hbox{
  \begin{tikzpicture}
    \draw[thick,fill=purple] (0,0.1) rectangle (.3,1.8);
    \draw[thick,fill=black] (0,1.7) rectangle (0.3,1.8);
    \draw[thick] (.3,.1) -- (0.95,.3) -- (0.95,1.6) -- (.3,1.8);
    \node at (.65,.9) {${\rm d}_k^m$};
    \foreach \i in {1,...,7} {
      \draw[very thick] (0,\i/4) -- (-.2,\i/4);
    }
    \foreach \i in {3,...,6} {
      \draw[very thick] (0.95,\i/4) -- (1.4,\i/4);
    }
    \draw[very thick] (-.2,0) -- (0.95,0) arc (-90:90:0.25);
  \end{tikzpicture}
}}
\end{equation}
Here, $m_+ = m+1$ and $m_- = m-1$ are the targets of the morphisms.
We will use $u_{k}^{m_\pm}$ to denote their vertical involutions.
As written, these are elements of $\Hom_{\TL_\bmu}(k+1,m_\pm)$.  That is, they are not necessarily monic but we are really considering their {\it images} in $\Delta_\bmu(m_\pm)$.
The two-dimensional cell modules are those for which some $m_+$ is the same as some $m'_-$ for $m,m'\in \supp n$.
All other cell modules are one-dimensional.

\begin{corollary}
  If $k$ is interior, $\TL_{(k,1)}$ is commutative.
\end{corollary}
\begin{proof}
  If $k$ is interior, then all the $m_+$ and $m_-$ are distinct and hence all the cell modules are one-dimensional.
  Hence by \cref{lem:linear_comm}, $\TL_\bmu$ is commutative.
  We will later show it is the direct sum of the two algebras $\TL_{(k+1)}\oplus \TL_{(k-1)}$, each of which are commutative as we know from \cref{sec:endomorphism}.
\end{proof}

Recall the goal is to enumerate $(\Lambda_\bmu)_0$.
To that end, let us evaluate the morphism $u_k^{m'_\pm}d_k^{m_\pm}$.
Firstly, it is clear that
\begin{equation}
  u_k^{m_+}d_k^{m_+} =
  \left({\rm u}_k^{m}e^{(k)}\otimes\id\right)\cdot\left(e^{(k)}{\rm d}_k^{m}\otimes\id\right) =
  {\rm u}_k^me^{(k)}{\rm d}_k^m\otimes\id
\end{equation}
and so $(u_k^{m_+}b_k^{m_+})_\id = \delta_{m,k}$ is inherited from the theory of $\TL_{(k)}$.

Next we examine $u_k^{m_-}d_k^{m_-}$:
\begin{equation}\label{eq:trace_a_clamp}
  \vcenter{\hbox{
  \begin{tikzpicture}
    \draw[thick,fill=purple] (0,0.1) rectangle (.3,1.8);
    \draw[thick,fill=black] (0,1.7) rectangle (0.3,1.8);
    \draw[thick] (.3,.1) -- (0.95,.3) -- (0.95,1.6) -- (.3,1.8);
    \draw[thick] (0,.1) -- (-0.65,.3) -- (-0.65,1.6) -- (0,1.8);
    \node at (.65,.9) {${\rm d}_k^m$};
    \node at (-.35,.9) {${\rm u}_k^m$};
    \foreach \i in {3,...,6} {
      \draw[very thick] (0.95,\i/4) -- (1.4,\i/4);
      \draw[very thick] (-0.65,\i/4) -- (-1.1,\i/4);
    }
    \draw[very thick] (-.65,0) -- (0.95,0) arc (-90:90:0.25);
    \draw[very thick] (-0.65,0.5) arc (90:270:0.25);
  \end{tikzpicture}
}}
\end{equation}
If we focus on the submorphism ${\rm u}_k^m e^{(k)}{\rm d}_k^m$ we can write
\begin{equation}\label{eq:clamping_an_idemp}
  {\rm u}_k^m e^{(k)}{\rm d}_k^m
  = \sum_{m' \in \supp k}\lambda_k^{m'} {\rm u}_k^m\overline{\rm d}_k^{m'}\overline{\rm u}_k^{m'}{\rm d}_k^{m}
  = \sum_{\substack{m' \in \supp k\\ m' \le m}}\lambda_k^{m'} {\rm u}_k^m\overline{\rm d}_k^{m'}\overline{\rm u}_k^{m'}{\rm d}_k^{m}.
\end{equation}
Clearly the value of the identity coefficient in \cref{eq:trace_a_clamp} is unchanged by summands factoring through $\underline{m-4}$ in \cref{eq:clamping_an_idemp}.
Indeed, composition by morphisms (on either side) cannot raise the through-degree, and tensoring with a strand raises it by at most one.

Now suppose that $m-2\not\in\supp k$.
Then almost all the summands of \cref{eq:clamping_an_idemp} vanish up to morphisms factoring through $\underline{m-2}$ and we are left with
\begin{equation}
  (u_k^{m_-}d_k^{m_-})_\id = 
  (\lambda_k^{m}{\rm u}_k^{m_-}\overline{\rm d}_k^{m_-}\overline{\rm u}_k^{m_-}{\rm d}_k^{m_-})_\id = 
  \frac{[m+1]}{[m]}(\lambda_k^m)^{-1}
\end{equation}
If $\ell \nmid m$ then this only doesn't vanish if $m = k$ and $\ell \nmid m+1$.
So let us
assume $\ell\mid m$, and write $k+1 = \pldigs{n_r, \ldots, n_0}$ with $m = k[S]$.
Then we see $m + 1 = \pldigs{n_r, \widetilde{n}_{r-1}, \ldots, \widetilde{n}_1,1}$.
Hence, if $0\not \in S$, we see that $n_0 = 1$ and so $m - k[S\cup\{0\}] = 2$, a contradiction, and so $0\in S$.
Let us assume that $\{0,1,\ldots, s\} \in S$ and hence $t_k > 0$.

Now, $\left(\lambda_k^{k[S]}\right)^{-1}$ has factor
$$
\frac{[\pldigs{n_r,\ldots, n_{s+1}, 0,\ldots, 0}]}
{[\pldigs{n_r,\ldots, n_{s+1}, -n_s,\ldots, -n_0}]}
$$
and $[m+1] = [\pldigs{n_r,\ldots,n_{s+1}, -n_s,\ldots, -n_0}]$ so 
$\frac{[m+1]}{[m]}\left(\lambda_k^{k[S]}\right)^{-1}$ has factor
$$
\frac{[\pldigs{n_r,\ldots, n_{s+1}, 0,\ldots, 0}]}
{[\pldigs{n_r,\ldots, n_{s+1}, -n_s,\ldots, -n_0-1}]}.
$$
Which descends to zero unless $n_0 = \ell -1$ and $n_1 = \ldots = n_s = p-1$, i.e. $s < t_k$.  If this occurs then the factor is
$$
\frac{[\pldigs{n_r,\ldots, n_{s+1}, 0,\ldots, 0}]}
{[\pldigs{n_r,\ldots, n_{s+1}-1, 0,\ldots, 0}]}
=
\frac{[n_r,\ldots, n_{s+1}]_p}
{[n_r,\ldots, n_{s+1}-1]_p}
$$
Now, since $s+1\not\in S$, $n_{s+1} \neq 0$ and if $n_{s+1} = 1$ then $m-2\in \supp k$.
Hence this factor does not vanish under specialisation.
It is clear that any other factors must vanish, hence we see that it must be the only factor.
That is to say $0\le s < t_{n}$, and each such $s$ describes a single $m = k[\{0,\ldots, s\}]$ for which the inner product does not vanish.

We now study the case where $m-2\in \supp{k}$.

\begin{lemma}\label{lem:id_of_trace_of_clamp}
  Suppose that $m, m-2 \in \supp k$.
  Then
  \begin{equation}
  \left(\tau\left({\rm u}_k^me^{(k)}{\rm d}_k^m\right)\right)_\id =
    \left(\lambda_k^{m-1}\right)^{-1}\left([p^{(t_k+1)} + 1] - [p^{(t_k+1)} - 1]\right)
  \end{equation}
  where $t_k$ is such that $m-1 \in \supp \mother{k}^{t_k+1}$.
\end{lemma}
\begin{proof}
Suppose that $S_1 = S_0 \cup\{t_k\}$ and $S_2 = S_0 \cup\{t_k-1,\ldots, 0\}$ are such that $m = k[S_2] = k[S_1] + 2$.
Then, if $k' = \mother{k}^{t_k+1}$, $m-1 = k'[S_0]$.
We calculate
\begin{equation}\label{eq:lobsided_clamp}
  \vcenter{\hbox{
  \begin{tikzpicture}
    \draw[thick] (0,0) -- (0.95,.2) -- (0.95,1.6) -- (0,1.8) -- cycle;
    \draw[thick] (0,0) -- (-0.95,.2) -- (-0.95,1.6) -- (0,1.8);
    \fill[thick] (-0.95,1.5) -- (-0.95,1.6) -- (0,1.8) -- (0,1.7) -- cycle;
    \node at (.5,.9) {${\rm d}_k^{k[S_2]}$};
    \node at (-.45,.9) {$\overline{\rm u}_k^{k[S_1]}$};
    \foreach \i in {2,...,7} {
      \draw[very thick] (-0.95,\i/5) -- (-1.3,\i/5);
    }
    \foreach \i in {1,...,7} {
      \draw[very thick] (0.95,\i/5+0.125) -- (1.3,\i/5+0.125);
    }
  \end{tikzpicture}
  }}
  =
  \vcenter{\hbox{
  \begin{tikzpicture}
    \draw[thick] (0,0.6) -- (0.95,.8) -- (0.95,1.6) -- (0,1.8) -- cycle;
    \draw[thick] (0,0.6) -- (-0.95,.8) -- (-0.95,1.6) -- (0,1.8);
    \fill[thick] (-0.95,1.5) -- (-0.95,1.6) -- (0,1.8) -- (0,1.7) -- cycle;
    \node at (.5,1.2) {${\rm d}_{k'}^{k'[S_0]}$};
    \node at (-.45,1.2) {$\overline{\rm u}_{k'}^{k'[S_0]}$};
    \draw[line width=3pt] (0.95, 1.2) -- (2.1, 1.2);
    \node at (1.5, 1.5) {\footnotesize $k'[S_0]$};
    \draw[line width=3pt] (-0.95, 1.3) -- (-2.4, 1.3);
    \node at (-1.7, 1.6) {\footnotesize $k'[S_0]-p^{(t_k)}$};

    \draw[line width=2pt] (-0.95, 1.1) arc (90:180:.4) arc (180:270:.4) -- (0,.3);
    \draw[thick] (0,.32) -- (2.1,.32);
    \draw[line width=1.5pt] (0,.29) arc (90:0:.3) arc (0:-90:.3) -- (-2.4,-.31);
    \node at (0.7,-.34) {\footnotesize $p^{(t_k)}-1$};
  \end{tikzpicture}
  }}
  =
  \vcenter{\hbox{
  \begin{tikzpicture}
    \draw[thick] (0,0.6) -- (0.95,.8) -- (0.95,1.6) -- (0,1.8) -- cycle;
    \draw[thick] (0,0.6) -- (-0.95,.8) -- (-0.95,1.6) -- (0,1.8);
    \fill[thick] (-0.95,1.5) -- (-0.95,1.6) -- (0,1.8) -- (0,1.7) -- cycle;
    \node at (.5,1.2) {${\rm d}_{k'}^{k'[S_0]}$};
    \node at (-.45,1.2) {$\overline{\rm u}_{k'}^{k'[S_0]}$};
    \draw[line width=3pt] (0.95, 1.2) -- (1.6, 1.2);
    \draw[line width=3pt] (-0.95, 1.3) -- (-1.6, 1.3);
    \draw[line width=1pt] (-0.95, 1.1) arc (90:180:.4) arc (180:270:.4) -- (0,.3)-- (1.6,.32);
  \end{tikzpicture}
  }}
\end{equation}
and hence
\begin{align*}
  \vcenter{\hbox{
  \begin{tikzpicture}
    \draw[thick,fill=purple] (0,0.1) rectangle (.3,1.8);
    \draw[thick,fill=black] (0,1.7) rectangle (0.3,1.8);
    \draw[thick] (.3,.1) -- (1.15,.3) -- (1.15,1.6) -- (.3,1.8);
    \draw[thick] (0,.1) -- (-0.85,.3) -- (-0.85,1.6) -- (0,1.8);
    \node at (.75,.9) {\small ${\rm d}_k^{k[S_2]}$};
    \node at (-.4,.9) {\small${\rm u}_k^{k[S_2]}$};
    \foreach \i in {2,...,7} {
      \draw[very thick] (1.15,\i/5+0.07) -- (1.5,\i/5+0.07);
      \draw[very thick] (-0.85,\i/5+0.07) -- (-1.2,\i/5+0.07);
    }
  \end{tikzpicture}
  }}&=\lambda_{k}^{k[S_2]}\vcenter{\hbox{
  \begin{tikzpicture}
    \draw[thick] (0.95,0) -- (0,.2) -- (0,1.6) -- (0.95,1.8) -- cycle;
    \draw[thick] (-0.95,0) -- (0,.2) -- (0,1.6) -- (-0.95,1.8) -- cycle;
    \fill[thick] (0,1.5) -- (0,1.6) -- (0.95,1.8) -- (0.95,1.7) -- cycle;
    \fill[thick] (0,1.5) -- (0,1.6) -- (-0.95,1.8) -- (-0.95,1.7) -- cycle;
    \node at (.5,.9) {\small$\overline{\rm u}_k^{k[S_2]}$};
    \node at (-.45,.9) {\small$\overline{\rm d}_k^{k[S_2]}$};
    \node at (-1.35,.9) {\small${\rm u}_k^{k[S_2]}$};
    \node at (1.35,.9) {\small${\rm d}_k^{k[S_2]}$};
    \draw[thick] (-1.9,0.2) -- (-0.95,0) -- (-0.95,1.8) -- (-1.9,1.6) -- cycle;
    \draw[thick] (1.9,0.2) -- (0.95,0) -- (0.95,1.8) -- (1.9,1.6) -- cycle;
    \foreach \i in {2,...,7} {
      \draw[very thick] (-1.9,\i/5) -- (-2.15,\i/5);
      \draw[very thick] (1.9,\i/5) -- (2.15,\i/5);
    }
  \end{tikzpicture}
  }}+\lambda_{k}^{k[S_1]}\vcenter{\hbox{
  \begin{tikzpicture}
    \draw[thick] (0.95,0) -- (0,.2) -- (0,1.6) -- (0.95,1.8) -- cycle;
    \draw[thick] (-0.95,0) -- (0,.2) -- (0,1.6) -- (-0.95,1.8) -- cycle;
    \fill[thick] (0,1.5) -- (0,1.6) -- (0.95,1.8) -- (0.95,1.7) -- cycle;
    \fill[thick] (0,1.5) -- (0,1.6) -- (-0.95,1.8) -- (-0.95,1.7) -- cycle;
    \node at (.5,.9) {\small$\overline{\rm u}_k^{k[S_1]}$};
    \node at (-.45,.9) {\small$\overline{\rm d}_k^{k[S_1]}$};
    \node at (-1.35,.9) {\small${\rm u}_k^{k[S_2]}$};
    \node at (1.35,.9) {\small${\rm d}_k^{k[S_2]}$};
    \draw[thick] (-1.9,0.2) -- (-0.95,0) -- (-0.95,1.8) -- (-1.9,1.6) -- cycle;
    \draw[thick] (1.9,0.2) -- (0.95,0) -- (0.95,1.8) -- (1.9,1.6) -- cycle;
    \foreach \i in {2,...,7} {
      \draw[very thick] (-1.9,\i/5) -- (-2.15,\i/5);
      \draw[very thick] (1.9,\i/5) -- (2.15,\i/5);
    }
  \end{tikzpicture}
  }}\\
  &=
  \left(\lambda_k^{m}\right)^{-1}
  \vcenter{\hbox{
  \begin{tikzpicture}
    \draw[thick] (-.5, 0) rectangle (.5,1.4);
    \node at (0,0.7) {$\JW_{m}$};
    \foreach \i in {1,...,6} {
      \draw[very thick] (-.5,\i/5) -- (-0.7,\i/5);
      \draw[very thick] (.5,\i/5) -- (0.7,\i/5);
    }
  \end{tikzpicture}
  }}
  + \lambda_k^{m-2}
  \vcenter{\hbox{
  \begin{tikzpicture}
    \begin{scope}[shift={(1.5,0)}]
      \draw[thick] (0,0.6) -- (0.95,.8) -- (0.95,1.6) -- (0,1.8) -- cycle;
      \draw[thick] (0,0.6) -- (-0.95,.8) -- (-0.95,1.6) -- (0,1.8);
      \fill[thick] (-0.95,1.5) -- (-0.95,1.6) -- (0,1.8) -- (0,1.7) -- cycle;
      \node at (.5,1.2) {${\rm d}_{k'}^{k[S_0]}$};
      \node at (-.45,1.2) {$\overline{\rm u}_{k'}^{k[S_0]}$};
      \draw[line width=3pt] (0.95, 1.2) -- (1.4, 1.2);
      \draw[line width=3pt] (-0.95, 1.3) -- (-1.5, 1.3);
      \draw[line width=1pt] (-0.95, 1.1) arc (90:180:.4) arc (180:270:.4) -- (0,.3)-- (1.4,.32);
    \end{scope}
    \begin{scope}[shift={(-1.5,0)}]
      \draw[thick] (0,0.6) -- (-0.95,.8) -- (-0.95,1.6) -- (0,1.8) -- cycle;
      \draw[thick] (0,0.6) -- (0.95,.8) -- (0.95,1.6) -- (0,1.8);
      \fill[thick] (0.95,1.5) -- (0.95,1.6) -- (0,1.8) -- (0,1.7) -- cycle;
      \node at (-.5,1.2) {${\rm d}_{k'}^{k[S_0]}$};
      \node at (.45,1.2) {$\overline{\rm u}_{k'}^{k[S_0]}$};
      \draw[line width=3pt] (-0.95, 1.2) -- (-1.4, 1.2);
      \draw[line width=3pt] (0.95, 1.3) -- (1.5, 1.3);
      \draw[line width=1pt] (-1.4,.32) -- (0.95,.3) arc (-90:90:.4);
    \end{scope}
  \end{tikzpicture}
  }}\\
  &=
  \left(\lambda_k^{m}\right)^{-1}
  \vcenter{\hbox{
  \begin{tikzpicture}
    \draw[thick] (-.5, 0) rectangle (.5,1.4);
    \node at (0,0.7) {$\JW_{m}$};
    \foreach \i in {1,...,6} {
      \draw[very thick] (-.5,\i/5) -- (-0.7,\i/5);
      \draw[very thick] (.5,\i/5) -- (0.7,\i/5);
    }
  \end{tikzpicture}
  }}
  + \lambda_k^{m-2}\left(\lambda_{k'}^{m-1}\right)^{-2}
  \vcenter{\hbox{
  \begin{tikzpicture}
    \begin{scope}[shift={(1.2,0)}]
      \draw[thick] (-.6, 0) rectangle (.6,1.4);
      \node at (0,0.7) {$\JW_{m-1}$};
      \foreach \i in {1,...,6} {
        \draw[very thick] (.6,\i/5) -- (0.8,\i/5);
      }
      \draw[very thick] (-.6,.2) arc (90:270:.2) -- (0.8,-.2);
    \end{scope}
    \begin{scope}[shift={(-0.8,0)}]
      \draw[thick] (-.6, 0) rectangle (.6,1.4);
      \node at (0,0.7) {$\JW_{m-1}$};
      \foreach \i in {1,...,6} {
        \draw[very thick] (-.6,\i/5) -- (-0.8,\i/5);
      }
      \foreach \i in {2,...,6} {
        \draw[very thick] (.6,\i/5) -- (1.4,\i/5);
      }
      \draw[very thick] (-0.8,-.2) -- (.6,-.2) arc (270:360:.2) arc (0:90:.2);
    \end{scope}
  \end{tikzpicture}
  }}.
\end{align*}
Thus we see that
\begin{equation}\label{eq:clampything}
  \left(
  \vcenter{\hbox{
  \begin{tikzpicture}
    \draw[thick,fill=purple] (0,0.1) rectangle (.3,1.8);
    \draw[thick,fill=black] (0,1.7) rectangle (0.3,1.8);
    \draw[thick] (.3,.1) -- (1.25,.3) -- (1.25,1.6) -- (.3,1.8);
    \draw[thick] (0,.1) -- (-0.95,.3) -- (-0.95,1.6) -- (0,1.8);
    \node at (.8,.9) {\small ${\rm d}_k^{k[S_2]}$};
    \node at (-.5,.9) {\small${\rm u}_k^{k[S_2]}$};
    \foreach \i in {3,...,7} {
      \draw[very thick] (1.25,\i/5) -- (1.7,\i/5);
      \draw[very thick] (-0.95,\i/5) -- (-1.4,\i/5);
    }
    \draw[very thick] (-.95, .4) arc (90:270:.2)  -- (1.25,0) arc (-90:90:.2);
  \end{tikzpicture}
  }}
  \right)_{\rm id}
  = \frac{[m+1]}{[m]}\left(\lambda_k^m\right)^{-1} + \lambda_k^{m-2}\left(\lambda_{k'}^{m-1}\right)^{-2}
\end{equation}
Now, recall that
\begin{equation}
  \vcenter{\hbox{
  \begin{tikzpicture}
    \draw[thick] (0,.2) -- (0.95,.4) -- (0.95,1.6) -- (0,1.8) -- cycle;
    \fill (0.95,1.6)--(0,1.8) -- (0,1.7) -- (0.95,1.5)--cycle;
    \node at (.5,.9) {\small $\overline{\rm d}_k^{k[S_2]}$};
  \end{tikzpicture}
  }}
  =
  \vcenter{\hbox{
  \begin{tikzpicture}
    \draw[thick] (0,0.6) -- (0.95,.8) -- (0.95,1.6) -- (0,1.8) -- cycle;
    \fill[thick] (0.95,1.5) -- (0.95,1.6) -- (0,1.8) -- (0,1.7) -- cycle;
    \node at (.5,1.2) {$\overline{\rm d}_{k'}^{k'[S_0]}$};
    \draw[line width=3pt] (0.95, 1.2) -- (1.2, 1.2);
    \draw (1.2,0) rectangle (3.4, 1.6);
    \node at (2.3, 0.8) {$\JW_{m+p^{(t_k)}-1}$};

    \draw[line width=3pt] (0, 0.3) -- (1.2, 0.3);
    \node at (0.6, -0.03) {\footnotesize $p^{(t_k)}$};

    \draw[line width=2pt] (0, -.4) -- (3.4, -.4) arc(-90:90:.4);
    \node at (2.2, -0.8) {\footnotesize $p^{(t_k)}-1$};

    \draw[line width=3pt] (3.4, 1.1) -- (4.0, 1.1);
  \end{tikzpicture}
  }}
\end{equation}
and so $\lambda_k^m = \lambda_{k'}^{m-1}\frac{[m+1]}{[m+p^{(t_k+1)}]}$.
Similarly, 
\begin{equation}
  \vcenter{\hbox{
  \begin{tikzpicture}
    \draw[thick] (0,.2) -- (0.95,.4) -- (0.95,1.6) -- (0,1.8) -- cycle;
    \fill (0.95,1.6)--(0,1.8) -- (0,1.7) -- (0.95,1.5)--cycle;
    \node at (.5,.9) {\small $\overline{\rm d}_k^{k[S_1]}$};
  \end{tikzpicture}
  }}
  =
  \vcenter{\hbox{
  \begin{tikzpicture}
    \draw[thick] (0,0.6) -- (0.95,.8) -- (0.95,1.6) -- (0,1.8) -- cycle;
    \fill[thick] (0.95,1.5) -- (0.95,1.6) -- (0,1.8) -- (0,1.7) -- cycle;
    \node at (.5,1.2) {$\overline{\rm d}_{k'}^{k'[S_0]}$};
    \draw[line width=3pt] (0.95, 1.2) -- (1.7, 1.2);
    \draw (1.7,-.8) rectangle (3.0, 1.6);
    \node at (2.35, 0.4) {$\JW_{m-2}$};

    \draw[line width=3pt] (0, 0.3) -- (0.95, 0.3) arc (-90:90:0.33);
    \node at (0.6, -0.03) {\footnotesize $p^{(t_k)}$};

    \draw[line width=2pt] (0, -.4) -- (1.7, -.4);
    \node at (0.6, -0.8) {\footnotesize $p^{(t_k)}-1$};

    \draw[line width=3pt] (3.0, 0.4) -- (3.3, 0.4);
  \end{tikzpicture}
  }}
\end{equation}
giving $\lambda_{k}^{m-2} = \lambda_{k'}^{m-1}\frac{[m-p^{(t_k+1)}]}{[m]}$ so we can simplify \cref{eq:clampything} to
\begin{equation}
\left(\lambda_{k'}^{m-1}\right)^{-1}\left(\frac{[m-p^{(t_k+1)}] + [m + p^{(t_k+1)}]}{[m]}\right) = 
\left(\lambda_{k'}^{m-1}\right)^{-1}\left([p^{(t_k+1)}+1] - [p^{(t_k+1)}-1]\right)
\end{equation}
as desired.
\end{proof}

As an application, consider the cell module $\Delta_\bmu(m-1)$ of $\TL_{\bmu}$ spanned by $\{d_k^{(m-2)_+}, d_k^{m_-}\}$.
Then we can use \cref{lem:id_of_trace_of_clamp} to see that $\langle d_k^{m_-}, d_k^{m_-}\rangle$ vanishes unless $k' = m-1$ or $\ell = 2$.
We already know that $\langle d_k^{(m-2)_+}, d_k^{(m-2)_+}\rangle$ is given by $(\lambda_{k}^{m-2})^{-1}$ which always vanishes.
Finally we can use \cref{eq:lobsided_clamp} to show that 
\begin{equation}
  {\rm u}_k^{m-2} e^{(k)}{\rm d}_k^m
  = \sum_{m' \in \supp k}\lambda_k^{m'} {\rm u}_k^m\overline{\rm d}_k^{m'}\overline{\rm u}_k^{m'}{\rm d}_k^{m}
  = \sum_{\substack{m' \in \supp k\\ m' \le m-2}}\lambda_k^{m'} {\rm u}_k^m\overline{\rm d}_k^{m'}\overline{\rm u}_k^{m'}{\rm d}_k^{m}.
\end{equation}
which taken mod $<m-2$ gives $\lambda_k^{m-2}{\rm u}_k^m\overline{\rm d}_k^{m-2}\overline{\rm u}_k^{m-2}{\rm d}_k^{m}$.
It is then clear that after adding the rest of the diagram, we get $\langle d_k^{(m-2)_+}, d_k^{m_-}\rangle = \lambda_{k'}^{m-1} = (\lambda_{k'}^{m-1})^{-1}$.

Thus two dimensional modules are completely degenerate unless $m-1 = k'$ in which case they are simple.
We have deduced:

\begin{theorem}\label{prop:main_two_part}
  Suppose that $\bmu = (k,1)\vdash n = \pldigs{n_r, \ldots, n_0}$.
  Let $T_n = \{0\le t \le t_k \;:\; n_t = 1\}$ and $S = \{0\le t < t_k \;:\; n_t >1\}$.

  Let
  \begin{align*}
    \Lambda_{1\rm{a}} &= 
      \begin{cases}
        \left\{k+1, k-1\right\} & \ell \nmid k \text{ and } \ell \nmid k+1\\
        \left\{k+1\right\} & \text{else}
      \end{cases}\\
    \Lambda_{1\rm{b}} &=
      \left\{\pldigs{n_r, \ldots, n_{t+1}-1, 0,\ldots, 0}-1 \;:\; t \in S\right\}\\
    \Lambda_2 &=
      \left\{\pldigs{n_r, \ldots, n_{t+1},0,\ldots,0}-1 \;:\; t \in T\right\}
  \end{align*}
  Then $\Lambda_0 = \Lambda_{1\rm a}\cup \Lambda_{1\rm b} \cup \Lambda_2$ and $\Lambda_{1\rm a}$ and $\Lambda_{1 \rm b}$ index the irreducible simple modules of dimension 1 and $\Lambda_2$ that of dimension 2.
\end{theorem}
This should be compared to its counterpart, \cite[Proposition 4.7]{sutton_tubbenhauer_wedrich_zhu_2021}.

\begin{example}
  We present the case of $\ell = 5$ and $p = 3$.  The below diagram differs from previous ones in that it is not a Bratteli diagram.  Each row represents the cell modules of $\TL_{(k,1)}$.  Grey dots are cell modules of $\TL_{k+1}$ that are not cell modules of $\TL_{(k,1)}$.  Orange dots are one dimensional cell modules and red dots two dimensional.  Those dots in a black outline are non-degenerate.
  \begin{center}
    \includegraphics[width=0.7\textwidth]{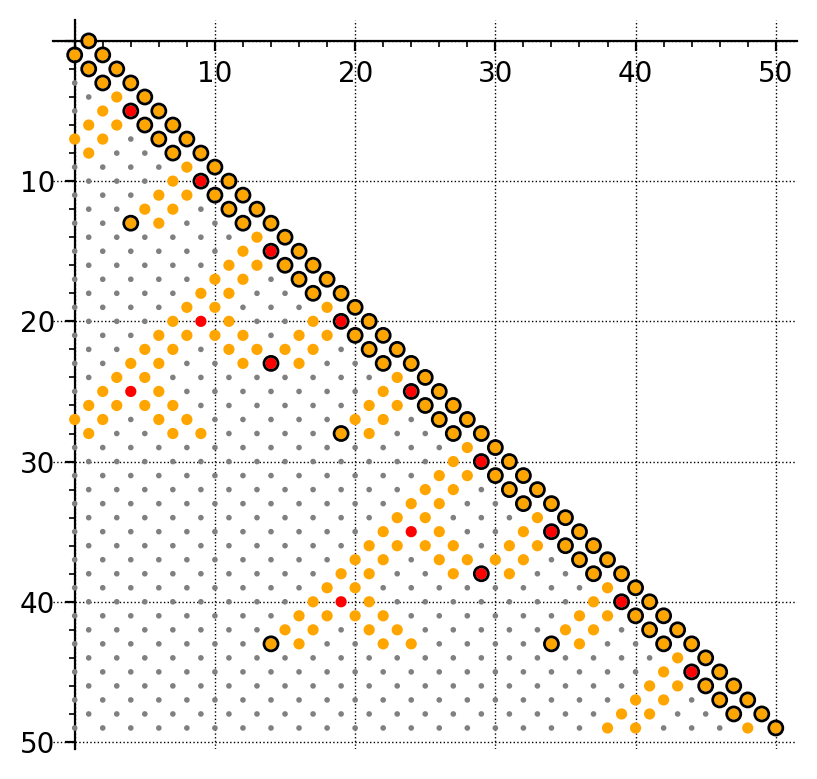}
  \end{center}
\end{example}
\begin{example}
    On the other hand, if $\ell = p = 2$ we obtain the following diagram.  \begin{center}
    \includegraphics[width=0.7\textwidth]{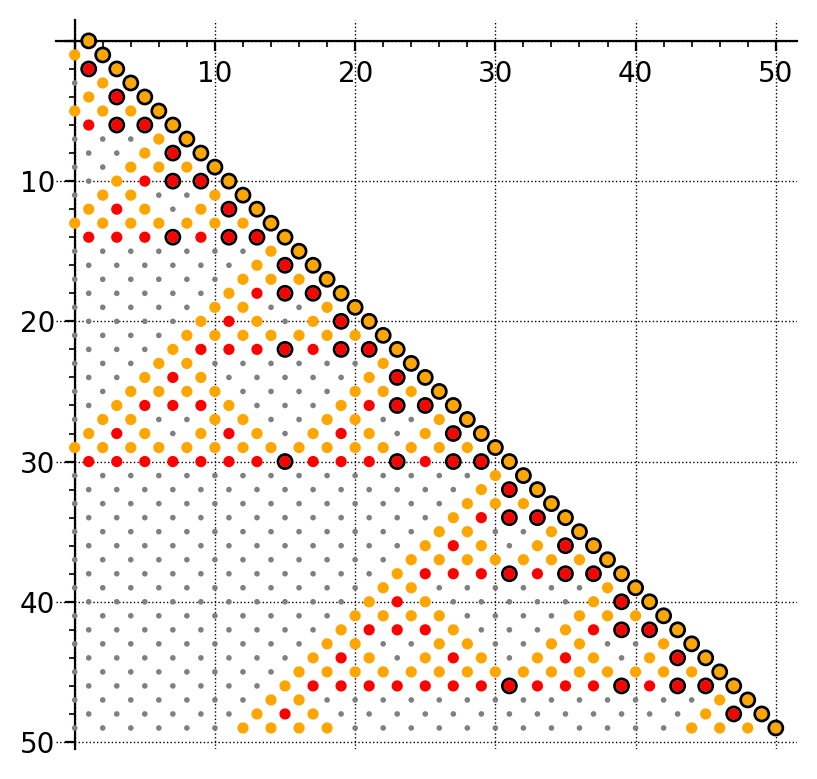}
  \end{center}
  Note that all irreducible modules are two dimensional (apart from the trivial).
  If $k$ is odd then $\TL_{\bmu}$ has a single irreducible module and all cell modules are one dimensional.
  Indeed, when $\ell = 2$, then in \cref{prop:main_two_part}, $\Lambda_{1\rm a}$ is always exactly $\{k + 1\}$.
  Further $T$ is non-empty iff $k$ is even.
  If $p = 2$ also, then $S = \emptyset$, showing the above claim.
\end{example}

\subsection{Two Part Seam Algebras as Inductions}
Recall that $P_{\mathbf{n}}(m)$ has a $\Delta_{\mathbf{n}}$-filtration.
Moreover, by applying the truncation functor, we see that the it has the same factors as the projective cover of $L_{\mathbf{m}}(m)$.
That is to say $\Delta_{\mathbf{m}}(r)$ appears in a filtration of $P_{\mathbf{m}}(m)$ iff $\Delta_{\mathbf{n}}(r)$ appears in a filtration of $P_{\mathbf{n}}(m)$.
This holds as $r \le m$ for all $\Delta_\mathbf{n}(r)$ appearing in $P_\mathbf{n}(m)$.

Since adding a strand gives rise inducing from $n$ to $n+1$, we see that
\begin{equation}
  \TL_{\bmu}=
  e^\bmu \cdot \TL_\mathbf{k+1} \cdot e^\bmu
  \simeq \End_{\TL_\mathbf{k+1}}(\TL_\mathbf{k+1}\cdot e^\bmu)
  \simeq \End_{\TL_\mathbf{k+1}}(P_\mathbf{k}(k)\ind)
\end{equation}
From the above observation and the known composition factors of $P_\mathbf{k}(k)$, we see that if $k$ has no tail and $n_0 \neq 1$, 
$P_\mathbf{k}(k)\ind \simeq P_\mathbf{k+1}(k+1) \oplus P_\mathbf{k+1}(k-1)$.
Indeed, we know it must be the direct sum of projective indecomposables and has a $\Delta_{\mathbf{k+1}}$ filtration.  Further, inducing $\Delta_{\mathbf k}(m)$ gives a module with a $\Delta$-filtration with factors $\Delta_\mathbf{k+1}(m+1)$ and $\Delta_\mathbf{k+1}(m-1)$.
In this case, these lie in different blocks and so
\begin{align}
  \TL_{\bmu}\simeq
  \End_{\TL_\mathbf{k+1}}(P_\mathbf{k}(k)\ind)
  &\simeq
  \End_{\TL_\mathbf{k+1}}(P_\mathbf{k+1}(k-1))
  \oplus
  \End_{\TL_\mathbf{k+1}}(P_\mathbf{k+1}(k+1))\\
  &\simeq
  \End_{\TL_\mathbf{k-1}}(P_\mathbf{k-1}(k-1))
  \oplus
  \TL_{(k+1)}\nonumber\\
  &\simeq
  \TL_{(k-1)}
  \oplus
  \TL_{(k+1)}.\nonumber
\end{align}
This reconfirms that $\TL_{\bmu}$ is commutative and lays bare the representation theory in terms of algebras we understand from \cref{sec:endomorphism}.

The second equivalence deserves some explanation.
While the second summand is plain, the first asserts that $\End_{\TL_\mathbf{k+1}}(P_\mathbf{k+1}(k-1)) \simeq \End_{\TL_\mathbf{k-1}}(P_\mathbf{k-1}(k-1))$.
This holds because the modules $P_{\mathbf{k+1}}(k-1)$ and $P_\mathbf{k-1}(k-1)$ each have standard filtrations by modules indexed by the same cell indices.
Each of these modules have simples indexed by the same indices and the only additional factors are those of the trivial module which is distinct from the head of $P_\mathbf{k+1}(k-1)$.

A similar character calculation shows that if $k \equiv_\ell - 1$ then $P_\mathbf{k}(k)\ind \equiv P_\mathbf{k+1}(k+1)$.
Indeed, certainly $P_\mathbf{k+1}(k+1)$ is a summand of $P_\mathbf{k}(k)\ind$ and by looking at the values of $\supp({\ell})$ it must have the same standard factors.  Thus
\begin{align}
  \TL_{\bmu}\simeq
  \End_{\TL_{k+1}}(P_k(k)\ind)
  &\simeq
  \End_{\TL_{k+1}}(P_{k+1}(k+1))\\
  &\simeq
  \TL_{(k+1)}.\nonumber
\end{align}

Indeed, character arguments as above are used in~\cite{sutton_tubbenhauer_wedrich_zhu_2021} to deduce similar results to \cref{prop:main_two_part}.

%% file: two-part.tex
\section{Two Part Eve Compositions}\label{sec:two-part}

Let $\bmu = (\mu_1,\mu_2) \vdash n$ be Eve and without loss of generality suppose $\mu_1 \ge \mu_2$.  This is so because the (Eve) Jones-Wenzl projectors have up-down symmetry.
\begin{center}
  \begin{tikzpicture}
    \foreach \x in {0, 1.5} {
      \begin{scope}[shift={(\x,0)}]
        \foreach \i in {0,...,14} {
          \draw[very thick] (-.2,.2+ \i/5) -- (.5,.2 + \i/5);
        }
        \draw[thick,fill=purple] (0,0.1) rectangle (0.3,1.85);
        \draw[thick,fill=purple] (0,1.95) rectangle (0.3,3.1);
      \end{scope}
    }
    \draw[thick] (.5,0.1) rectangle (1.3, 3.1);
    \draw [decorate,decoration={brace,amplitude=5pt},xshift=-1pt,yshift=0pt] (-0.3,0.1) -- (-0.3,1.9) node [black,midway,xshift=-10pt] {\footnotesize $\mu_1$};
    \draw [decorate,decoration={brace,amplitude=5pt},xshift=-1pt,yshift=0pt] (-0.3,1.9) -- (-0.3,3.0) node [black,midway,xshift=-10pt] {\footnotesize $\mu_2$};
  \end{tikzpicture}
\end{center}
\begin{remark}
  When we relax the condition that $\bmu$ is Eve we might find that we no longer have a $(\mu_1, \mu_2) \longleftrightarrow (\mu_2, \mu_1)$ symmetry.
  Indeed, the $\JW_{\mu_i}$ are no longer symmetric so the two cases are not related by the vertical symmetry.
  This is discussed in \cref{sec:further}.
\end{remark}

\begin{theorem}
  Let $\bmu = (\mu_1,\mu_2) \vdash n$ be an Eve composition.
  Write $\mu_1 = a_1p^{(i_1)}-1 \ge \mu_2 = a_2p^{(i_2)}-1$.
  \begin{enumerate}[(i)]
    \item The algebra $\TL_\bmu$ has $k$-dimension equal to $\min\{\mu_1,\mu_2\}$.
    \item The set $\Lambda_\bmu = E_{\mu_1,\mu_2} \subseteq \Lambda_\mathbf n$ and each cell module is one dimensional.
    \item Further, $(\Lambda_\bmu)_0 = \{m \in E_{\mu_1,\mu_2}\;:\; (n+m)/2 + 1 \triangleright (n-m)/2\}$.
        We can also write this as 
    \begin{equation}
      (\Lambda_\bmu)_0 = \left\{n - 2\sum_{i=0}^{i_2}k_i p^{(i)} \;:\;
  k_{i_2} \in X\text{ and }
  k_{i} \le \left\lfloor\frac{\ell \wedge p -1}{2}\right\rfloor
  \text{ for } 0 \le i < i_2
      \right\},
    \end{equation}
    where if $i_1 > i_2$ the set $X = \{0,\ldots,a_2 - 1\}$ and if
    $i_1 = i_2$ with $(a_1 + a_2-1) p^{(i_1)} = c p^{(i_1)} + d p^{(i_1+1)}$,
    \begin{equation}
      X =
        \{0,\ldots, \lfloor c/2\rfloor\} \cup \{c+1,\ldots,\min\{a_2, \lfloor (p\wedge \ell +c)/2\rfloor\}\}.
    \end{equation}
  \end{enumerate}
\end{theorem}
Interestingly we see that the composition $\bmu$ controls which cell modules exist, but not whether or not they are degenerate --- that is controlled entirely by $n$.
Again this should be compared to~\cite[Proposition 4.11]{sutton_tubbenhauer_wedrich_zhu_2021}.

\begin{proof}

Recall from \cref{eg:rings_tl_mu} part (iii) that over characteristic zero (or $(\ell, p)$ where $\ell > \mu_1, \mu_2$)
\begin{equation}\label{eq:two_part_iso}
  \TL_\bmu \simeq k[x] \Big/ \left(\prod_{r = 1}^{\mu_2} \left(x - \frac{[r][n-r+1]}{[\mu_1][\mu_2]}\right)\right).
\end{equation}
Such an algebra has $\mu_2$ cell modules indexed by each $1\le r \le \mu_2$.
The involution is the identity and each cell module is one dimensional.

It is clear that the cell modules of $\TL_\bmu$ are indexed by
\begin{equation}
  \Lambda_\bmu = \{n, n-2, \ldots, n-2\mu_2\} = E_{\mu_1,\mu_2}.
\end{equation}
By comparing \cref{eq:two_part_recurse} we see that the generator $U_1$ acts on the module indexed by $n-2k$ as the scalar $[k][n-k+1]/[\mu_1][\mu_2]$.
Thus in fact the isomorphism in \cref{eq:two_part_iso} is an isomorphism preserving the given cell data.

We have thus deduced that there are $\mu_2$ cell modules, each of which is one dimensional.
Since the dimension of a cellular algebra is given by the sum of the squares of the dimensions of its cell modules, we have shown both (i) and (ii).

From \cref{eq:useful_gauss}, we can see that the determinant of the Gram matrix for the cell module $\Delta_\bmu(n-2k)$ is
\begin{align}\label{eq:two_part_gram}
  \frac{\Theta(\mu_1,\mu_2,n-2k)}{[n-2k+1]} &=
  \frac{[n-k+1]}{[n-2k+1]}\gaussianquant{n-k}{k}\Big/\gaussianquant{\mu_1}{k}\gaussianquant{\mu_2}{k}\\
                                            &=\gaussianquant{n-k+1}{k}\Big/\gaussianquant{\mu_1}{k}\gaussianquant{\mu_2}{k}\nonumber.
\end{align}
The determinant is well defined over mixed characteristic since it is a polynomial in the elements of the Gram matrix, which are themselves defined over mixed characteristic.
Each cell module is one-dimensional and so iff this determinant vanishes, the entire form is degenerate and $n-2k \not\in (\Lambda_{\bmu})_0$.

However, an equivalent condition for $\mu_i$ to be Eve is that all quantum binomials $\gaussianquant{\mu_i}{k}$ are non-zero in the ring~\cite[Section 10.3]{spencer_2020}.
Hence the denominator in \cref{eq:two_part_gram} is non-zero and so the determinant vanishes precisely when $\gaussianquant{n-k+1}{k} = 0$.

Recall the notation that $n \triangleright r$ if $n_i \ge r_i$ for all $i$, where $n = \sum_i n_i p^{(i)}$ and $r = \sum_i r_ip^{(i)}$ are the $(\ell,p)$-adic expansions.
Then the quantum binomial $\gaussianquant{n}{r}$ is non-zero iff $n \triangleright r$.
Hence the determinant vanishes exactly when $n +1 - k \triangleright k$.

If $\mu_1 = a_1 p^{(i_1)}$ and $\mu_2 = a_2 p^{(i_2)}$ for $i_1\neq i_2$, then $n+1 = [a_1, 0,\ldots,0, a_2-1,p-1,\ldots,\ell-1]_{\ell,p}$.
Since $k \le \mu_2$, when subtracting $k$ from $n+1$, there are no carries.
Hence the set of $k$ for which the cell module indexed by $n-2k$ is non-degenerate is
\begin{equation}
\left\{\sum_{i=0}^{i_2} k_i p^{(i)}\;:\;
  k_{i_2} \le \left\lfloor\frac{a_2-1}{2}\right\rfloor\text{ and }
  k_{i} \le \left\lfloor\frac{\ell \wedge p -1}{2}\right\rfloor
\right\}
\end{equation}
On the other hand let $i_1 = i_2$ and $(a_1 + a_2-1) p^{(i_1)} = c p^{(i_1)} + d p^{(i_1+1)}$ for $d \in \{0,1\}$ and $c < a_2$.
We can thus write $n+1 = \pldigs{d,c,p-1,\ldots, \ell-1}$.
Then again there are no carries when $k$ is subtracted from $n+1$ except possibly at the $i_1$-th place.
Write $k = \sum_{i = 0}^{i_1}k_i p^{(i)}$ and assume that $n+1-k \triangleright k$.
Similarly to the first case, $k_i \le \lfloor(\ell\wedge p -1)/2\rfloor$ for $i < i_1$.
Then if $k_{i_1} \le c$ we must have that $k_{i_1} \le \lfloor c/2\rfloor$.
Otherwise $c < k_{i_1} \le a_2$ and $k_{i_1} \le \lfloor (p\wedge \ell +c)/2\rfloor$.
\end{proof}

%% file: small_tensor.tex
\section{Tensor with Small Idempotents}\label{sec:small_tensor}
Here we consider the composition $\bmu = (n, k)$ where $1 \le k \le \ell$.
This corresponds to tensoring the $n$-th Jones-Wenzl idempotent by the $k$-th where $k$ lies on, below, or just above the first multiple of $\ell$.
We will focus on the case where $\ell \mid n+1$.

\subsection{Eve Idempotent}
Through this sub-section, suppose $1\le k \le \ell-1$ in the algebra $\TL_\bmu$ for $\bmu = (n, k)$.
As such, $k$ is Eve.

\begin{proposition}
  If $\bmu = (n, k)$ with $1 \le k \le \ell -1$ and $\ell \mid n + 1$, then $\TL_\bmu$  is commutative and
  \begin{equation}
    \Lambda_\bmu = \left(\supp n + \{k, k-2, \ldots, -k\}\right) \cap \N_0.
  \end{equation}
  Further,
  \begin{equation}
    (\Lambda_\bmu)_0 =  \{n-k-2i\mid 0\leq i\leq \lfloor k/2\rfloor\}.
  \end{equation}
  In particular, $\TL_\bmu$ has $\lfloor \frac{k}{2}\rfloor+1$ irreducible modules.
\end{proposition}
\begin{proof}
  By performing a similar factorisation to that in \cref{sec:seam} we see that a basis diagram of any cell module is given by the composition of two diagrams $\mathbf{t} = (\mathbf{t_0 \otimes \id)\cdot \mathbf{t}}_1$ where $\mathbf{t}_0 : \underline{n} \to \underline{m_1}$ can be taken to be ${\rm d}_k^{m_1}$ and $\mathbf{t}_1 : \underline{m_1 + k} \to \underline{m}$ links no two of the last $k$ sites.
  It is clear then that $\Lambda_\bmu$ has the desired form.
  Now, since $\ell \mid n+1$, all elements of $\supp n$ are at least $2\ell$ apart and hence all the cell modules are one-dimensional.
  Hence $\TL_\bmu$ is commutative.

Next we compute the Gram form on these cell modules.
Since the modules are one dimensional, this is a simple coefficient.
Observe that the chosen basis for each module is spanned by a morphism of the form
\begin{equation}
  \vcenter{\hbox{
  \begin{tikzpicture}
    \draw[thick, fill=purple] (-.3,0.1) rectangle (.0,1.8);
    \fill (-.3,1.65) rectangle (0,1.8);
    \draw[thick, fill=purple] (-.3,0.1) rectangle (.0,-.6);
    \draw[thick] (0,.1) -- (0.6,.3) -- (0.6,1.6) -- (0,1.8) --cycle;
    \node at (.32,.9) {${\rm d}_n^{m_1}$};
    \draw[line width=2pt] (0,-.2) -- (0.6,-.2) arc (-90:90:0.35);
    \draw[line width=2pt] (0.6,1) -- (1.4, 1);
    \draw[line width=2pt] (0,-.4) -- (1.4, -.4);
    \draw[line width=2pt] (0,.95) -- (0,.95);
    \node at (1.1, 0.1) {\tiny$t$};
    \node at (1.0, 1.2) {\tiny$m_1-t$};
    \node at (0.7, -.57) {\tiny$k-t$};
  \end{tikzpicture}
}}
\end{equation}
and so the Gram constant on $\Delta_\bmu(m_1 + k -2t)$ can be calculated as the coefficient of the identity diagram in 
\begin{equation}
  \vcenter{\hbox{
  \begin{tikzpicture}
    \draw[thick] (-0.15,.1) -- (-0.75,.3) -- (-0.75,1.6) -- (-0.15,1.8) --cycle;
    \draw[thick] (0.15,.1) -- (0.75,.3) -- (0.75,1.6) -- (0.15,1.8) --cycle;
    \node at (.47,.9) {${\rm d}_n^{m_1}$};
    \node at (-.45,.9) {${\rm u}_n^{m_1}$};
    \draw[line width=2pt] (0.15,-.2) -- (0.75,-.2) arc (-90:90:0.35);
    \draw[line width=2pt] (-0.15,-.2) -- (-0.75,-.2);
    \draw[line width=2pt] (-.75, .5) arc (90:270:0.35);
    \draw[line width=2pt] (0.75,1) -- (1.55, 1);
    \draw[line width=2pt] (0.15,-.6) -- (1.55, -.6);
    \draw[line width=2pt] (-0.75,1) -- (-1.55, 1);
    \draw[line width=2pt] (-0.15,-.6) -- (-1.55, -.6);
    \node at (1.25, 0.1) {\tiny$t$};
    \node at (1.15, 1.2) {\tiny$m_1-t$};
    \node at (0.85, -.77) {\tiny$k-t$};
    \node at (-1.25, 0.1) {\tiny$t$};
    \node at (-1.15, 1.2) {\tiny$m_1-t$};
    \node at (-0.85, -.77) {\tiny$k-t$};
    \draw[thick, fill=purple] (-.15,.1) rectangle (.15,1.8);
    \fill (-.15,1.65) rectangle (.15,1.8);
    \draw[thick, fill=purple] (-.15,0) rectangle (.15,-.8);
  \end{tikzpicture}
  }} = \sum_{\substack{m' \in \supp n\\m' \le m_1}} \lambda_n^{m'}\vcenter{\hbox{
  \begin{tikzpicture}
    \draw[thick] (0,.3) -- (-0.6,.1) -- (-0.6,1.8) -- (0,1.6) --cycle;
    \draw[thick] (0,.3) -- (0.6,.1) -- (0.6,1.8) -- (0,1.6) --cycle;
    \fill (0,1.45) -- (-0.6,1.65) -- (-0.6,1.8) -- (0,1.6) --cycle;
    \fill (0,1.45) -- (0.6,1.65) -- (0.6,1.8) -- (0,1.6) --cycle;
    \node at (.32,.9) {$\overline{\rm u}_n^{m'}$};
    \node at (-.32,.9) {$\overline{\rm d}_n^{m'}$};
    \draw[thick] (-1.2,.3) -- (-0.6,.1) -- (-0.6,1.8) -- (-1.2,1.6) --cycle;
    \draw[thick] (1.2,.3) -- (0.6,.1) -- (0.6,1.8) -- (1.2,1.6) --cycle;
    \node at (.92,.9) {${\rm d}_n^{m_1}$};
    \node at (-.92,.9) {${\rm u}_n^{m_1}$};
    \draw[line width=2pt] (0.15,-.2) -- (1.2,-.2) arc (-90:90:0.35);
    \draw[line width=2pt] (-0.15,-.2) -- (-1.2,-.2);
    \draw[line width=2pt] (-1.2, .5) arc (90:270:0.35);
    \draw[line width=2pt] (1.2,1) -- (2.0, 1);
    \draw[line width=2pt] (0.15,-.6) -- (2.0, -.6);
    \draw[line width=2pt] (-1.2,1) -- (-2.0, 1);
    \draw[line width=2pt] (-0.15,-.6) -- (-2.0, -.6);
    \node at (1.7, 0.1) {\tiny$t$};
    \node at (1.6, 1.2) {\tiny$m_1-t$};
    \node at (1.3, -.77) {\tiny$k-t$};
    \node at (-1.7, 0.1) {\tiny$t$};
    \node at (-1.6, 1.2) {\tiny$m_1-t$};
    \node at (-1.3, -.77) {\tiny$k-t$};
    \draw[thick, fill=purple] (-.15,0) rectangle (.15,-.8);
  \end{tikzpicture}
}}
\end{equation}
Now, if $m' < m_1$, then $m' \le m_1 - 2\ell <m_1 + k - 2t$.  Hence the only term that contributes is
\begin{equation}\label{eq:inner_product_easy_small_diag}
  \lambda_n^{m_1}\vcenter{\hbox{
  \begin{tikzpicture}
    \draw[thick] (0,.3) -- (-0.6,.1) -- (-0.6,1.8) -- (0,1.6) --cycle;
    \draw[thick] (0,.3) -- (0.6,.1) -- (0.6,1.8) -- (0,1.6) --cycle;
    \fill (0,1.45) -- (-0.6,1.65) -- (-0.6,1.8) -- (0,1.6) --cycle;
    \fill (0,1.45) -- (0.6,1.65) -- (0.6,1.8) -- (0,1.6) --cycle;
    \node at (.32,.9) {$\overline{\rm u}_n^{m_1}$};
    \node at (-.32,.9) {$\overline{\rm d}_n^{m_1}$};
    \draw[thick] (-1.2,.3) -- (-0.6,.1) -- (-0.6,1.8) -- (-1.2,1.6) --cycle;
    \draw[thick] (1.2,.3) -- (0.6,.1) -- (0.6,1.8) -- (1.2,1.6) --cycle;
    \node at (.92,.9) {${\rm d}_n^{m_1}$};
    \node at (-.92,.9) {${\rm u}_n^{m_1}$};
    \draw[line width=2pt] (0.15,-.2) -- (1.2,-.2) arc (-90:90:0.35);
    \draw[line width=2pt] (-0.15,-.2) -- (-1.2,-.2);
    \draw[line width=2pt] (-1.2, .5) arc (90:270:0.35);
    \draw[line width=2pt] (1.2,1) -- (2.0, 1);
    \draw[line width=2pt] (0.15,-.6) -- (2.0, -.6);
    \draw[line width=2pt] (-1.2,1) -- (-2.0, 1);
    \draw[line width=2pt] (-0.15,-.6) -- (-2.0, -.6);
    \node at (1.7, 0.1) {\tiny$t$};
    \node at (1.6, 1.2) {\tiny$m_1-t$};
    \node at (1.3, -.77) {\tiny$k-t$};
    \node at (-1.7, 0.1) {\tiny$t$};
    \node at (-1.6, 1.2) {\tiny$m_1-t$};
    \node at (-1.3, -.77) {\tiny$k-t$};
    \draw[thick, fill=purple] (-.15,0) rectangle (.15,-.8);
  \end{tikzpicture}
}} = \left(\lambda_n^{m_1}\right)^{-1}\vcenter{\hbox{
  \begin{tikzpicture}
    \draw[line width=2pt] (0.15,-.2) arc (-90:90:0.35);
    \draw[line width=2pt] (-0.15, .5) arc (90:270:0.35);
    \draw[line width=2pt] (0.15,1) -- (0.95, 1);
    \draw[line width=2pt] (0.15,-.6) -- (0.95, -.6);
    \draw[line width=2pt] (-0.95,1) -- (-0.15, 1);
    \draw[line width=2pt] (-0.15,-.6) -- (-0.95, -.6);
    \node at (0.65, 0.1) {\tiny$t$};
    \node at (0.65, 1.2) {\tiny$m_1-t$};
    \node at (0.55, -.77) {\tiny$k-t$};
    \node at (-0.65, 0.1) {\tiny$t$};
    \node at (-0.65, 1.2) {\tiny$m_1-t$};
    \node at (-0.55, -.77) {\tiny$k-t$};
    \draw[thick, fill=purple] (-.15,0.1) rectangle (.15,1.6);
    \draw[thick, fill=purple] (-.15,0) rectangle (.15,-.8);
  \end{tikzpicture}
}}
\end{equation}
For any morphism $ f : \underline{m} \to \underline{m}$, the coefficient of the identity diagram is given by the full trace, $\tau^m(f \cdot \JW_m)/[m+1]$.
Hence, the identity coefficient of the morphism in \cref{eq:inner_product_easy_small} is given by
\begin{equation}\label{eq:inner_product_easy_small}
  \left(\lambda_n^{m_1}\right)^{-1} \frac{\Theta(m_1, k, m_1 + k - 2t)}{[m_1+k-2t+1]} = 
  \left(\lambda_n^{m_1}\right)^{-1} (-1)^{m_1+k-t}
 % \frac{[m_1 + k -t + 1]}{[m_1 + k -2t+1]}
  \frac{\gaussianquant{m_1+k-t+1}{t}}{\gaussianquant{m_1}{t}\gaussianquant{k}{t}}.
\end{equation}
Now, since $t \le k < \ell$, the denominator of \cref{eq:inner_product_easy_small} does not vanish.
Thus in order to check when it is zero, we require $(\lambda_n^{m_1})^{-1}$ to not be zero (i.e. $m_1 = n$) and $m_1 + k -t + 1 \triangleright t$.
Since $\ell \mid m_1 + 1$, this is equivalent to $k -t \ge t$.
Thus $t \in \{0,1, \ldots, \lfloor \frac{k}2\rfloor\}$ as desired.
\end{proof}

\begin{example}
  A critical example is when $k = \ell - 1 $.
  Here we present the data for $p = 2$ and $\ell = 6$.
  Note we are restricting to cases where $6 \mid n + 1$.
  All cell modules for $\TL_{(n,5)}$ are one dimensional (orange).  Grey dots are the cell indices of cell modules of $\TL_\mathbf{n+5}$ and circled dots indicate elements of $(\Lambda_{(n,5)})_0$.
  \begin{center}
    \includegraphics[width=0.8\textwidth]{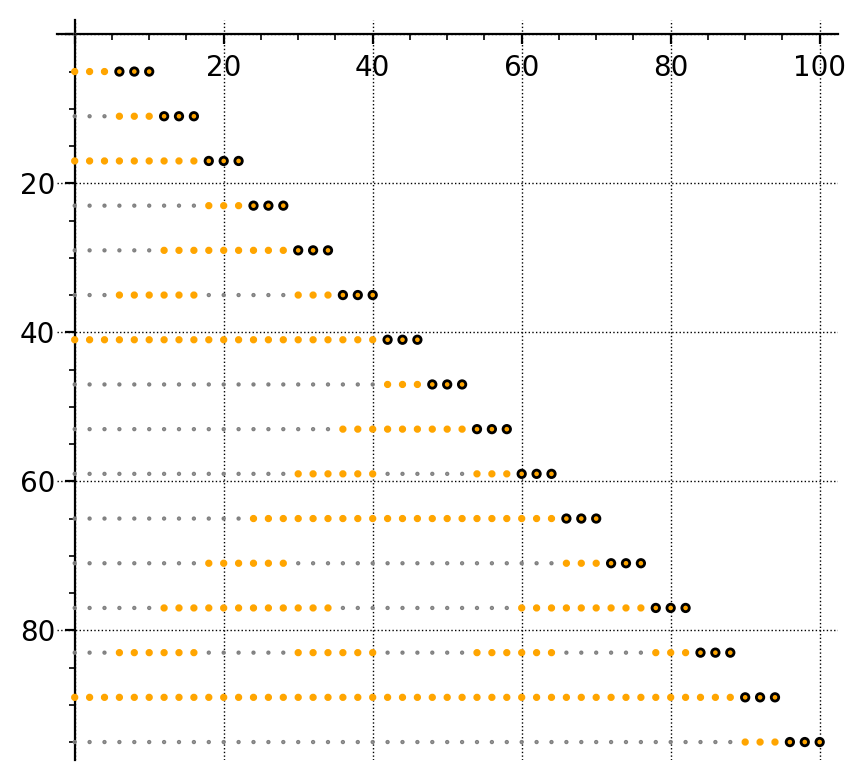}
  \end{center}
\end{example}

\subsection{First non-Eve Idempotent}
We now investigate the algebra $\TL_{(n,\ell)}$ where $\ell \mid n + 1$.
We are motivated by the self-similarity of the representation theory of $\TL_n$ under the $\ell$-dilation $n \mapsto (n+1)/\ell$.

\begin{proposition}
  Suppose that $\bmu = (n, \ell)$ with $\ell \mid n+1$.
  Write $n+1 = \pldigs{n_r, \ldots, n_1,0}$ and let $t_n$ be the largest such that $n_1 = n_2 = \cdots = n_{t_k} = p-1$.
  Set $T = \{1 \le t \le t_k\;:\; n_t = 1\}$ and $S = \{0 \le t < t_k \;:\; n_t > 1\}$.  Unless $p=2$, $T$ has at most one element.
  Then $\TL_\bmu$ has
  \begin{equation}
    \Lambda_\bmu = \left( \supp n + \left\{ \ell, \ell-2, \ldots, -\ell \right\}\right) \cap \N_0.
  \end{equation}
  Write
  \begin{align*}
    \Lambda_{2\rm x} &= \left(\supp n + \left(\{\ell-2, \ell-4, \ldots, -\ell+2\right\}\right)\cap N_0\\
      \Lambda_{2\rm y} &= \bigcup_{t \in T}\supp \mother[t]{n}\\
      \Lambda_1 &= \Big(\supp n + \{\ell, -\ell\}\Big) \setminus \Lambda_{2\rm y}.
  \end{align*}
  Then $\Lambda_1$ indexes cell modules of dimension 1, and $\Lambda_{2\rm x}$ and $\Lambda_{2 \rm y}$ index those of dimension 2.
  If
  \begin{align}
    \Lambda_{1\rm a} &= \{n + \ell\}\nonumber\\
    \Lambda_{1\rm b} &= \left\{\pldigs{n_r, \ldots, n_{t+1} - 1,0,\ldots, 0} - 1 \;:\; t \in S\right\}\nonumber\\
    \Lambda_{2\rm a} &= \left\{n+\ell-2t \;:\; 1\le t \le \lfloor \ell / 2\rfloor\right\}\nonumber\\
    \Lambda_{2\rm b} &= \left\{\pldigs{n_r, \ldots, n_{t+1},0,\ldots, 0} - 1 \;:\; t \in T\right\}
  \end{align}
  then $\Lambda_0 = \Lambda_{1\rm a} \cup\Lambda_{1\rm b} \cup\Lambda_{2\rm a} \cup\Lambda_{2\rm b}$ and $\Lambda_{1\rm a} \cup\Lambda_{1\rm b}$ index the irreducible modules of dimension 1 and $\Lambda_{2\rm a} \cup\Lambda_{2\rm b}$ those of dimension 2.
\end{proposition}
\begin{proof}
Note that $e^{(\ell)} = \JW_{\ell - 1} \otimes \id$ and
$\TL_{(\ell)}$ has two cell modules, indexed by $\{\ell, \ell-2\}$.
Hence we see that
\begin{equation}
  \Lambda_\bmu = \left( \supp n + \{ \ell, \ell-2, \ldots, -\ell \}\right) \cap \N_0
\end{equation}
Suppose that $t \in \{1,\ldots, \ell-1\}$.
Then if $m_1 \in \supp n$, the module $\Delta_\bmu(m_1 + \ell - 2t)$ is spanned by
\begin{align}
  x_1^{m_1+\ell-2t} &=
  \vcenter{\hbox{
  \begin{tikzpicture}
    \draw[thick, fill=purple] (-.3,0.1) rectangle (.0,1.8);
    \fill (-.3,1.65) rectangle (0,1.8);
    \draw[thick, fill=purple] (-.3,0.1) rectangle (.0,-.6);
    \fill (-.3,-0.05) rectangle (0,0.1);
    \draw[thick] (0,.1) -- (0.6,.3) -- (0.6,1.6) -- (0,1.8) --cycle;
    \node at (.32,.9) {${\rm d}_n^{m_1}$};
    \draw[line width=2pt] (0,-.2) -- (0.6,-.2) arc (-90:90:0.35);
    \draw[line width=2pt] (0.6,1) -- (1.4, 1);
    \draw[line width=2pt] (0,-.4) -- (1.4, -.4);
    \draw[line width=2pt] (0,.95) -- (0,.95);
    \node at (1.1, 0.1) {\tiny$t$};
    \node at (1.0, 1.2) {\tiny$m_1-t$};
    \node at (0.7, -.57) {\tiny$\ell-t$};
  \end{tikzpicture}
}} =
\vcenter{\hbox{
  \begin{tikzpicture}
    \draw[thick, fill=purple] (-.3,0.1) rectangle (.0,1.8);
    \fill (-.3,1.65) rectangle (0,1.8);
    \draw[thick, fill=purple] (-.3,0.1) rectangle (.0,-.45);
    \draw[thick] (0,.1) -- (0.6,.3) -- (0.6,1.6) -- (0,1.8) --cycle;
    \node at (.32,.9) {${\rm d}_n^{m_1}$};
    \draw[line width=2pt] (0,-.05) -- (0.6,-.05) arc (-90:90:0.3);
    \draw[line width=2pt] (0.6,1) -- (1.4, 1);
    \draw[line width=2pt] (0,-.2) -- (1.4, -.2);
    \draw[line width=2pt] (0,.95) -- (0,.95);
    \draw[very thick] (-.3,-.55) -- (1.4,-.55);
    \node at (1.15, 0.2) {\tiny$t$};
    \node at (1.1, 1.2) {\tiny$m_1-t$};
    \node at (0.9, -.37) {\tiny$\ell-t-1$};
  \end{tikzpicture}
}}
\\
x_2^{m_1+\ell-2t} &=
\vcenter{\hbox{
  \begin{tikzpicture}
    \draw[thick, fill=purple] (-.3,0.1) rectangle (.0,1.8);
    \fill (-.3,1.65) rectangle (0,1.8);
    \draw[thick, fill=purple] (-.3,0.1) rectangle (.0,-.6);
    \fill (-.3,-0.05) rectangle (0,0.1);
    \draw[thick] (0,.1) -- (0.6,.3) -- (0.6,1.6) -- (0,1.8) --cycle;
    \node at (.32,.9) {${\rm d}_n^{m_1}$};
    \draw[line width=2pt] (0,-.05) -- (0.6,-.05) arc (-90:90:0.3);
    \draw[line width=2pt] (0.6,1) -- (1.4, 1);
    \draw[line width=2pt] (0,-.2) -- (1.4, -.2);
    \draw[line width=2pt] (0,.95) -- (0,.95);
    \draw[very thick] (0,-.5) -- (0.1,-.5) arc(-90:90:0.06) -- (0,-.38);
    \node at (1.25, 0.2) {\tiny$t-1$};
    \node at (1.2, 1.2) {\tiny$m_1-t + 1$};
    \node at (0.9, -.37) {\tiny$\ell-t-1$};
  \end{tikzpicture}
}} =
\vcenter{\hbox{
  \begin{tikzpicture}
    \draw[thick, fill=purple] (-.3,0.1) rectangle (.0,1.8);
    \fill (-.3,1.65) rectangle (0,1.8);
    \draw[thick, fill=purple] (-.3,0.1) rectangle (.0,-.45);
    \draw[thick] (0,.1) -- (0.6,.3) -- (0.6,1.6) -- (0,1.8) --cycle;
    \node at (.32,.9) {${\rm d}_n^{m_1}$};
    \draw[line width=2pt] (0,-.05) -- (0.6,-.05) arc (-90:90:0.3);
    \draw[line width=2pt] (0.6,1) -- (1.4, 1);
    \draw[line width=2pt] (0,-.2) -- (1.4, -.2);
    \draw[line width=2pt] (0,.95) -- (0,.95);
    \draw[very thick] (-.3,-.55) -- (0.1,-.55) arc(-90:90:0.08) -- (0,-.39);
    \node at (1.25, 0.2) {\tiny$t-1$};
    \node at (1.2, 1.2) {\tiny$m_1-t + 1$};
    \node at (0.9, -.37) {\tiny$\ell-t-1$};
  \end{tikzpicture}
}}
\end{align}
It is clear that $\langle x_1^m, x_1^m \rangle = G_{(n,\ell-1)}^{m-1}$ which vanishes unless $t \in \{ 1, \ldots, \lfloor \frac{\ell}2\rfloor \}$.
On the other hand, $\langle x_2^m, x_2^m \rangle = [\ell]/[\ell-1] G_{(n, \ell-2)}^{m - 2}$ which always vanishes.
To compute cross terms, we must use the following well known form of the Jones-Wenzl idempotents (see~\cite[Theorem 4.5]{frenkel_khovanov_1997}).
\begin{equation}\label{eq:single_clasp_jones_wenzl}
  \vcenter{\hbox{
    \begin{tikzpicture}
      \draw[very thick, fill = purple] (-.45,0) rectangle (0.45,2);
      \node at (0,1) {$\JW_n$};
      \draw[line width=2pt] (-.8,1) -- (-.45,1);
      \draw[line width=2pt] (.8,1) -- (.45,1);
    \end{tikzpicture}
  }} = 
  \vcenter{\hbox{
    \begin{tikzpicture}
      \draw[very thick, fill = purple] (-.65,0.2) rectangle (0.65,2);
      \node at (0,1.1) {$\JW_{n-1}$};
      \draw[line width=2pt] (-1,1.1) -- (-.65,1.1);
      \draw[line width=2pt] (1,1.1) -- (.65,1.1);
      \draw[very thick] (-1,0) -- (1,0);
    \end{tikzpicture}
  }} + \sum_{i=1}^{n-1} \frac{(-1)^i [i]}{[n]}
  \vcenter{\hbox{
    \begin{tikzpicture}
      \draw[very thick, fill = purple] (-.65,0.2) rectangle (0.65,2);
      \node at (0,1.1) {$\JW_{n-1}$};
      \draw[line width=2pt] (1,1.1) -- (.65,1.1);
      \draw[very thick] (1,0) -- (-0.65,0);
      \draw[very thick] (-.65,0.4) arc (90:270:0.2);
      \draw[very thick] (-1.6,1.1) arc (-90:90:0.2);
      \draw[line width=2pt] (-.65, 1.7) edge[out=180,in=0] (-1.4,1.9);
      \draw[line width=2pt] (-1.4,1.9) -- (-1.6,1.9);
      \draw[line width=2pt] (-.65, 0.9) edge[out=180,in=0] (-1.4,0.7);
      \draw[line width=2pt] (-1.4,0.7) -- (-1.6,0.7);
      \node at (-1.1, 2.1) {\tiny$i-1$};
      \node at (-1.2, 0.5) {\tiny$n-1-i$};
    \end{tikzpicture}
  }}
\end{equation}
This can be verified by noting that the morphism in \cref{eq:single_clasp_jones_wenzl} is killed by the right action of all cups and thus must be $\JW_{n}$ by uniqueness under that property~\cite[Lemma 1.1]{martin_spencer_2021}.

Then the inner product $\langle x_1^{m}, x_2^{m}\rangle$ is the identity coefficient in
\begin{equation}\label{eq:twisty_inner_product}
  \vcenter{\hbox{
    \begin{tikzpicture}
      \draw[thick, fill=purple] (-.3,0.1) rectangle (.0,1.8);
      \fill (-.3,1.65) rectangle (0,1.8);
      \draw[thick, fill=purple] (-.3,0.1) rectangle (.0,-.45);
      \draw[thick] (0,.1) -- (0.6,.3) -- (0.6,1.6) -- (0,1.8) --cycle;
      \draw[thick] (-.3,.1) -- (-.9,.3) -- (-.9,1.6) -- (-.3,1.8) --cycle;
      \node at (.32,.9) {${\rm d}_n^{m_1}$};
      \node at (-.62,.9) {${\rm u}_n^{m_1}$};
      \draw[line width=2pt] (0,-.05) -- (0.6,-.05) arc (-90:90:0.3);
      \draw[line width=2pt] (-.3,-.05) -- (-.9,-.05);
      \draw[line width=2pt] (-.9, 0.55) arc (90:270:0.3);
      \draw[line width=2pt] (0.6,1) -- (1.4, 1);
      \draw[line width=2pt] (-0.9,1) -- (-1.7, 1);
      \draw[line width=2pt] (0,-.2) -- (1.4, -.2);
      \draw[line width=2pt] (-.3,-.2) -- (-1.7, -.2);
      \draw[line width=2pt] (0,.95) -- (0,.95);
      \draw[very thick] (-1.7,-.55) -- (0.1,-.55) arc(-90:90:0.08) -- (0,-.39);
      \node at (1.25, 0.2) {\tiny$t-1$};
      \node at (-1.50, 0.2) {\tiny$t$};
      \node at (1.2, 1.2) {\tiny$m_1-t + 1$};
      \node at (-1.5, 1.2) {\tiny$m_1-t$};
      \node at (0.9, -.37) {\tiny$\ell-t-1$};
      \node at (-1.2, -.37) {\tiny$\ell-t-1$};
    \end{tikzpicture}
  }} =
  \left(\lambda_{n}^{m_1}\right)^{-1}
  \vcenter{\hbox{
    \begin{tikzpicture}
      \draw[thick, fill=purple] (-.15,0.2) rectangle (.15,1.8);
      \draw[thick, fill=purple] (-.15,0.1) rectangle (.15,-.45);
      \draw[line width=2pt] (.15,-.05) arc (-90:90:0.3);
      \draw[line width=2pt] (-.15, 0.55) arc (90:270:0.3);
      \draw[line width=2pt] (0.15,1) -- (1.4, 1);
      \draw[line width=2pt] (-0.15,1) -- (-1.4, 1);
      \draw[line width=2pt] (0.15,-.2) -- (1.4, -.2);
      \draw[line width=2pt] (-.15,-.2) -- (-1.4, -.2);
      \draw[very thick] (-1.4,-.55) -- (0.25,-.55) arc(-90:90:0.08) -- (0.15,-.39);
      \node at (0.6, 0.6) {\tiny$t-1$};
      \node at (-.45, 0.6) {\tiny$t$};
      \node at (0.8, 1.2) {\tiny$m_1-t + 1$};
      \node at (-.8, 1.2) {\tiny$m_1-t$};
      \node at (1.0, -0) {\tiny$\ell-t-1$};
      \node at (-1.05, -0) {\tiny$\ell-t-1$};
    \end{tikzpicture}
  }}
\end{equation}
This identify follows by the same argument in \cref{eq:inner_product_easy_small_diag}.

Applying \cref{eq:single_clasp_jones_wenzl} to the lower Jones-Wenzl projector, and noting that the only remaining term after expanding is the one for which $i = t$, gives the coefficient of the identity as
\begin{equation}
  \left(\lambda_{n}^{m_1}\right)^{-1}\frac{(-1)^t[t]\Theta(m_1, \ell-2, m_1+\ell-2t)}{[\ell-1][m_1+\ell-2t+1]}
  = \left(\lambda_{n}^{m_1}\right)^{-1}(-1)^t \frac{[t]}{[\ell-1]} \gaussianquant{m_1+\ell-t}{t-1}\Big/\gaussianquant{m_1}{t-1}\gaussianquant{\ell-2}{t-1}.
\end{equation}
Again, this vanishes unless $t \in \{0, \lfloor \frac \ell 2\rfloor\}$ and $m_1 = n$.

The final case to consider is when $t = 0$ or $t= \ell$.
Here one needs to know if two elements of $\supp n$ differ by $2\ell$.  However, the analysis is identical to that in \cref{ssec:two_part_seams} and so is not repeated here.
\end{proof}

\begin{remark}
  In the same way that the theta network encodes the inner product on the two part Eve cell modules:
  \begin{equation}
    \left(
    \vcenter{\hbox{
      \begin{tikzpicture}
        \draw[thick, fill=purple] (-.25,0.1) rectangle (.25,1.2);
        \draw[thick, fill=purple] (-.25,-0.1) rectangle (.25,-1.2);
        \draw[line width=2pt] (.25,-.35) arc (-90:90:0.35);
        \draw[line width=2pt] (-.25, 0.35) arc (90:270:0.35);
        \draw[line width=2pt] (0.25,0.8) -- (1.4, 0.8);
        \draw[line width=2pt] (-0.25,0.8) -- (-1.4, 0.8);
        \draw[line width=2pt] (0.25,-0.8) -- (1.4, -0.8);
        \draw[line width=2pt] (-0.25,-0.8) -- (-1.4, -0.8);
        \node at (0.8, 0) {\tiny$i$};
        \node at (-.8, 0) {\tiny$i$};
        \node at (0.8, 1.0) {\tiny$r-i$};
        \node at (-.8, 1.0) {\tiny$r-i$};
        \node at (0.8, -1.0) {\tiny$s-i$};
        \node at (-0.8, -1.0) {\tiny$s-i$};
      \end{tikzpicture}
    }}
    \right)_\id
    = \frac{1}{[r + s - 2i + 1]}
    \vcenter{\hbox{
      \begin{tikzpicture}[scale=0.7]
        \draw[thick, fill=purple] (-.25,0.1) rectangle (.25,1.2);
        \draw[thick, fill=purple] (-.25,-0.1) rectangle (.25,-1.2);
        \draw[thick, fill=purple] (-.25,-1.4) rectangle (.25,-2.5);
        \draw[line width=2pt] (.25,-.35) arc (-90:90:0.35);
        \draw[line width=2pt] (-.25, 0.35) arc (90:270:0.35);

        \draw[line width=2pt] (.25,-1.65) arc (-90:90:0.35);
        \draw[line width=2pt] (-.25, -0.95) arc (90:270:0.35);

        \draw[line width=2pt] (.25,0.95-3.1) arc (-90:90:1.55);
        \draw[line width=2pt] (-.25, 0.95) arc (90:270:1.55);

        \node at (0.8, 0) {\tiny$i$};
        \node at (-.8, 0) {\tiny$i$};
        \node at (1.3, -.65) {\tiny$r-i$};
        \node at (-1.3, -.65) {\tiny$r-i$};
        \node at (1.1, -1.3) {\tiny$s-i$};
        \node at (-1.1, -1.3) {\tiny$s-i$};
      \end{tikzpicture}
    }} = \frac{\Theta(r, s, r + s - 2i)}{[r+s-2i+1]}
  \end{equation}
  \cref{eq:twisty_inner_product} can be considered as a sort of ``singly twisted'' theta network:
  \begin{equation}
    \left(
    \vcenter{\hbox{
      \begin{tikzpicture}
        \draw[thick, fill=purple] (-.25,0.1) rectangle (.25,1.2);
        \draw[thick, fill=purple] (-.25,-0.1) rectangle (.25,-1.2);
        \draw[line width=2pt] (.25,-.35) arc (-90:90:0.35);
        \draw[line width=2pt] (-.25, 0.35) arc (90:270:0.35);
        \draw[line width=2pt] (0.25,0.8) -- (1.4, 0.8);
        \draw[line width=2pt] (-0.25,0.8) -- (-1.4, 0.8);
        \draw[line width=2pt] (0.25,-0.8) -- (1.4, -0.8);
        \draw[line width=2pt] (-0.25,-0.8) -- (-1.4, -0.8);
        \node at (0.9, 0) {\tiny$i-1$};
        \node at (-.8, 0) {\tiny$i$};
        \node at (0.9, 1.0) {\tiny$r-i+1$};
        \node at (-.8, 1.0) {\tiny$r-i$};
        \node at (0.8, -0.65) {\tiny$s-i$};
        \node at (-0.8, -0.65) {\tiny$s-i$};
        \draw[very thick] (-1.4,-1.3) -- (0.25,-1.3) arc (-90:90:0.1);
      \end{tikzpicture}
    }}
    \right)_\id
    = \frac{1}{[r + s - 2i + 2]}
    \vcenter{\hbox{
      \begin{tikzpicture}[scale=0.7]
        \draw[thick, fill=purple] (-.25,0.1) rectangle (.25,1.2);
        \draw[thick, fill=purple] (-.25,-0.1) rectangle (.25,-1.2);
        \draw[thick, fill=purple] (-.25,-1.8) rectangle (.25,-2.9);
        \draw[line width=2pt] (.25,-.35) arc (-90:90:0.35);
        \draw[line width=2pt] (-.25, 0.35) arc (90:270:0.35);

        \draw[line width=2pt] (.25,-2.35) arc (-90:90:0.85);
        \draw[line width=2pt] (-.25, -0.65) arc (90:270:0.85);

        \draw[line width=2pt] (.25,0.45-3.1) arc (-90:90:1.8);
        \draw[line width=2pt] (-.25, 0.95) arc (90:270:1.8);

        \node at (1.0, 0) {\tiny$i-1$};
        \node at (-.8, 0) {\tiny$i$};
        \node at (1.5, 1.0) {\tiny$r-i+1$};
        \node at (-1.5, 1.0) {\tiny$r-i$};
        \node at (1.3, -0.8) {\tiny$s-i$};
        \node at (-1.3, -0.8) {\tiny$s-i$};

        \draw[very thick] (.25, -1.0) arc (90:-90:0.25) -- (-.25, -1.5) arc (90:270:0.25);
      \end{tikzpicture}
    }}
  \end{equation}
\end{remark}

%% file: general-composition-cells.tex
\section{Cell Data for Arbitrary Compositions}\label{sec:cell_arbitrary}
Suppose now that $\bmu = (\mu_1, \ldots, \mu_r) \vdash n$ is an arbitrary composition.
We will now generalise the ideas in the previous sections to determine cell data for $\TL_\bmu$.

\begin{definition}
  For a given $\bmu$, let the set of \emph{associated compositions}, $\ass\bmu$ be 
  \begin{equation*}
    \ass\bmu = \{(a_1, \ldots,a_r) : a_i \in \supp\mu_i\},
  \end{equation*}
  so that $|\ass \bmu| = 2^{\sum_i \generation{\mu_i}}$.
\end{definition}

The idea for constructing a cell basis can be summed up by the following diagram.
Here, to fit on the page we have written the morphism from ``bottom to top'' instead of ``left to right''.
\input{ideas.tex}
Notice we factor the element of the cell module into an idempotent, some elements of cell modules for $\TL_{(n)}$ and one further diagram.

Recall the definition of $E_\mathbf{a}$ from \cref{sec:valenced_cell_data}.
\begin{proposition}
  If $\bmu \vdash n$, then $\TL_\bmu$ is cellular with cell modules described by
  \begin{equation}\label{eq:subsetsum}
    \Lambda_\bmu  = \bigcup_{\mathbf{a}\in \ass\bmu} E_\mathbf{a}
  \end{equation}
  and
  \begin{equation}\label{eq:dim_general_cell}
    \dim\Delta_\bmu(m) = \sum_{\mathbf{a}\in\ass\bmu} C^\mathbf{a}_m.
  \end{equation}
\end{proposition}
\begin{proof}
  If effect, we generalise the factorisation of \cref{sec:seam}.
  Let $\mathbf{t}$ be a diagram from $\underline{n} \to \underline{m}$ such that $e^\bmu\cdot\mathbf{t}$ is not zero in $\Delta(m)$.
  Let $B_i = \{\mu_1 + \cdots + \mu_{i-1}+1,\ldots,\mu_1 + \cdots + \mu_i\}$ be the $i$-th ``bucket'' of source sites on $\mathbf{t}$.
  Suppose that, in $\mathbf{t}$, $a_i$ of the source sites in bucket $B_i$ are not connected to other source sites in bucket $B_i$ and let $a = \sum_{i} a_i$.
  Then $\mathbf{t}$ uniquely factorises into
  \begin{equation}
    \mathbf{t} = (\mathbf{t}_0^1 \otimes\cdots\mathbf{t}_0^r) \cdot \mathbf{t}_1
  \end{equation}
  where $\mathbf{t}_0^i$ is a monic diagram $\underline{\mu_i}\to\underline{a_i}$ and $\mathbf{t}_1$ is a monic diagram $\underline{a}\to\underline{m}$ such that no two  source sites in the buckets described by the $a_i$ are connected.
  Let this describe the associated tuple to a diagram, $\ass \mathbf{t} = (a_1,\ldots, a_r)$.

  Order the $\mathbf{t}$ lexicographically by their associated tuples.
  Now suppose that $e^\bmu\cdot \mathbf{t}$ is not in the linear span of $e^\bmu\cdot\mathbf{t}'$ for smaller $\mathbf{t}'$.
  If $e^{\mu_i}\cdot \mathbf{t}_0^i$ has through degree less than $i$ then $e^\bmu\cdot \mathbf{t}$ lies in the linear span of $e^\bmu\cdot\mathbf{t}'$ where the $\mathbf{t}' \prec \mathbf{t}$.
  Indeed, by expanding at bucket $i$ we can express $e^{\mu_i}\cdot\mathbf{t}_0^i$ as a linear combination of $e^{\mu_i}\cdot \mathbf{t}'^i_0$ where $\mathbf{t}'^i_0$ have through degree in $\supp \mu_i$ that are less than $a_i$.
  Replacing $\mathbf{t}_0^i$ by such $\mathbf{t}'^i_0$ results in diagrams $\mathbf{t}'$ such that $a_i$ is smaller and no $a_j$ increase for $j \neq i$.
  Such a tuple is lexicographically smaller than $(a_i)$, showing the claim.

  Thus for this to be a new, linearly independent vector, $\mathbf{t}_0^i$ is a monic diagram from $\underline{\mu_i} \to \underline{a_i}$ and $a_i\in\supp{\mu_i}$.

  Since there is a unique morphism from $\underline{\mu_i} \to \underline{a_i}$ for each $a_i \in \supp{\mu_i}$, we can calculate the total number of morphisms by summing the appropriate $C^\mathbf{a}_m$.
\end{proof}
\begin{remark}
  Note that $C_m^\mathbf{a} = 0$ when $m \not\in E_\mathbf{a}$ so we could have written the sum in \cref{eq:dim_general_cell} over all tuples.
  However, the formulation of summing over associated tuples makes the factorisation clearer.
\end{remark}

\begin{remark}
  Unfortunately, it is not likely that much more will be able to be said about general compositions without more powerful machinery.
  The recursive nature of the calculations of $C^\mathbf{a}_m$ and the combinatorial nature of $\ass\bmu$ means that the analysis given for select compositions in this paper is unlikely to extend easily to general compositions (though the ideas may well be developed into something more useful).

  This factorisation, for example, is particularly useful in calculating Gram matrices for certain compositions that are either Eve or almost Eve (first generation idempotents with large cell indices).
  However, it does not yield much in the general case.
  For example, it becomes important to know the value of ${\rm u}_n^{m'} e^{(n)}{\rm d}_n^{m'}$ {\it exactly} and not just modulo lower morphisms, if one wants to calculate inner products on a cell module of a composition including $n$.
\end{remark}

%% file: ideas.tex
\begin{center}
  \begin{tikzpicture}[scale=0.8]
    \draw[very thick, fill=purple] (0,0.5) rectangle (2.4,1);
    \fill (0,0.5) rectangle (0.1,1);
    \draw[line width=2pt] (1.2, 0) -- (1.2,0.5);
    \draw[very thick, fill=purple] (2.6,0.5) rectangle (7.4,1);
    \fill (2.6,0.5) rectangle (2.7,1);
    \draw[line width=2pt] (5, 0) -- (5,0.5);
    \draw[very thick, fill=purple] (7.6,0.5) rectangle (8.4,1);
    \fill (7.6,0.5) rectangle (7.7,1);
    \draw[line width=2pt] (8, 0) -- (8,0.5);
    \draw[very thick, fill=purple] (8.6,0.5) rectangle (10.4,1);
    \fill (8.6,0.5) rectangle (8.7,1);
    \draw[line width=2pt] (9.5, 0) -- (9.5,0.5);
    \draw[very thick, fill=purple] (10.6,0.5) rectangle (13.4,1);
    \fill (10.6,0.5) rectangle (10.7,1);
    \draw[line width=2pt] (12, 0) -- (12,0.5);

    \draw (0,1) -- (2.4,1) -- (2.1,1.6) -- (0.3,1.6) -- cycle;
    \draw (2.6,1) -- (7.4,1) -- (7.1,1.6) -- (2.9,1.6) -- cycle;
    \draw (7.6,1) -- (8.4,1) -- (8.1,1.6) -- (7.9,1.6) -- cycle;
    \draw (8.6,1) -- (10.4,1) -- (10.1,1.6) -- (8.9,1.6) -- cycle;
    \draw (10.6,1) -- (13.4,1) -- (13.1,1.6) -- (10.9,1.6) -- cycle;

    \draw [decorate,decoration={brace,amplitude=5pt},xshift=-1pt,yshift=0pt] (-0.6,0.4) -- (-0.6,1.0) node [black,midway,xshift=-10pt] {\footnotesize $e^\bmu$};

    \draw [decorate,decoration={brace,amplitude=5pt},xshift=-1pt,yshift=0pt] (-0.6,1) -- (-0.6,2.4) node [black,midway,xshift=-17pt] {\footnotesize $\Delta_\bmu(m)$};

    \node at (1.2,-.3) {\tiny$\mu_1$};
    \node at (5,-.3) {\tiny$\mu_2$};
    \node at (8,-.3) {\tiny$\mu_3$};
    \node at (9.5,-.3) {\tiny$\mu_4$};
    \node at (12,-.3) {\tiny$\mu_5$};

    \begin{scope}[shift={(0,1)}]
      \draw[very thick] (0.4,0)--(0.4,.6);
      \draw[very thick] (0.5,0)--(0.5,.6);
      \draw[very thick] (1.4,0) arc (0:180:.4);
      \draw[very thick] (1.3,0) arc (0:180:.3);
      \draw[very thick] (1.2,0) arc (0:180:.2);
      \draw[very thick] (1.1,0) arc (0:180:.1);
      \draw[very thick] (1.5,0)--(1.5,.6);
      \draw[very thick] (1.6,0)--(1.6,.6);
      \draw[very thick] (1.7,0)--(1.7,.6);
      \draw[very thick] (1.8,0)--(1.8,.6);
      \draw[very thick] (1.9,0)--(1.9,.6);
      \draw[very thick] (2.0,0)--(2.0,.6);
    \end{scope}

    \begin{scope}[shift={(2.6,1)}]
      \foreach \i in {0,...,10} {
        \draw[very thick] (0.4 + \i/10,0)--(0.4 + \i/10,.6);
      }
      \draw[very thick] (1.9,0) arc (0:180:.2);
      \draw[very thick] (1.8,0) arc (0:180:.1);
      \foreach \i in {16,...,40} {
        \draw[very thick] (0.4 + \i/10,0)--(0.4 + \i/10,.6);
      }
    \end{scope}

    \begin{scope}[shift={(7.6,1)}]
      \foreach \i in {0,...,0} {
        \draw[very thick] (0.4 + \i/10,0)--(0.4 + \i/10,.6);
      }
    \end{scope}

    \begin{scope}[shift={(8.6,1)}]
      \foreach \i in {0,...,5} {
        \draw[very thick] (0.4 + \i/10,0)--(0.4 + \i/10,.6);
      }
      \draw[very thick] (1.6,0) arc (0:180:.3);
      \draw[very thick] (1.5,0) arc (0:180:.2);
      \draw[very thick] (1.4,0) arc (0:180:.1);
    \end{scope}

    \begin{scope}[shift={(10.6,1)}]
      \foreach \i in {0,...,5} {
        \draw[very thick] (0.8+\i/10,0) arc (0:180:\i/10);
      }
      \foreach \i in {0,...,3} {
        \draw[very thick] (1.7+\i/10,0) arc (0:180:\i/10);
      }
      \foreach \i in {0,...,2} {
        \draw[very thick] (2.3+\i/10,0) arc (0:180:\i/10);
      }
    \end{scope}

    \begin{scope}[shift={(0,1.6)}]
      \draw (0.3,0) -- (13.1,0) -- (12.8,0.6) -- (0.6,0.6) -- cycle;
      \draw (0.4,-0.05) rectangle (2.0,0.05);
      \draw (3.0,-0.05) rectangle (7.0,0.05);
      \draw (7.95,-0.05) rectangle (8.05,0.05);
      \draw (9.,-0.05) rectangle (9.5,0.05);
    \end{scope}
  \end{tikzpicture}
\end{center}

%% file: further.tex
\section{Further work}\label{sec:further}
\subsection{Other Idempotents}
In general, the idempotent $e^{(n)}$ is not the only idempotent describing the projective cover of the trivial module.
Though in any given decomposition of $\TL_n$, there is exactly one factor isomorphic to $P_n(n)$, there are multiple decompositions possible.

Indeed, as a first example, consider the idempotent $\bar e^{(n)}$ which is simply the vertical flip of the morphism $e^{(n)}$.
It is clear that $\TL_n \cdot \bar e^{(n)} \simeq \TL_n \cdot e^{(n)}$ by the isomorphism sending $x \cdot e^{(n)} \mapsto \bar x \cdot e^{(n)}$.
Equivalently, there is a trivial isomorphism of cell data.

This stems from the fact that the ``ladder construction'' in \cref{def:ladder} is lopsided.
A more general construction would allow for the new strands in \label{eq:ladder_step} to be placed either above or below the previous morphism.
It only takes a minor alteration of the argument in \cite{martin_spencer_2021} to show that the resulting morphisms each describe $P_n(n)$.

The root cause of this is the tower of algebras
\begin{equation}\label{eq:tl_tower}
  k = \TL_0 \subset \TL_1 \subset \TL_2\subset \cdots\subset \TL_n \subset \TL_{n+1}\subset \cdots.
\end{equation}
Traditionally, $\TL_n\subset \TL_{n+1}$ by ``adding a strand below'', but one might as well add a strand above.
By using these different embeddings throughout, one determines different, equivalent idempotents.

Since these are equivalent idempotents, the results should be the same.
For example, \cref{prop:main_two_part} should describe $\TL_{(1,k)}$ as well as $\TL_{(k,1)}$.
However, evaluating the traces on the ``wrong side'' of $e^{(k)}$ is difficult, and it is likely to only get worse with non-lopsided idempotents.

Finding a method of calculation which works through all possible \cref{eq:tl_tower} may provide some insight into dealing with non-lopsided ladder constructions in general.

\begin{question}
  Are there other towers of the form \cref{eq:tl_tower} that do not consist of adding strands above and below in some order?
\end{question}

This question can be seen as a dual to the construction of Goodman and Wenzl in~\cite{goodman_wenzl_1993} where the authors construct idempotents describing projective covers of all the simple modules in characteristic $(\ell, 0)$ by considering paths that end at $(n,m)$.
In some sense, our observation implies that the projective covers should be indexed by pairs of paths.

\subsection{Other algebras}
The Temperley-Lieb algebras are only one family out of many families of cellular algebras.
Indeed, as endomorphisms of tilting modules for quantum groups as described in~\cite{andersen_stroppel_tubbenhauer_2018}, Temperley-Lieb algebras fall into the greater study of tilting modules for algebraic groups.

This has recently been expanded to more general ``standard categories with duals''~\cite{bellamy_thiel_2021} which allows one to assign cellular structures to endomorphism algebras over categories that are not necessarily highest-weight categories.

If the endomorphism algebras sit in a tower similar to \cref{eq:tl_tower}, compatible with a monoidal category structure, it may be possible to construct similar endomorphisms to $e^\bmu$.
If so, the factorisation in \cref{sec:cell_arbitrary} may be effective at studying this problem.
In particular, $(\TL_n)_{n\in\N_0}$ is in Schur-Weyl duality to $U_q(\mathfrak{sl}_2)$.
The objects in duality to $U_q(\mathfrak{sl}_n)$ for higher $n$ are known as {\it webs} or {\it spiders} and also admit cellular, diagrammatic behaviour.
It is hoped that they will admit a similar analysis as is indicated by the success in describing ``clasps'' (higher order Jones-Wenzl elements) in~\cite{elias_2015}.

Additionally, this question has application to Soergel bimodule theory~\cite{elias_williamson_2016}.
Here we again are interested in the endomorphism spaces of objects, this time Soergel bimodules.
Through the diagrammatic definition of the category similar techniques may allow one to study the breakdown of these objects (in particular the Bott-Samelson objects) into their indecomposable summands -- a key question in the area.